\documentclass[12pt]{article}
\usepackage[english]{babel}
\usepackage{latexsym,dsfont}
\usepackage{fullpage}
\usepackage[T1]{fontenc}
\usepackage{amsmath}
\usepackage{amsfonts}
\usepackage{amssymb}
\usepackage{mathrsfs}
\usepackage{amsthm}
\usepackage{latexsym}
\usepackage{array}
\usepackage{mathrsfs}
\usepackage{amsxtra}
\usepackage{amscd}
\usepackage{mathtools}
\usepackage{verbatim}
\usepackage{stmaryrd}
\usepackage{enumerate}
\usepackage{frcursive}
\usepackage{hyperref}
\hypersetup{
    colorlinks=true,
    linkcolor=blue,
	linktocpage=true
}
\usepackage{xcolor}

\numberwithin{equation}{section}

\theoremstyle{plain}
\newtheorem{Theorem}{Theorem}[section]
\newtheorem{Definition}[Theorem]{Definition}
\newtheorem{Proposition}[Theorem]{Proposition}
\newtheorem{Lemma}[Theorem]{Lemma}

\newtheorem{Corollary}[Theorem]{Corollary}
\newtheorem{Remark}[Theorem]{Remark}

\newtheorem{innerAssumption}{Assumption}
\newenvironment{Assumption}[1]
{\renewcommand\theinnerAssumption{#1}\innerAssumption}
{\endinnerAssumption}

\definecolor{red}{rgb}{1.0,0.0,0.0}

\definecolor{blu}{rgb}{0.0,0.0,1.0}

\definecolor{gre}{rgb}{0.03,0.50,0.03}

\definecolor{amethyst}{rgb}{0.6, 0.4, 0.8}

\definecolor{blue-violet}{rgb}{0.54, 0.17, 0.89}

\definecolor{darkviolet}{rgb}{0.58, 0.0, 0.83}

\def \trans{^{\scriptscriptstyle{\intercal }}}

\def \E{\mathbb{E}}
\def \F{\mathbb{F}}
\def \N{\mathbb{N}}
\def \P{\mathbb{P}}
\def \R{\mathbb{R}}

\def \Ac{{\cal A}}
\def \Bc{{\cal B}}

\def \Fc{{\cal F}}
\def \Gc{{\cal G}}

\def \Mc{{\cal M}}
\def \Nc{{\cal N}}
\def \Pc{{\cal P}}

\def \Wc{{\cal W}}

\def \eps{\varepsilon}

\allowdisplaybreaks

\begin{document}

\title{Master Bellman equation in the Wasserstein space:\\ Uniqueness of viscosity solutions}

\author{
Andrea COSSO\footnote{University of Bologna, Italy; andrea.cosso@unibo.it} \quad
Fausto GOZZI\footnote{Luiss University, Roma, Italy; fgozzi@luiss.it} \quad
Idris KHARROUBI\footnote{LPSM, UMR CNRS 8001, Sorbonne University and Universit\'e de Paris; idris.kharroubi@upmc.fr} \quad \\
Huy\^en PHAM\footnote{LPSM, UMR CNRS 8001, Sorbonne University and Universit\'e de Paris; pham@lpsm.paris;
The work of this author  is supported by FiME (Finance for Energy Market Research Centre) and the ``Finance et D\'eveloppement Durable - Approches Quantitatives'' EDF - CACIB Chair.
} \quad
Mauro ROSESTOLATO\footnote{University of Lecce; mauro.rosestolato@gmail.com}}

\maketitle

\begin{abstract}
We study the Bellman equation in the Wasserstein space arising in the study of mean field control problems, namely stochastic optimal control problems for McKean-Vlasov diffusion processes.
Using the standard notion of viscosity solution \`a la Crandall-Lions extended to our Wasserstein setting,
we prove a comparison result under general conditions   on the drift and reward coefficients, which
coupled with the dynamic programming principle, implies that the value function is the unique viscosity solution of the Master Bellman equation.
This is the first uniqueness result in such a second-order context. The classical arguments used in the standard cases of equations in finite-dimensional spaces or in infinite-dimensional separable Hilbert spaces do not extend to the present framework, due to the awkward nature of the underlying Wasserstein space. The adopted strategy is based on finite-dimensional approximations of the value function obtained in terms of the related cooperative $n$-player game, and on the construction of a smooth gauge-type function, built starting from a regularization of a sharpe estimate of the Wasserstein metric; such a gauge-type function is used to generate maxima/minima through a suitable extension of the Borwein-Preiss generalization of Ekeland's variational principle on the Wasserstein space.
\end{abstract}

\vspace{5mm}


\vspace{5mm}


\section{Introduction}

The main goal of this paper is to develop a viscosity theory for second-order partial differential equations on the Wasserstein space related to the so-called mean field (or McKean-Vlasov) control problems, namely stochastic optimal control problems for McKean-Vlasov diffusion processes. Such partial differential equations are also known as Master Bellman equations or Bellman equations in the Wasserstein space, see for instance \cite{benetal15,CD18_I,WuZhang20}. The topic of mean field optimal control is a very recent area of research, on which there are however already many papers and the two monographs \cite{benetal13,CD18_I}, to which we refer for a thorough introduction. Mean field control problems are strictly related to mean field games, developed by Lasry and Lions in \cite{LL1,LL2,laslio} (see also Lions' lectures at Coll{\`e}ge de France \cite{LionsVideo}) and by Huang, Caines, Malham\'e \cite{huangetal}. Both mean field control problems and mean field games can be interpreted as searches for equilibria of stochastic differential games with a continuum of players, symmetrically interacting each other through the empirical distribution of the entire population. These two problems differ because of the notion of equilibrium adopted. Mean field games arise when the concept of Nash's non-cooperative equilibrium is employed, while mean field control problems are related to Pareto optimality where players can be identified with a single ``representative agent'', see for instance
\cite[Section 6.2, pages 514-515]{CD18_I}.
In the latter case the stochastic differential game can be thought as an optimization problem of a central planner, who is looking for a common strategy in order to optimize some collective  objective functional.

The state space of mean field control problems is the set of probability measures, and usually
the Wasserstein space $\Pc_2(\R^d)$ of probability measures having finite second moment is adopted. Various notions of differentiability for maps defined on spaces of probability measures are available, and some of them are particularly relevant in the theory of optimal transportation, see \cite{AGS08,Vi09} for a detailed presentation of these geometric approaches. The Master Bellman equation (see equation \eqref{HJB} below) adopts instead the notion of differentiability introduced by Lions \cite{LionsVideo} (see also \cite{carda12,caretal20,CD18_I}, and Section \ref{S:HJB}), whose nature is more functional
analytic than geometric. Such a definition seems to be  the natural choice in the study of second-order Bellman equations in the Wasserstein space and related stochastic optimal control problems. In fact, it gave rise to a stochastic calculus on the space of proba\-bility measures, and in particular to an It\^o formula (chain rule)  for maps defined on the Wasserstein space (we recall it in our Theorem \ref{T:Ito}), which allows to relate the value function of the control problem to the Bellman equation (we recall it in our Theorem \ref{T:Exist}).  Regarding the relation between partial differential equations adopting the derivatives introduced by Lions (as in the present paper) and equations using notions of differentiability as those adopted in optimal transport theory, we mention  results in this direction in the first-order case  in \cite{GT19} and in a second-order semi-linear case in \cite{GMS21} (see also Remark \ref{rm:defvisc}).

The theory of partial differential equations in the Wasserstein space is an emerging research topic, whose rigorous investigation is still at an early stage. There are already well-posedness results in the first-order case, see \cite{AG08,ganngutud,AF14,GS14,GS15,GT19}, even  for equations adopting different notions of derivative with respect to the measure. They however do not admit an extension to the second-order case, which is notoriously a different and more challenging problem. Concerning second-order equations, papers \cite{phawei17,PW18,BCP18,CossoPham19,CKGPR20} focus on the existence of viscosity solutions, proving that the value function solves in the viscosity sense the Master Bellman equation. All those articles adopt the notion of viscosity solution \`a la Crandall-Lions, properly adapted to the Wasserstein space, as we do in the present paper (see Definition \ref{D:Viscosity}). Notice that, even if these papers dealt with the uniqueness property, they established it only for the so-called lifted Bellman equation, which is formulated on the Hilbert space of corresponding random variables so that standard results apply. We also recall that the relation between such a lifted equation and the original Bellman equation in the Wasserstein space is not  rigorously clarified, and in particular whether the lifted value function is a viscosity solution to the lifted equation.  Actually, it is not yet clear under which conditions test functions in the lifting Hilbert space
are related to test functions in the Wasserstein space, see discussion in Remark \ref{rm:defvisc}.

Uniqueness for second-order equations in the Wasserstein space is only addressed in the two papers  \cite{WuZhang20} and \cite{buretal20}. In \cite{WuZhang20}, a new notion of viscosity-type solution is adopted, which differs from the Crandall-Lions definition since the maximum/minimum condition is formulated on compact subsets of the Wasserstein space. This modification makes easier to prove uniqueness, which is completely established in some specific cases. On the other hand, \cite{buretal20} studies viscosity solutions \`a la Crandall-Lions for a class of integro-differential Bellman equations of particular type. More precisely, the coefficients of the McKean-Vlasov stochastic differential equations, as well as the coefficients of the reward functional, do not depend on the state process itself, but only on its probability distribution. This allows to consider only deterministic functions of time as control processes in the mean field control problem, so that the Master Bellman equation has a particular form. Moreover, in \cite{buretal20} the Master Bellman equation is formulated on the subset of the Wasserstein space of probability measures having finite exponential moments, equipped with the topology of weak convergence, which makes such a space $\sigma$-compact and allows establishing uniqueness in this context.

In the present paper we prove, under general conditions  on the drift and reward coefficients, existence and uniqueness of viscosity solutions for Master Bellman equations arising in the study of mean field optimal control problems. This is the first uniqueness result for such class of equations in the present context. Classical arguments based on Ishii's lemma used in the standard cases of equations in finite-dimensional spaces or in infinite-dimensional separable Hilbert spaces seem hard to extend  to the present framework, due to the awkward nature of the underlying Wasserstein space. The adopted strategy is instead based on refinements of early ideas from the theory of viscosity solutions \cite{L83b} and relies on the existence of a candidate solution to the equation, which in our case is the value function $v$ of the mean field control problem. In particular, we prove (see Theorem \ref{T:Comparison}) that any viscosity subsolution $u_1$ (resp. supersolution $u_2$) is smaller (resp. greater) than the candidate solution $v$. In \cite{L83b}, the arguments for proving $u_1\leq v$ (or, similarly, $v\leq u_2$) are as follows: one performs a smoothing $v_n$ of $v$ through its control representation, take a maximum of $u_1-v_n$ (relying on the local compactness of the finite-dimensional space), and exploit the viscosity subsolution property of $u_1$ with $v_n$ as test function. In \cite{lio89} such a methodology is extended to the infinite-dimensional case, relying on Ekeland's variational principle in order to generate maxima/minima.

In the context of equations in the Wasserstein space, the  above arguments require the following adjustments. Firstly, the smoothing of $v$ is based on a propagation of chaos result \cite{Lacker}, namely on a finite-dimensional approximation of the value function through value functions of non-degenerate cooperative $n$-player games. Secondly, in order to generate maxima/minima the idea is to perturb $u_1-v_n$ (or $u_2-v_n$) relying on a suitable extension of the Borwein-Preiss generalization of Ekeland's principle, see \cite[Theorem 2.5.2]{BZ05}. According to the latter, $u_1-v_n$ can be perturbed using a so-called gauge-type function (see Definition \ref{D:Gauge}). For the proof of the comparison theorem, such a perturbation has to be smooth. In an infinite-dimensional Hilbert space setting, an example of smooth gauge-type function is the square of the norm. In the present context, the main issue  is  to construct a {\it smooth} gauge-type function. This is achieved in Section \ref{S:SmoothVP}, starting from a sharp estimate of the square of the Wasserstein metric (see \eqref{UpperBound}) and performing a smoothing of such a quantity (see Lemma \ref{L:SmoothGauge}).
Due to the complexity of the techniques employed, our results are formulated under boundedness assumptions on the coefficients.
The extension to more general cases covering path dependent cases (like in \cite{CKGPR20,CFGRT,RenRosest}) and/or applications like the ones mentioned in \cite[Introduction]{DGZZ} seems possible and will be the object of future research. Similarly refinements of the results showing that viscosity solutions can have a certain degree of regularity (on the line of what is done e.g. in \cite{RosestolatoSwiech17})
seems possible and will be studied in further research.

The rest of the paper is organized as follows. In Section \ref{S:MKV_Control_Pb} we formulate the mean field optimal control problem and state the assumptions that are used throughout the paper; in such a section we also prove some properties of the value function $v$ and state the dynamic programming principle. In Section \ref{S:HJB} we recall the notion of differentiability introduced by Lions, we state the It\^o formula, we introduce the Master Bellman equation, and we give the definition of viscosity solution. Section \ref{S:SmoothVP} is devoted to the construction of the smooth gauge-type function, from which we derive the smooth variational principle on $[0,T]\times\Pc_2(\R^d)$, namely Theorem \ref{T:SmoothVarPrinc}. In Section \ref{S:Comparison} we prove the comparison theorem (Theorem \ref{T:Comparison}), from which we deduce the uniqueness result (Corollary \ref{C:Uniqueness}). Finally, in Appendix \ref{S:AppApprox} we perform the smooth finite-dimensional approximation of the value function; in particular, in subsection \ref{SubS:A1} we approximate the mean field control problem with non-degenerate control problems; then, in subsection \ref{SubS:CooperativeGame} we introduce the related cooperative $n$-player game and state the propagation of chaos result.

\section{Mean field optimal control problem}
\label{S:MKV_Control_Pb}

\paragraph{Wasserstein spaces of probability measures.}
Given a Polish space $({\rm S},d_{\rm S})$, we denote by $\Pc(\rm S)$ the set of all probability measures on $(\rm S,\Bc(\rm S))$. We also define, for every $q\geq 1$,
\[
\Pc_q({\rm S}) \ := \ \left\{\mu \in \Pc({\rm S})\colon
\hbox{ for some (and hence for all) $x_0\in{\rm S}$, } \int_{\rm S} d_{\rm S}(x_0,x)^q\,\mu(dx)<+\infty\right\}.
\]
The set $\Pc_q({\rm S})$ is endowed with the $q$-Wasserstein distance defined as
\begin{align}\label{Wasserstein}
\Wc_q(\mu,\mu') \ &:= \ \inf\bigg\{\int_{{\rm S}\times{\rm S}}d_{\rm S}(x,y)^q\,\pi(dx,dy)\colon
  \pi \in \Pc({\rm S}\times{\rm S})\\
 & \hspace{2cm} \text{ such that } \;\pi(\cdot\times{\rm S})= \mu \mbox{ and }\pi({\rm S}\times\cdot)=\mu'\bigg\}^{1\over q}\,,\qquad q\geq 1, \notag
\end{align}
for every $\mu,\mu'\in \Pc_q({\rm S})$. The space $\big(\Pc_q({\rm S}),\Wc_q)$ is a Polish space, see for instance \cite[Theorem 6.18]{Vi09}.

\paragraph{Probabilistic setting and control processes.}
We fix a complete probability space $(\Omega,\Fc,\P)$ on which a $m$-dimensional Brownian motion $B=(B_t)_{t\geq0}$ is defined. We denote by $\F^B=(\Fc_t^B)_{t\geq0}$ the $\P$-completion of the filtration generated by $B$, which is also right-continuous, so that it satisfies the usual conditions. We assume that there exists a sub-$\sigma$-algebra $\Gc$ of $\Fc$ satisfying the following properties.
\begin{enumerate}[i)]
\item $\Gc$ and $\Fc_\infty^B$ are independent.
\item $\Gc$ is ``rich enough'', namely $\Pc_2(\R^d)=\{\P_\xi$ such that $\xi\colon\Omega\rightarrow\R^d,$ with $\xi$ being $\Gc$-measurable and $\E|\xi|^2<\infty\}$. Recall from \cite[Lemma 2.1]{CKGPR20} that such a requirement is equivalent to the existence of a $\Gc$-measurable random variable $U\colon\Omega\rightarrow\R$ having uniform distribution on $[0,1]$.
\end{enumerate}
We denote by $\F=(\Fc_t)_{t\geq0}$ the  filtration defined as
\[
\Fc_t \ = \ \Gc\vee\Fc_t^B\,, \qquad t\geq0\,.
\]
We observe that $\F$ satisfies the usual conditions of $\P$-completeness and right-continuity.\\
Finally, we fix a finite time horizon $T>0$ and a Polish space $A$. We then denote by $\Ac$ the set of control processes, namely the family of all $\F$-progressively measurable processes $\alpha\colon[0,T]\times\Omega\rightarrow A$.

\paragraph{Assumptions and state equation.}
We consider the functions $b\colon[0,T]\times\R^d\times\Pc_2(\R^d)\times A\rightarrow\R^d$,  $\sigma\colon[0,T]\times\R^d\times A \rightarrow\R^{d\times m}$, $f\colon[0,T]\times\R^d\times\Pc_2(\R^d)\times A\rightarrow\R$, $g\colon\R^d\times\Pc_2(\R^d)\rightarrow\R$ on which we impose the following assumptions  
(notice that $\sigma$ does not depend on $\mu$).

\begin{Assumption}{\bf(A)}\label{AssA}\quad
\begin{enumerate}[\upshape (i)]
\item The functions $b,\sigma,f,g$ are  continuous.
\item There exists a constant $K\geq0$ such that
\begin{align*}
|b(t,x,\mu,a)-b(t,x',\mu',a)| + |\sigma(t,x,a) - \sigma(t,x',a)| \ &\leq \ K\big(|x - x'| + \Wc_2(\mu,\mu')\big), \\
 |b(t,x,\mu,a)| +   |\sigma(t,x,a)|  \ &\leq \ K,
\end{align*}
for all $(t,a)\in[0,T]\times A$, $(x,\mu),(x',\mu')\in\R^d\times\Pc_2(\R^d)$, $|x-x'|$ denoting the Euclidean norm of $x-x'$ in $\R^d$, $\langle\cdot,\cdot\rangle$ denoting the scalar product, 
$|\sigma(t,x,a)|:=(\textup{tr}(\sigma\sigma\trans)(t,x,a))^{1/2}$ $=(\sum_{i,j}|\sigma_{i,j}(t,x,a)|^2)^{1/2}$  denoting the Frobenius norm of the matrix  $\sigma(t,x,a)$.
\item[\textup{ (iii)}] There exists a constant $K\geq0$ such that
\begin{align*}
|f(t,x,\mu,a)-f(t,x',\mu',a)| + |g(x,\mu)-g(x',\mu')| \ &\leq \ K\big(|x - x'| + \Wc_2(\mu,\mu')\big),
\\
|f(t,x,\mu,a)| + |g(x,\mu)|  \ &  \leq \ K,
\end{align*}
for all $(t,a)\in[0,T]\times A$, $(x,\mu),(x',\mu')\in\R^d\times\Pc_2(\R^d)$.
\end{enumerate}
\end{Assumption}

\begin{Assumption}{\bf(B)}\label{AssB}\quad
There exist constants $K\geq0$ and $\beta\in(0,1]$ such that
\begin{align*}
|b(t,x,\mu,a) - b(s,x,\mu,a)| &+  |\sigma(t,x,a) - \sigma(s,x,a)|  \\
&+ |f(t,x,\mu,a) - f(s,x,\mu,a)| \ \leq \ K|t - s|^\beta,
\end{align*}
for all $t,s\in[0,T],(x,\mu,a)\in\R^d\times\Pc_2(\R^d)\times A$.
\end{Assumption}

\begin{Assumption}{\bf(C)}\label{AssC}
The Polish space $A$ is a  compact subset of a Euclidean space.
\end{Assumption}

\begin{Assumption}{\bf(D)}\label{AssD}
For any $a\in A$, the function $\sigma(\cdot,\cdot,a)$ belongs to $C^{1,2}([0,T]\times\R^d)$. Moreover, there exists some constant $K\geq0$ such that
\[
\big|\partial_t\sigma(t,x,a)\big| + \big|\partial_{x_i}\sigma(t,x,a)\big| + \big|\partial_{x_ix_j}^2\sigma(t,x,a)\big| \ \leq \ K,
\]
for all $(t,x,a)\in[0,T]\times\R^d\times A$ and any $i,j=1,\ldots,d$.
\end{Assumption}

\begin{Remark}
Assumptions \ref{AssB} and \ref{AssD} are required in the proof of Theorem \ref{T:SmoothApprox} in order to exploit regularity results for uniformly parabolic Bellman equations.  In particular, Assumption \ref{AssD} is taken from \cite[Section 7 of Chapter 4]{Krylov80} in order to get suitable bounds on the second derivatives (see \cite[Theorem 4.7.4]{Krylov80}). On the other hand,  Assumptions \ref{AssA} and \ref{AssC} are required in the propagation of chaos result, that is Theorem \ref{T:PropagChaos}. All these assumptions are therefore required in Theorem \ref{T:Comparison} and Corollary \ref{C:Uniqueness}. Finally, notice that the results of the present section are stated under Assumption \ref{AssA}, however they hold under weaker assumptions, see \cite{CKGPR20}.
\end{Remark}

\noindent For every $t\in[0,T]$, $\xi\in L^2(\Omega,\Fc_t,\P;\R^d)$, $\alpha\in\Ac$, the state process evolves according to the following controlled McKean-Vlasov stochastic differential equation:
\begin{equation}\label{SDE}
X_s \ = \ \xi + \int_t^s b\big(r,X_r,\P_{X_r},\alpha_r\big)\,dr
+ \int_t^s \sigma(r,X_r,\alpha_r) \,dB_r, \qquad s\in[t,T].
\end{equation}

\begin{Proposition}
Suppose that Assumption \textup{\ref{AssA}} holds. For every $t\in[0,T]$, $\xi\in L^2(\Omega,\Fc_t,\P;\R^d)$, $\alpha\in\Ac$, there exists a unique (up to $\P$-indistinguishability) continuous $\F$-progressively measurable process $X^{t,\xi,\alpha}=(X_s^{t,\xi,\alpha})$ solution to equation \eqref{SDE} satisfying
\[
\E\Big[\sup_{s\in[t,T]}\big|X_s^{t,\xi,\alpha}\big|^2\Big]^{1/2} \ \leq \ C_1\Big(1 + \E\big[|\xi|^2\big]^{1/2}\Big),	
\]
for some constant $C_1$, independent of $t,\xi,\alpha$.\\
Moreover, for every $\xi'\in L^2(\Omega,\Fc_t,\P;\R^d)$ it holds that
\begin{equation}\label{Estimate_Lipschitz}
\E\Big[\sup_{s\in[t,T]}\big|X_s^{t,\xi,\alpha} - X_s^{t,\xi',\alpha}\big|^2\Big]^{1/2} \ \leq \ C_2\,\E\big[|\xi - \xi'|^2\big]^{1/2},
\end{equation}
for some constant $C_2$, independent of $t,\xi,\xi',\alpha$.
\end{Proposition}
\begin{proof}[\textbf{Proof.}]
See \cite[Proposition 2.8]{CKGPR20}.
\end{proof}

\paragraph{Reward functional and lifted value function.} We consider the \emph{reward functional} $J$, given by
\begin{equation}\label{RewardFunct}
J(t,\xi,\alpha) \ = \ \E\bigg[\int_t^T f\big(s,X_s^{t,\xi,\alpha},\P_{X_s^{t,\xi,\alpha}},\alpha_s\big)\,ds + g\big(X_T^{t,\xi,\alpha},\P_{X_T^{t,\xi,\alpha}}\big)\bigg],
\end{equation}
and the function $V$, to which we refer as the \emph{lifted value function}, defined as
\[
V(t,\xi) \ = \ \sup_{\alpha\in\Ac} J(t,\xi,\alpha), \qquad \forall\,(t,\xi)\in[0,T]\times L^2(\Omega,\Fc_t,\P;\R^d).
\]

\begin{Proposition}\label{P:ValueFunction}
Suppose that Assumption \textup{\ref{AssA}} holds. The function $V$ satisfies the following properties.
\begin{enumerate}[\upshape 1)]
\item $V$ is bounded.
\item $V$ is jointly continuous, namely: for every $\{(t_n,\xi_n)\}_n,(t,\xi)$, with $t_n,t\in[0,T]$ and $\xi_n\in L^2(\Omega,\Fc_{t_n},\P;\R^d),\xi\in L^2(\Omega,\Fc_t,\P;\R^d)$, such that $|t_n-t|+\E|\xi_n-\xi|^2\rightarrow0$, it holds that $V(t_n,\xi_n)\rightarrow V(t,\xi)$.
\item There exists a constant $L\geq0$ $($depending only on $T$, $K$, $C_2$ in \eqref{Estimate_Lipschitz}$)$ such that
\begin{equation}\label{Value_Lifted_Lipschitz}
|V(t,\xi) - V(t,\xi')| \ \leq \ L\,\E\big[|\xi - \xi'|^2\big]^{1/2},
\end{equation}
for all $t\in[0,T]$, $\xi,\xi'\in L^2(\Omega,\Fc_t,\P;\R^d)$.
\end{enumerate}
\end{Proposition}
\begin{proof}[\textbf{Proof.}]
Item 1) is a direct consequence of the boundedness of $f$ and $g$, while item 2) follows from \cite[Proposition 3.3]{CKGPR20}. Concerning item 3), we begin noticing that
\[
|V(t,\xi) - V(t,\xi')| \ \leq \ \sup_{\alpha\in\Ac} |J(t,\xi,\alpha) - J(t,\xi',\alpha)|.
\]
Then, the Lipschitz continuity of $V$ follows from the Lipschitz continuity of $J$. To this regard, we have
\begin{align*}
|J(t,\xi,\alpha) - J(t,\xi',\alpha)| \ &\leq \ \int_t^T \E\big[\big|f\big(s,X_s^{t,\xi,\alpha},\P_{X_s^{t,\xi,\alpha}},\alpha_s\big) - f\big(s,X_s^{t,\xi',\alpha},\P_{X_s^{t,\xi',\alpha}},\alpha_s\big)\big|\big]ds \\
&\quad \ + \E\big[\big|g\big(X_T^{t,\xi,\alpha},\P_{X_T^{t,\xi,\alpha}}\big) - g\big(X_T^{t,\xi',\alpha},\P_{X_T^{t,\xi',\alpha}}\big)\big|\big].
\end{align*}
By the Lipschitz continuity of $f$ and $g$, together with inequality $\Wc_2(X_s^{t,\xi,\alpha},X_s^{t,\xi',\alpha})\leq\E[|X_s^{t,\xi,\alpha}-X_s^{t,\xi',\alpha}|^2]^{1/2}$, we obtain estimate \eqref{Value_Lifted_Lipschitz}.
\end{proof}

\paragraph{Law invariance property and dynamic programming principle.} We recall from \cite{CKGPR20} that $V$ satisfies the fundamental \emph{law invariance property}.

\begin{Theorem}\label{T:LawInv}
Suppose that Assumption \textup{\ref{AssA}} holds. Then, the map $V$ satisfies the law invariance property: for every $t\in[0,T]$ and $\xi,\xi'\in L^2(\Omega,\Fc_t,\P;\R^d)$, with $\P_\xi=\P_{\xi'}$, it holds that
\[
V(t,\xi) \ = \ V(t,\xi').
\]
\end{Theorem}
\begin{proof}[\textbf{Proof.}]
See \cite[Theorem 3.5]{CKGPR20}.
\end{proof}

\noindent As a consequence of Theorem \ref{T:LawInv}, if
Assumption \textup{\ref{AssA}} holds, we can define the \emph{value function} $v\colon[0,T]\times\Pc_2(\R^d)\rightarrow\R$ as
\begin{equation}\label{Value}
v(t,\mu) \ = \ V(t,\xi), \qquad \forall\,(t,\mu)\in[0,T]\times\Pc_2(\R^d),	
\end{equation}
for any $\xi\in L^2(\Omega,\Fc_t,\P;\R^d)$. By Proposition \ref{P:ValueFunction} we immediately deduce the following result.
\begin{Proposition}\label{P:ValueFunction}
Suppose that Assumption \textup{\ref{AssA}} holds. The function $v$ satisfies the following properties.
\begin{enumerate}[\upshape 1)]
\item $v$ is bounded.
\item $v$ is jointly continuous on $[0,T]\times\Pc_2(\R^d)$.
\item For all $t\in[0,T]$, $\mu,\mu'\in\Pc_2(\R^d)$,
\[
|v(t,\mu) - v(t,\mu')| \ \leq \ L\,\Wc_2(\mu,\mu'),
\]
with $L$ as in \eqref{Value_Lifted_Lipschitz}.
\end{enumerate}
\end{Proposition}
\begin{proof}[\textbf{Proof.}]
The claim follows directly from Proposition \ref{P:ValueFunction}, we only report the proof of item 3). By \eqref{Value_Lifted_Lipschitz} and \eqref{Value}, we have
\[
|v(t,\mu) - v(t,\mu')| \ \leq \ |V(t,\xi) - V(t,\xi')| \ \leq \ L\,\E\big[|\xi - \xi'|^2\big]^{1/2},
\]
for any $\xi,\xi'\in L^2(\Omega,\Fc_t,\P;\R^d)$, with $\P_\xi=\mu$ and $\P_{\xi'}=\mu'$. Hence
\begin{align*}
|v(t,\mu) - v(t,\mu')| &\leq L\inf\Big\{\E\big[|\xi - \xi'|^2\big]^{1/2}\colon\xi,\xi'\in L^2(\Omega,\Fc_t,\P;\R^d),\text{ with }\P_\xi=\mu\text{ and }\P_{\xi'}=\mu'\Big\} \\
&= L\,\Wc_2(\mu,\mu').
\end{align*}
\end{proof}

\noindent Finally, we state the dynamic programming principle for $v$.

\begin{Theorem}
Suppose that Assumption \textup{\ref{AssA}} holds. Then, $v$ satisfies the dynamic programming principle: for all $t,s\in[0,T]$, with $t\leq s$, $\mu\in\Pc_2(\R^d)$, it holds that
\[
v(t,\mu) \ = \ \sup_{\alpha\in\Ac}\bigg\{\E\bigg[\int_t^s f\big(r,X_r^{t,\xi,\alpha},\P_{X_r^{t,\xi,\alpha}},\alpha_r\big)\,dr\bigg] + v\big(s,\P_{X_s^{t,\xi,\alpha}}\big)\bigg\},
\]
for any $\xi\in L^2(\Omega,\Fc_t,\P;\R^d)$ with $\P_\xi=\mu$.
\end{Theorem}
\begin{proof}[\textbf{Proof.}]
See \cite[Corollary 3.8]{CKGPR20}.
\end{proof}

\section{Master Bellman equation}
\label{S:HJB}

\paragraph{$L$-derivatives and It\^o's formula along a flow of probability measures.}

We refer to \cite[Section 5.2]{CD18_I} for the definitions of the $L$-derivatives of first and second-order of a map $u\colon[0,T]\times\Pc_2(\R^d)\rightarrow\R$ with respect to $\mu$, which are given by $\partial_\mu u\colon[0,T]\times\Pc_2(\R^d)\times\R^d\rightarrow\R$ and $\partial_x\partial_\mu u\colon[0,T]\times\Pc_2(\R^d)\times\R^d\rightarrow\R^{d\times d}$. We recall that such definitions are based on the notion of \emph{lifting} of a map $u\colon[0,T]\times\Pc_2(\R^d)\rightarrow\R$, which is a map $U\colon[0,T]\times L^2(\Omega,\Fc,\P;\R^d)\rightarrow\R$ satisfying
\begin{equation}\label{eq:liftingdef}
U(t,\xi) \ = \ u(t,\P_\xi),
\end{equation}
for every $t\in[0,T]$, $\xi\in L^2(\Omega,\Fc,\P;\R^d)$ (here, to alleviate notation, we have defined the lifting on the same probability space $(\Omega,\Fc,\P)$ on which the mean field control problem was defined; however, any other probability space supporting a random variable with uniform distribution on $[0,1]$ can be used).
We observe that derivatives in the present context can be defined in different ways: the so-called ``flat'' derivative or the intrinsic notion of differential in the Wasserstein space. We refer for instance to \cite[Chapter 5]{CD18_I} for a survey and some equivalence results.

\begin{Definition}
$C^{1,2}([0,T]\times\Pc_2(\R^d))$ is the set of continuous functions $u\colon[0,T]\times\Pc_2(\R^d)\rightarrow\R$ such that:
\begin{enumerate}[\upshape 1)]
\item the lifting $U$ of $u$ admits a continuous Fr\'echet derivative $D_\xi U\colon[0,T]\times L^2(\Omega,\Fc,\P;\R^d)\rightarrow L^2(\Omega,\Fc,\P;\R^d)$,
in which case there exists, for any $(t,\mu)$ $\in$ $[0,T]\times\Pc_2(\R^d)$,  a measurable function $\partial_\mu u(t,\mu)$ $:$ $\R^d$ $\rightarrow$ $\R^d$, such that $D_\xi U(t,\xi)$ $=$ $\partial_\mu u(t,\mu)(\xi)$, for any $\xi$ $\in$ $L^2(\Omega,\Fc,\P;\R^d)$ with law $\mu$.
\item The map $(t,x,\mu)$ $\in$ $[0,T]\times\R^d\times\Pc_2(\R^d)$ $\mapsto$ $\partial_\mu u(t,\mu)(x)$ $\in$ $\R^d$ is jointly continuous;
\item $\partial_t u$ and $\partial_x\partial_\mu u$ exist and the maps $(t,\mu)$ $\in$ $[0,T]\times\Pc_2(\R^d)$ $\mapsto$ $\partial_t u(t,\mu)$ $\in$ $\R$,
$(t,x,\mu)$  $\in$ $[0,T]\times\R^d\times\Pc_2(\R^d)$ $\mapsto$ $\partial_x\partial_\mu u(t,\mu)(x)$ $\in$ $\R^{d\times d}$  are continuous.
\end{enumerate}
\end{Definition}

\begin{Definition}
$C_2^{1,2}([0,T]\times\Pc_2(\R^d))$  is the subset of $C^{1,2}([0,T]\times\Pc_2(\R^d))$ of functions $u\colon[0,T]\times\Pc_2(\R^d)\rightarrow\R$ satisfying, for some constant $C\geq0$,
\[
|\partial_\mu u(t,\mu)(x)| + |\partial_x\partial_\mu u(t,\mu)(x)| \ \leq \ C\big(1 + |x|^2\big),
\]	
for all $(t,\mu,x)\in[0,T]\times\Pc_2(\R^d)\times\R^d$.
\end{Definition}

\begin{Theorem}[It\^o's formula]\label{T:Ito}
Let $u\in C_2^{1,2}([0,T]\times\Pc_2(\R^d))$, $t\in[0,T]$, $\xi\in L^2(\Omega,\Fc_t,\P;\R^d)$. Let also $\beta\colon[0,T]\times\Omega\rightarrow\R^d$ and $\vartheta\colon[0,T]\times\Omega\rightarrow\R^{d\times m}$ be bounded and $\F$-progressively measurable processes.
Consider the $d$-dimensional It\^o process
\[
X_s \ = \ \xi + \int_t^s \beta_r\,dr + \int_t^s \vartheta_r\,dB_r, \qquad \forall\,s\in[t,T].
\]
Then, it holds that
\begin{align*}
u(s,\P_{X_s}) \ &= \ u(t,\P_\xi) + \int_t^s \partial_t u(r,\P_{X_r})\,dr + \int_t^s \E\Big[\big\langle \beta_r,\partial_\mu u(r,\P_{X_r})(X_r)\big\rangle\Big]dr \\
&\quad \ + \frac{1}{2}\int_t^s \E\Big[\textup{tr}\Big(\vartheta_r\vartheta_r\trans\partial_x\partial_\mu u(r,\P_{X_r})(X_r)\Big)\Big],
\end{align*}
for all $s\in[t,T]$.
\end{Theorem}
\begin{proof}[\textbf{Proof.}]
The claim follows from \cite[Proposition 5.102]{CD18_I} (see also \cite[Theorem 4.15 and Remark 4.16]{CKGPR20}).
\end{proof}

\paragraph{Viscosity solutions.}

\noindent Now, consider the second-order partial differential equation on $[0,T]\times\Pc_2(\R^d)$:
\begin{equation}\label{PDE_general}
\hspace{-5mm}\begin{cases}
\vspace{2mm}\displaystyle
\partial_t u(t,\mu) = F\big(t,\mu,u(t,\mu),\partial_\mu u(t,\mu)(\cdot),\partial_x\partial_\mu u(t,\mu)(\cdot)\big), &(t,\mu)\in[0,T)\times\Pc_2(\R^d), \\
\displaystyle u(T,\mu) = \int_{\R^d} g(x,\mu)\mu(dx), &\,\mu\in\Pc_2(\R^d),
\end{cases}
\end{equation}
with $F\colon[0,T]\times\Pc_2(\R^d)\times\R\times L^2(\R^d,\Bc(\R^d),\mu;\R^d)\times L^2(\R^d,\Bc(\R^d),\mu;\R^{d\times d})\rightarrow\R$, where $\Bc(\R^d)$ are the Borel subsets of $\R^d$, and
$L^2(\R^d,\Bc(\R^d),\mu;\R^d)$ is the set of $\Bc(\R^d)$-measurable functions that are square-integrable with respect to $\mu$.

\begin{Definition}
A function $u\colon[0,T]\times\Pc_2(\R^d)\rightarrow\R$ is a classical solution to equation \eqref{PDE_general} if  $u\in C_2^{1,2}([0,T]\times\Pc_2(\R^d))$ and satisfies \eqref{PDE_general}.
\end{Definition}

\begin{Definition}\label{D:Viscosity}
A continuous function $u\colon[0,T]\times\Pc_2(\R^d)\rightarrow\R$ is a viscosity subsolution (resp. supersolution) to equation \eqref{PDE_general} if:
\begin{itemize}
\item $u(T,\mu)\leq(\text{resp. $\geq$})\,\int_{\R^d}g(x,\mu)\mu(dx)$, for every $\mu\in\Pc_2(\R^d)$;
\item for every $(t,\mu)\in[0,T)\times\Pc_2(\R^d)$ and any $\varphi\in  C_2^{1,2}([0,T]\times\Pc_2(\R^d))$ such that $u-\varphi$ has a maximum at $(t,\mu)$ $($with value $0$$)$, then \eqref{PDE_general} is satisfied with the inequality $\geq$ $($resp. $\leq$$)$ instead of the equality and with $\varphi$ in place of $u$.
\end{itemize}
Finally, $u$ is a viscosity solution of
\eqref{PDE_general} if it is both a viscosity subsolution and a viscosity supersolution.
\end{Definition}

\begin{Remark}
\label{rm:defvisc}
The above definition of viscosity solution is exactly in the spirit
of the definition of Crandall and Lions for second-order equations (see for instance \cite{CIL92}) in finite dimension. In \cite{CIL92} it is proved that this definition is equivalent to the one using second-order semidifferentials (jets), while here such equivalence is not obvious.

We say that our definition is an ``intrinsic'' definition to distinguish it from the definition, adopted first in \cite{phawei17}, which exploits the lifted equation (in the sense that a function is a viscosity solution if its lifting along \eqref{eq:liftingdef} satisfies equation  \eqref{PDE_general} with $F$ substituted by its lifting $\tilde F$) and which, for this reason, we call ``lifted'' definition.

The relationship between these two definitions is not obvious. Indeed,  as shown in Exam\-ple 2.1 in \cite{bucetal17}, the lifted function of a smooth function on the Wasserstein space may not be smooth on the lifted Hilbert space, and so a viscosity solution  in the intrinsic sense may not be a viscosity solution in the lifted sense.
In the first-order case a kind of equivalence result between
two related definitions is provided in \cite[Theorem 4.4]{GT19}.
In the second-order semi-linear case some results in this direction are provided in \cite[Section 5]{GMS21}.  We are not aware of any results on the fully non-linear second-order case.

As we recalled in the introduction an intrinsic notion of viscosity solution is employed also in the papers  \cite{WuZhang20} and \cite{buretal20}.
The definition introduced in \cite[Definition 4.4]{WuZhang20},
differs from our definition since test functions must satisfy the maximum/minimum condition on suitable compact subsets of the Wasserstein space, denoted by $\mathcal{P}_L$. Using this modification the authors prove first a comparison result among regular sub/supersolutions (``partial comparison'') and then a general comparison result with the assumption that the supremum of classical subsolutions and the infimum of classical supersolutions coincide.
On the other hand, \cite{buretal20} studies viscosity solutions \`a la Crandall-Lions for a class of integro-differential Bellman equations of particular type. More precisely, the coefficients of the McKean-Vlasov stochastic differential equations, as well as the coefficients of the reward functional, do not depend on the state process itself, but only on its probability distribution. This allows to consider only deterministic functions of time as control processes in the mean field control problem, so that the Master Bellman equation has a particular form. Moreover, in \cite{buretal20} the Master Bellman equation is formulated on the subset of the Wasserstein space of probability measures having finite exponential moments, equipped with the topology of weak convergence, which makes such a space $\sigma$-compact and allows establishing uniqueness in this context.
\end{Remark}

\noindent Now, we consider the Master Bellman equation, namely equation \eqref{PDE_general} with
\begin{align*}
F(t,\mu,r,p(\cdot),M(\cdot)) \ = \ - \int_{\R^d}\sup_{a\in A}\bigg\{f(t,x,\mu,a) + \big\langle b(t,x,\mu,a),p(x)\big\rangle & \\
+ \, \frac{1}{2}\text{tr}\big[(\sigma\sigma\trans)(t,x,a)M(x)\big] & \bigg\}\mu(dx).
\end{align*}
Therefore, equation \eqref{PDE_general} becomes
\begin{equation}\label{HJB}
\begin{cases}
\vspace{2mm}
\displaystyle\partial_t u(t,\mu) + \int_{\R^d}\sup_{a\in A}\bigg\{f(t,x,\mu,a) + \big\langle b(t,x,\mu,a),\partial_\mu u(t,\mu)(x)\rangle \\
\displaystyle\vspace{2mm}+\,\dfrac{1}{2}\text{tr}\big[(\sigma\sigma\trans)(t,x,a)\partial_x\partial_\mu u(t,\mu)(x)\big]\bigg\}\mu(dx) = 0, &\hspace{-1cm}(t,\mu)\in[0,T)\times \Pc_2(\R^d), \\
\displaystyle u(T,\mu) = \int_{\R^d} g(x,\mu)\mu(dx), &\hspace{-1cm}\,\mu\in\Pc_2(\R^d).
\end{cases}
\end{equation}

\begin{Remark}
As described in \cite[Section 5.2]{CKGPR20}, to which we refer for more details, equation \eqref{HJB} can be written in various alternative forms. In particular, \eqref{HJB} corresponds to \cite[equation (5.17)]{CKGPR20}, the only difference being the presence of $\sup_{a\in A}$ which in \cite[equation (5.17)]{CKGPR20} is replaced by $\textup{ess\,sup}_{a\in A}$. However, as described in \cite[Remark 5.8]{CKGPR20}, under assumption \ref{AssA}, $\textup{ess\,sup}_{a\in A}$ can be replaced by $\sup_{a\in A}$.\\
Finally, we mention that an alternative form of equation \eqref{HJB} is the following (corresponding to equation \cite[equation (5.16)]{CKGPR20}):
\[
\begin{cases}
\vspace{2mm}
\displaystyle\partial_t u(t,\mu) + \sup_{\mathrm a\in\Mc}\bigg\{\int_{\R^d} f(t,x,\mu,\mathrm a(x))\mu(dx) + \int_{\R^d}\big\langle b(t,x,\mu,\mathrm a(x)),\partial_\mu u(t,\mu)(x)\big\rangle\mu(dx) \\
\displaystyle\vspace{2mm}+\,\dfrac{1}{2}\int_{\R^d}\textup{tr}\big[(\sigma\sigma\trans)(t,x,\mathrm a(x))\partial_x\partial_\mu u(t,\mu)(x)\big]\mu(dx)\bigg\} = 0, &\hspace{-4cm}(t,\mu)\in[0,T)\times \Pc_2(\R^d), \\
\displaystyle u(T,\mu) = \int_{\R^d} g(x,\mu)\mu(dx), &\hspace{-4cm}\,\mu\in\Pc_2(\R^d),
\end{cases}
\]
where $\Mc$ is the set of Borel-measurable maps $\mathrm a\colon\R^d\rightarrow A$.
\end{Remark}

\begin{Theorem}\label{T:Exist}
Let Assumptions \textup{\ref{AssA}}  and  \ref{AssB} hold. Then, the value function $v$, given by \eqref{Value}, is a viscosity solution of equation \eqref{HJB}.
\end{Theorem}
\begin{proof}[\textbf{Proof.}]
See \cite[Theorem 5.5]{CKGPR20}. Notice that in \cite{CKGPR20} a different definition of viscosity solution is adopted, with $\varphi\in C_b^{1,2}([0,T]\times\Pc_2(\R^d))$, \textit{i.e.} $\varphi\in C^{1,2}([0,T]\times\Pc_2(\R^d))$ and has bounded derivatives, rather than $\varphi\in C_2^{1,2}([0,T]\times\Pc_2(\R^d))$, but the arguments remain the same.
\end{proof}

\begin{Remark}
From \cite[Theorem 5.5]{CKGPR20} we have that Theorem \ref{T:Exist} still holds if we replace \ref{AssB} with the following weaker assumption: the functions $b,\sigma,f$ are uniformly continuous in the time variable $t$, uniformly with respect to $(x,\mu,a)$. Similarly, Assumption \ref{AssA} can be weakened, see \cite{CKGPR20}.
\end{Remark}

\section{Smooth variational principle}
\label{S:SmoothVP}

As  described in the introduction,  the comparison theorem (Theorem \ref{T:Comparison}) relies on a smooth variational principle on $[0,T]\times\Pc_2(\R^d)$, to which the present section is devoted. Such a result is obtained from an extension of the Borwein-Preiss variational principle, for which we refer to \cite{BP87} and, in particular, for its general form, to \cite[Theorem 2.5.2]{BZ05}. An essential tool of \cite[Theorem 2.5.2]{BZ05} is the concept of gauge-type function, whose definition is given below.

\begin{Definition}\label{D:Gauge}
Let $d_2$ be a metric on $\Pc_2(\R^d)$ such that $(\Pc_2(\R^d),d_2)$ is complete. Consider the set $[0,T]\times\Pc_2(\R^d)$ endowed with the metric $((t,\mu),(s,\nu))\mapsto|t-s|+d_2(\mu,\nu)$. A map $\rho\colon([0,T]\times\Pc_2(\R^d))^2\rightarrow[0,+\infty)$ is said to be a \textbf{gauge-type function} if the following holds.
\begin{enumerate}[\upshape a)]
\item $\rho((t,\mu),(t,\mu))=0$, for every $(t,\mu)\in[0,T]\times\Pc_2(\R^d)$.
\item $\rho$ is continuous on $([0,T]\times\Pc_2(\R^d))^2$.
\item For all $\eps>0$, there exists $\eta>0$ such that, for all $(t,\mu),(s,\nu)\in[0,T]\times\Pc_2(\R^d)$, the inequality $\rho((t,\mu),(s,\nu))\leq\eta$ implies $|t-s|+d_2(\mu,\nu)\leq\eps$.
\end{enumerate}
\end{Definition}

\noindent In the sequel we construct a gauge-type function on $[0,T]\times\Pc_2(\R^d)$, taking a particular metric $d_2$ on $\Pc_2(\R^d)$, namely the so-called Gaussian-smooothed 2-Wasserstein distance, see \cite{NGK21}. To this regard, we denote by $\Nc_\varrho:=\Nc(0,\varrho^2 I_d)$, for every $\varrho>0$, the $d$-dimensional multivariate normal distribution with zero mean and covariance matrix $\varrho^2 I_d$, with $I_d$ being the identity matrix of order $d$. Then, for every $\varrho>0$, the Gaussian-smoothed 2-Wasserstein distance is defined as
\[
\Wc_2^{(\varrho)}(\mu,\nu) \ := \ \Wc_2\big(\mu*\Nc_\varrho,\nu*\Nc_\varrho\big), \qquad \forall\,\mu,\nu\in\Pc_2(\R^d),
\]
where $*$ denotes the convolution of probability measures.

\begin{Lemma}\label{L:GaussianWass}
For every $\varrho>0$, $\Wc_2^{(\varrho)}$ is a metric on $\Pc_2(\R^d)$, inducing the same topology as $\Wc_2$. Moreover, $(\Pc_2(\R^d),\Wc_2^{(\varrho)})$ is a complete metric space.
\end{Lemma}
\begin{proof}[\textbf{Proof.}]
The first part follows from \cite[Proposition 1]{NGK21}. It remains to prove that the metric space $(\Pc_2(\R^d),\Wc_2^{(\varrho)})$ is complete. Let $\{\mu_n\}_n\subset\Pc_2(\R^d)$ be a Cauchy sequence with respect to $\Wc_2^{(\varrho)}$. Then, $\{\mu_n*\Nc_\varrho\}_n$ is a Cauchy sequence with respect to $\Wc_2$. Since $(\Pc_2(\R^d),\Wc_2)$ is complete, there exists some $\bar\nu\in\Pc_2(\R^d)$ such that $\Wc_2(\mu_n*\Nc_\varrho,\bar\nu)\rightarrow0$ as $n\rightarrow\infty$. It follows (see for instance \cite[Proposition 7.1.5]{AGS08}) that $\{\mu_n*\Nc_\varrho\}_n$ has uniformly integrable second moments, namely
\begin{equation}\label{UnifInt2nd}
\lim_{k\rightarrow\infty} \sup_n\int_{|y|\geq k} |y|^2\,(\mu_n*\Nc_\varrho)(dy) \ = \ 0.
\end{equation}
Now notice that, given $a,b,c\in\R_+$, $h\in\N$, if $a\leq b+c$ then $a1_{\{a\geq h\}}\leq2b1_{\{b\geq h/2\}}+2c1_{\{c\geq h/2\}}$. Hence, by the elementary inequality $|x|^2\leq2|x+z|^2+2|z|^2$, valid for every $x,z\in\R^d$, we get
\begin{equation}\label{Ineq_Unif_Int}
|x|^2\,1_{\{|x|\geq\sqrt{h}\}} \ \leq \ 4|x+z|^2\,1_{\{|x+z|\geq\sqrt{h/2}\}}+4|z|^2\,1_{\{|z|\geq\sqrt{h/2}\}}, \qquad \forall\,x,z\in\R^d,\,h\in\N.
\end{equation}
Integrating the above inequality on $\R^d\times\R^d$ with respect to the product measure $\mu_n(dx)\Nc_\varrho(dz)$, we obtain (setting $k:=\sqrt{h}$, to simplify notation)
\begin{align*}
\int_{|x|\geq k} |x|^2\,\mu_n(dx) \ &\leq \ 4\int\!\!\!\int_{|x+z|\geq k/\sqrt{2}} |x+z|^2\,\mu_n(dx)\Nc_\varrho(dz) + 4\int_{|z|\geq k/\sqrt{2}} |z|^2\,\Nc_\varrho(dz) \\
&= \ 4\int_{|y|\geq k/\sqrt{2}} |y|^2\,(\mu_n*\Nc_\varrho)(dy) + 4\int_{|z|\geq k/\sqrt{2}} |z|^2\,\Nc_\varrho(dz).
\end{align*}
Then, by \eqref{UnifInt2nd}, we deduce that $\{\mu_n\}_n$ has uniformly integrable second moments. This implies that $\{\mu_n\}_n$ is tight, so that we can apply \cite[Proposition 7.1.5]{AGS08}, from which we deduce the existence of a subsequence $\{\mu_{n_k}\}_k$ converging to some $\nu\in\Pc_2(\R^d)$ with respect to $\Wc_2$. Notice that (we denote by $\varphi_{\pi}$ the characteristic function of the probability measure $\pi\in\Pc(\R^d)$)
\[
\varphi_{\mu_{n_k}*\Nc_\varrho}(u) \ = \ \varphi_{\mu_{n_k}}(u)\,\textup{e}^{-\frac{1}{2}\varrho^2|u|^2} \ \overset{k\rightarrow\infty}{\longrightarrow} \ \varphi_\nu(u)\,\textup{e}^{-\frac{1}{2}\varrho^2|u|^2} \ = \ \varphi_{\nu*\Nc_\varrho}(u).
\]
Then, by L\'evy's continuity theorem it follows that $\Wc_2(\mu_{n_k}*\Nc_\varrho,\nu*\Nc_\varrho)\rightarrow0$ as $k\rightarrow\infty$. This implies that $\nu*\Nc_\varrho=\bar\nu$. By a standard argument, the entire sequence $\{\mu_n*\Wc_\sigma\}_n$ converges to $\nu*\Nc_\varrho$ with respect to $\Wc_2$. This shows that $\Wc_2^{(\varrho)}(\mu_n,\nu)\rightarrow0$ as $n\rightarrow\infty$ and concludes the proof.
\end{proof}

\noindent Our aim is to find a gauge-type function $\rho=\rho((t,\mu),(t_0,\mu_0))$ smooth with respect to $(t,\mu)$, for every fixed $(t_0,\mu_0)$, on $[0,T]\times\Pc_2(\R^d)$ endowed with the metric $((t,\mu),(s,\nu))\mapsto|t-s|+\Wc_2^{(\varrho)}(\mu,\nu)$. The construction of our smooth gauge-type function (whose definition is given in Lemma \ref{L:SmoothGauge} below) relies on a sharp upper bound of $\Wc_2$ obtained in \cite{DSS13,FG15} (see also \cite[Section 5.1.2]{CD18_I}), which is valid in any dimension $d$ and is reported in Lemma \ref{L:UpperBound}. Notice however that, in the particular case $d=1$, ad hoc gauge-type functions may be constructed in easier ways, as for instance relying on the following inequality (see \cite[Proposition 7.14]{BL19}):
\begin{equation}\label{UpperBound1D}
\Wc_2(\mu,\nu)^2 \ \leq \ 4\int_{-\infty}^{+\infty} |x|\,\big|F_\mu(x) - F_\nu(x)\big|\,dx, \qquad \forall\,\mu,\nu\in\Pc_2(\R),
\end{equation}
where $F_\mu$ and $F_\nu$ are the cumulative distribution functions of $\mu$ and $\nu$, respectively. When $d\in\N$, the upper bound of Lemma \ref{L:UpperBound} can be viewed as a $d$-dimensional analogue of \eqref{UpperBound1D}.

\begin{Lemma}\label{L:UpperBound}
For every integer $\ell\geq0$, let $\mathscr P_\ell$ denote the partition of $(-1,1]^d$ into $2^{d\ell}$ translations of $(-2^{-\ell},2^{-\ell}]^d$. Moreover, let $B_0:=(-1,1]^d$ and, for every integer $n\geq1$, $B_n:=(-2^n,2^n]^d\backslash(-2^{n-1},2^{n-1}]^d$. Then, for every $\mu,\nu\in\Pc_2(\R^d)$, the following inequality holds:
\begin{equation}\label{UpperBound}
\big(\Wc_2(\mu,\nu)\big)^2 \ \leq \ c_d \sum_{n\geq0} 2^{2n} \sum_{\ell\geq0} 2^{-2\ell} \sum_{B\in\mathscr P_\ell} \big|\mu\big((2^nB)\cap B_n\big) - \nu\big((2^nB)\cap B_n\big)\big|,
\end{equation}
where $2^nB:=\{2^nx\in\R^d\colon x\in B\}$ and $c_d>0$ is a constant depending only on $d$.
\end{Lemma}
\begin{proof}[\textbf{Proof.}]
Inequality \eqref{UpperBound} follows from \cite[Lemma 5 and Lemma 6]{FG15} (or, equivalently, \cite[Lemma 5.11 and Lemma 5.12]{CD18_I}).
\end{proof}

\noindent Next lemma provides the claimed smooth gauge-type function and it is the main result of the present section. Notice that such a gauge-type function is obtained performing a smoothing of the right-hand side of \eqref{UpperBound}, proceeding as follows.
\begin{enumerate}[a)]
\item Firstly, the absolute value of the difference $\mu((2^n B)\cap B_n)-\nu((2^n B)\cap B_n)$ appearing in \eqref{UpperBound} is replaced by $\sqrt{|\mu((2^n B)\cap B_n)-\nu((2^n B)\cap B_n)|^2+\delta_{n,\ell}^2}-\delta_{n,\ell}$, with $\delta_{n,\ell}=2^{-(4n+2d\ell)}$. In other words, we replace $|\cdot|$ by the smooth function $\sqrt{\cdot+\delta_{n,\ell}^2}-\delta_{n,\ell}$. The particular choice of $\delta_{n,\ell}$ will be used to obtain the convergence of a certain series (see \eqref{Series_delta_n,ell}).
\item Secondly, as already mentioned, our function will be of gauge-type on $[0,T]\times\Pc_2(\R^d)$, with $[0,T]\times\Pc_2(\R^d)$ endowed with the metric $((t,\mu),(s,\nu))\mapsto|t-s|+\Wc_2^{(\varrho)}(\mu,\nu)$. As a consequence, we consider \eqref{UpperBound} for $\Wc_2^{(\varrho)}(\mu,\nu)=\Wc_2(\mu*\Nc_{\varrho},\nu*\Nc_{\varrho})$. This implies that $\mu((2^n B)\cap B_n)$ and $\nu((2^n B)\cap B_n)$ are replaced respectively by $(\mu*\Nc_{\varrho})((2^n B)\cap B_n)$ and $(\nu*\Nc_{\varrho})((2^n B)\cap B_n)$.
\end{enumerate}

\begin{Lemma}\label{L:SmoothGauge}
We adopt the same notations as in Lemma \ref{L:UpperBound}. Let  $\varrho>0$ and $\rho_{2,\varrho}\colon([0,T]\times\Pc_2(\R^d))^2\rightarrow[0,+\infty)$ be defined as
\begin{align*}
&\rho_{2,\varrho}\big((t,\mu),(s,\nu)\big) \ = \ |t - s|^2 \, + \\
&+ c_d \sum_{n\geq0} 2^{2n} \sum_{\ell\geq0} 2^{-2\ell} \sum_{B\in\mathscr P_\ell} \Big(\sqrt{\big|(\mu*\Nc_{\varrho})\big((2^nB)\cap B_n\big) - (\nu*\Nc_{\varrho})\big((2^nB)\cap B_n\big)\big|^2 + \delta_{n,\ell}^2} - \delta_{n,\ell}\Big). \notag
\end{align*}
with $\delta_{n,\ell}:=2^{-(4n+2d\ell)}$. Then, the following holds.
\begin{enumerate}[\upshape 1)]
\item $\rho_{2,\varrho}$ is a gauge-type function on $[0,T]\times\Pc_2(\R^d)$, with $[0,T]\times\Pc_2(\R^d)$ endowed with the metric $((t,\mu),(s,\nu))\mapsto|t-s|+\Wc_2^{{(\varrho)}}(\mu,\nu)$;
\item for every fixed $(t_0,\mu_0)\in[0,T]\times\Pc_2(\R^d)$, the map $(t,\mu)\mapsto\rho_{2,\varrho}((t,\mu),(t_0,\mu_0))$ is in $C^{1,2}([0,T]\times\Pc_2(\R^d))$;
\item there exists a constant $C_d$ (depending only on the dimension $d$) such that
\begin{align}
\big|\partial_t\rho_{2,\varrho}\big((t,\mu),(t_0,\mu_0)\big)\big| \ &\leq \ 2\,T, \label{bound_time_derivative} \\
\big|\partial_\mu\rho_{2, \varrho}\big((t,\mu),(t_0,\mu_0)\big)(x)\big| \ &\leq \ \frac{C_d}{\varrho^2}\bigg(\int_{\R^d} |y|^3\,\zeta_{\varrho}(y)\,dy + |x|^2\int_{\R^d} |y|\,\zeta_{\varrho}(y)\,dy\bigg), \label{bounds_derivatives_1} \\
\big|\partial_x\partial_\mu\rho_{2,\varrho}\big((t,\mu),(t_0,\mu_0)\big)(x)\big| \ &\leq \ C_d\bigg(\int_{\R^d}|y|^2\,\big(\sqrt{d}\,\varrho^{-2} + |y|^2 \varrho^{-4}\big)\,\zeta_{\varrho}(y)\,dy \label{bounds_derivatives_2} \\
&\quad \ + |x|^2 \int_{\R^d}\big(\sqrt{d}\,\varrho^{-2} + |y|^2 \varrho^{-4}\big)\,\zeta_{\varrho}(y)\,dy\bigg), \notag
\end{align}
for all $(t,\mu),(t_0,\mu_0)\in[0,T]\times\Pc_2(\R^d)$, $x\in\R^d$,  where
\begin{equation}\label{zeta_varrho}
\zeta_\varrho(y) \ = \ \frac{1}{(2\pi)^{d/2}\varrho^d} \textup{e}^{-\frac{1}{2}\frac{|y|^2}{\varrho^2}}, \qquad \forall\,y\in\R^d.
\end{equation}
\end{enumerate}
\end{Lemma}
\begin{proof}[\textbf{Proof.}]
We split the proof into four steps.

\vspace{1mm}

\noindent\emph{Step I.} \emph{Uniform convergence of the series in $\rho_{2,\varrho}$.} We prove a preliminary result concerning the series in $\rho_{2,\varrho}$. Let $\mathcal M$ be a subset of $\Pc_2(\R^d)$ such that $\{\mu*\Nc_{\varrho}\}_{\mu\in\Mc}$ has uniformly integrable second moments. Our aim is to prove that the series appearing in the definition of $\rho_{2,\varrho}$ converges uniformly with respect to $\mu,\nu\in\Mc$. More precisely, we prove that for every $\eps>0$ there exists $N=N(\eps)\in\N$ such that
\begin{equation}\label{SeriesUnifConv}
\sup_{\mu,\nu\in\Mc}\sum_{n\geq N} 2^{2n} \sum_{\ell\geq0} 2^{-2\ell} \sum_{B\in\mathscr P_\ell} \big|(\mu*\Nc_{\varrho})\big((2^nB)\cap B_n\big) - (\nu*\Nc_{\varrho})\big((2^nB)\cap B_n\big)\big| \ \leq \ \eps.
\end{equation}
Then, the claim follows from the elementary inequality $\sqrt{a^2+\delta_{n,\ell}^2}-\delta_{n,\ell}\leq|a|$, valid for every $a\in\R$. Let us prove \eqref{SeriesUnifConv}. First of all, notice that
\begin{align*}
&\sum_{n\geq0} 2^{2n} \sum_{\ell\geq0} 2^{-2\ell} \sum_{B\in\mathscr P_\ell} \big|(\mu*\Nc_{\varrho})\big((2^nB)\cap B_n\big) - (\nu*\Nc_{\varrho})\big((2^nB)\cap B_n\big)\big| \\
&\leq \ \sum_{n\geq0} 2^{2n} \sum_{\ell\geq0} 2^{-2\ell} \sum_{B\in\mathscr P_\ell} (\mu*\Nc_{\varrho})\big((2^nB)\cap B_n\big) + \sum_{n\geq0} 2^{2n} \sum_{\ell\geq0} 2^{-2\ell} \sum_{B\in\mathscr P_\ell} (\nu*\Nc_{\varrho})\big((2^nB)\cap B_n\big).
\end{align*}
Observe also that $\sum_{B\in\mathscr P_\ell} (\mu*\Nc_{\varrho})((2^nB)\cap B_n)=(\mu*\Nc_\varrho)(B_n)$, therefore $\sum_{\ell\geq0} 2^{-2\ell}\sum_{B\in\mathscr P_\ell} (\mu*\Nc_\varrho)((2^nB)\cap B_n)=(4/3)(\mu*\Nc_\varrho)(B_n)$, since $\sum_{\ell\geq0} 2^{-2\ell}=4/3$. So, in particular, \eqref{SeriesUnifConv} follows if we prove that for every $\eps>0$ there exists $N=N(\eps)\in\N$ such that
\[
\sup_{\mu\in\Mc}\sum_{n\geq N} 2^{2n} (\mu*\Nc_{\varrho})(B_n) \ \leq \ \eps.
\]
Recalling that $B_n=(-2^n,2^n]^d\backslash(-2^{n-1},2^{n-1}]^d$, we obtain $2^{2n}\leq4|x|^2/d$, $\forall\,x\in B_n$. Hence, for every $N\in\N$,
\[
\sum_{n\geq N} 2^{2n} (\mu*\Nc_{\varrho})(B_n) \ \leq \ \frac{4}{d} \int_{\R^d\backslash(-2^{N-1},2^{N-1}]^d} |x|^2\,(\mu*\Nc_{\varrho})(dx).
\]
Since the family $\{\mu*\Nc_{\varrho}\}_{\mu\in\Mc}$ has uniformly integrable second moments, the claim follows.

\vspace{1mm}

\noindent\emph{Step II. $\rho_{2,\varrho}$ is a gauge-type function on $[0,T]\times\Pc_2(\R^d)$ with respect to the metric $((t,\mu),(s,\nu))\mapsto|t-s|+\Wc_2^{{(\varrho)}}(\mu,\nu)$.} It is clear that $\rho_{2, \varrho}$ satisfies item a) of Definition \ref{D:Gauge}. Concerning items b) and c), we split the rest of the proof of Step II into two substeps.

\vspace{1mm}

\noindent\emph{$\rho_{2,\varrho}$ satisfies item b) of Definition \ref{D:Gauge}.} Our aim is to prove that, given $\{(t_k,\mu_k)\}_k,\{(s_k,\nu_k)\}_k\subset[0,T]\times\Pc_2(\R^d)$ and $(t,\mu),(s,\nu)\in[0,T]\times\Pc_2(\R^d)$, if $|t_k-t|+\Wc_2^{{(\varrho)}}(\mu_k,\mu)+|s_k-s|+\Wc_2^{{(\varrho)}}(\nu_k,\nu)\rightarrow0$ then 
$\rho_{2, \varrho}((t_k,\mu_k),(s_k,\nu_k))\rightarrow\rho_{2,\varrho}((t,\mu),(s,\nu))$. In particular, we have to prove that, if $\Wc_2^{{(\varrho)}}(\mu_k,\mu)+\Wc_2^{{(\varrho)}}(\nu_k,\nu)\rightarrow0$, then
\begin{align*}
&\sum_{n\geq0} 2^{2n} \sum_{\ell\geq0} 2^{-2\ell} \sum_{B\in\mathscr P_\ell} \Big(\sqrt{\big|\big((\mu_k*\Nc_{\varrho}) - (\nu_k*\Nc_{\varrho})\big)\big((2^nB)\cap B_n\big)\big|^2 + \delta_{n,\ell}^2} - \delta_{n,\ell}\Big) \\
&\overset{k\rightarrow\infty}{\longrightarrow} \ \sum_{n\geq0} 2^{2n} \sum_{\ell\geq0} 2^{-2\ell} \sum_{B\in\mathscr P_\ell} \Big(\sqrt{\big|\big((\mu*\Nc_{\varrho}) - (\nu*\Nc_{\varrho})\big)\big((2^nB)\cap B_n\big)\big|^2 + \delta_{n,\ell}^2} - \delta_{n,\ell}\Big). \notag
\end{align*}
Since $\Wc_2^{{(\varrho)}}(\mu_k,\mu)=\Wc_2(\mu_k*\Nc_{\varrho},\mu*\Nc_{\varrho})$ and $\Wc_2^{{(\varrho)}}(\nu_k,\nu)=\Wc_2(\nu_k*\Nc_{\varrho},\nu*\Nc_{\varrho})$, we have that $\Wc_2(\mu_k*\Nc_{\varrho},\mu*\Nc_{\varrho})+\Wc_2(\nu_k*\Nc_{\varrho},\nu*\Nc_{\varrho})\rightarrow0$. Now, recall from \cite[Proposition 7.1.5]{AGS08} that this implies that $\{\mu_k*\Nc_{\varrho}\}_k$ (resp. $\{\nu_k*\Nc_{\varrho}\}_k$) weakly converges to $\mu*\Nc_{\varrho}$ (resp. $\nu*\Nc_{\varrho}$) and has uniformly integrable second moments. Since both $\mu*\Nc_{\varrho}$ and $\nu*\Nc_{\varrho}$ are absolutely continuous with respect to the Lebesgue measure on $\R^d$, by the weak convergence (and, in particular, by the portmanteau theorem) we deduce that
\[
\lim_{k\rightarrow\infty} (\mu_k*\Nc_{\varrho})\big((2^nB)\cap B_n\big) \ = \ (\mu*\Nc_{\varrho})\big((2^nB)\cap B_n\big).
\]
Similarly $(\nu_k*\Nc_{\varrho})((2^nB)\cap B_n)\rightarrow(\nu*\Nc_{\varrho})((2^nB)\cap B_n)$. In addition, since $\{\mu_k*\Nc_{\varrho}\}_k$ and $\{\nu_k*\Nc_{\varrho}\}_k$ have uniformly integrable second moments, from Step I we can interchange the limit with the series, so that the claim follows.

\vspace{1mm}

\noindent\emph{$\rho_{2,\varrho}$ satisfies item c) of Definition \ref{D:Gauge}.} Our aim is to prove the following: for every $\eps>0$, there exists $\eta_\eps>0$ such that, for all $(t,\mu),(s,\nu)\in[0,T]\times\Pc_2(\R^d)$, the inequality $\rho_{2,\varrho}((t,\mu),(s,\nu))\leq\eta_\eps$ implies
\begin{equation}\label{item_c}
|t - s|^2 + c_d \sum_{n\geq0} 2^{2n} \sum_{\ell\geq0} 2^{-2\ell} \sum_{B\in\mathscr P_\ell} \big|(\mu*\Nc_{\varrho})\big((2^nB)\cap B_n\big) - (\nu*\Nc_{\varrho})\big((2^nB)\cap B_n\big)\big| \ \leq \ \frac{\eps^2}{2}.
\end{equation}
As a matter of fact, recalling that $\Wc_2^{{(\varrho)}}(\mu,\nu)=\Wc_2(\mu*\Nc_{\varrho},\nu*\Nc_{\varrho})$, from inequality \eqref{UpperBound} we conclude that
\[
|t-s| + \Wc_2^{{(\varrho)}}(\mu,\nu) \ \leq \ \sqrt{2\,|t-s|^2 + 2\,\Wc_2^{{(\varrho)}}(\mu,\nu)^2} \ \leq \ \eps.
\]
Let us prove that \eqref{item_c} holds with $\eta_\eps$ given by
\begin{equation}\label{eta_eps}
\eta_\eps \ := \ \Big(\sqrt{8c_d+\eps^2/2} - \sqrt{8c_d}\Big)^2.
\end{equation}
To this end, denote by $d_{2,\varrho}((t,\mu),(s,\nu))$ the left-hand side of \eqref{item_c}, namely
\[
d_{2,\varrho}\big((t,\mu),(s,\nu)\big) \ := \ |t - s|^2 + c_d \sum_{n\geq0} 2^{2n} \sum_{\ell\geq0} 2^{-2\ell} \sum_{B\in\mathscr P_\ell} \big|\big((\mu*\Nc_{\varrho}) - (\nu*\Nc_{\varrho})\big)\big((2^nB)\cap B_n\big)\big|.
\]
Moreover, for every $n\geq0$, $\ell\geq0$, $B\in\mathscr P_\ell$, $\mu,\nu\in\Pc_2(\R^d)$, denote
\begin{align*}
a_{n,\ell}(\mu,\nu,B) \ &:= \ \Big(\sqrt{\big|\big((\mu*\Nc_{\varrho}) - (\nu*\Nc_{\varrho})\big)\big((2^nB)\cap B_n\big)\big|^2 + \delta_{n,\ell}^2} - \delta_{n,\ell}\Big), \\
b_{n,\ell}(\mu,\nu,B) \ &:= \ \big|\big((\mu*\Nc_{\varrho}) - (\nu*\Nc_{\varrho})\big)\big((2^nB)\cap B_n\big)\big|.
\end{align*}
Notice that
\begin{align}
\rho_{2,\varrho}\big((t,\mu),(s,\nu)\big) \ &= \ |t - s|^2 + c_d\sum_{n\geq0}2^{2n}\sum_{\ell\geq0} 2^{-2\ell} \sum_{B\in\mathscr P_\ell} a_{n,\ell}(\mu,\nu,B), \label{d_2} \\
d_{2,\varrho}\big((t,\mu),(s,\nu)\big) \ &= \ |t - s|^2 + c_d\sum_{n\geq0}2^{2n}\sum_{\ell\geq0} 2^{-2\ell} \sum_{B\in\mathscr P_\ell} b_{n,\ell}(\mu,\nu,B). \label{rho_2}
\end{align}
Now, consider $(t,\mu),(s,\nu)\in[0,T]\times\Pc_2(\R^d)$ such that $\rho_{2,\varrho}((t,\mu),(s,\nu))\leq\eta_\eps$, with $\eta_\eps$ given by \eqref{eta_eps}. Then
\begin{equation}\label{a_n_eta}
c_d\,2^{2(n-\ell)}\,a_{n,\ell}(\mu,\nu,B) \ \leq \ \rho_{2,\varrho}((t,\mu),(s,\nu)) \ \leq \ \eta_\eps.
\end{equation}
Since $a_{n,\ell}(\mu,\nu,B)=\sqrt{|b_{n,\ell}(\mu,\nu,B)|^2+\delta_{n,\ell}^2}-\delta_{n,\ell}$, we obtain
\[
b_{n,\ell}(\mu,\nu,B) \ = \ \sqrt{|a_{n,\ell}(\mu,\nu,B)|^2 + 2\,\delta_{n,\ell}\,a_{n,\ell}(\mu,\nu,B)} \ \leq \ a_{n,\ell}(\mu,\nu,B) + \sqrt{2\,\delta_{n,\ell}\,a_{n,\ell}(\mu,\nu,B)},
\]
where we have used the elementary inequality $\sqrt{x+y}\leq\sqrt{x}+\sqrt{y}$, valid for every $x,y\geq0$. Therefore, by \eqref{a_n_eta} we get
\[
b_{n,\ell}(\mu,\nu,B) \ \leq \ a_{n,\ell}(\mu,\nu,B) + \sqrt{\frac{2}{c_d}\eta_\eps}\,\sqrt{\delta_{n,\ell}\,2^{-2(n-\ell)}} \ = \ a_{n,\ell}(\mu,\nu,B) + \sqrt{\frac{2}{c_d}\eta_\eps}\,2^{-(3n+(d-1)\ell)}
\]
where the last equality follows from the fact that $\delta_{n,\ell}=2^{-(4n+2d\ell)}$. Hence, from \eqref{d_2} and \eqref{rho_2} we obtain
\[
d_{2,\varrho}\big((t,\mu),(s,\nu)\big) \ \leq \ \rho_{2,\varrho}\big((t,\mu),(s,\nu)\big) + c_d\sqrt{\frac{2}{c_d}\eta_\eps}\sum_{n\geq0}2^{2n}\sum_{\ell\geq0} 2^{-2\ell} \sum_{B\in\mathscr P_\ell} 2^{-(3n+(d-1)\ell)}
\]
Recalling that $\rho_{2,\varrho}((t,\mu),(s,\nu))\leq\eta_\eps$ and also that $\mathscr P_\ell$ contains $2^{d\ell}$ sets (see the statement of Lemma \ref{L:UpperBound}), we get
\begin{align}\label{Series_delta_n,ell}
d_{2,\varrho}\big((t,\mu),(s,\nu)\big) \ &\leq \ \eta_\eps + \sqrt{2\,c_d\,\eta_\eps}\sum_{n\geq0}\sum_{\ell\geq0} 2^{2n}\,2^{-2\ell}\,2^{d\ell}\,2^{-(3n+(d-1)\ell)} \\
&= \ \eta_\eps + \sqrt{2\,c_d\,\eta_\eps}\sum_{n\geq0}\sum_{\ell\geq0} 2^{-n}\,2^{-\ell} \ = \ \eta_\eps + 4\sqrt{2\,c_d\,\eta_\eps} \ = \ \frac{\eps^2}{2}, \notag
\end{align}
where the last equality follows from the definition of $\eta_\eps$.

\vspace{1mm}

\noindent\emph{Step III. The map $(t,\mu)\mapsto\rho_{2, \varrho}((t,\mu),(t_0,\mu_0))$ is in $C^{1,2}([0,T]\times\Pc_2(\R^d))$.} Recall from \eqref{zeta_varrho} that $\zeta_{\varrho}\colon\R^d\rightarrow\R$ denotes the density function of the multivariate normal distribution $\Nc_{\varrho}=\Nc(0,\varrho^2 I_d)$. Then, the map 
$\rho_{2,\varrho}$ can be written as
\begin{align*}
&\rho_{2,\varrho}\big((t,\mu),(t_0,\mu_0)\big) \ = \ |t - t_0|^2 \, + \\
&\hspace{1cm}+ \, c_d \sum_{n\geq0} 2^{2n} \sum_{\ell\geq0} 2^{-2\ell} \sum_{B\in\mathscr P_\ell} \bigg(\sqrt{\bigg|\int_{\R^d} \phi_n^B(y)\,\mu(dy) - \int_{\R^d} \phi_n^B(y)\,\mu_0(dy)\bigg|^2 + \delta_{n,\ell}^2} - \delta_{n,\ell}\bigg),
\end{align*}
where
\[
\phi_n^B(x) \ := \ \int_{(2^n B)\cap B_n} \zeta_{\varrho}(z-x)\,dz, \qquad \forall\,x\in\R^d.
\]
We split the rest of the proof of Step III into two substeps.

\vspace{1mm}

\noindent\emph{First-order derivatives.} By direct calculation, we have $\partial_t \rho_{2,\varrho}((t,\mu),(t_0,\mu_0))=2(t - t_0)$. Moreover, we claim that $\partial_\mu \rho_{2,\varrho}((t,\mu),(t_0,\mu_0))(x)$ is given by
\begin{align}\label{partial_mu_rho_2}
&c_d \sum_{n\geq0} 2^{2n} \sum_{\ell\geq0} 2^{-2\ell} \sum_{B\in\mathscr P_\ell} \frac{\int_{\R^d} \phi_n^B(y)\,(\mu-\mu_0)(dy)}{\sqrt{\big|\int_{\R^d} \phi_n^B(y)\,(\mu-\mu_0)(dy)\big|^2 + \delta_{n,\ell}^2}} \partial_x\phi_n^B(x) \notag \\
&= \ c_d \sum_{n\geq0} 2^{2n} \sum_{\ell\geq0} 2^{-2\ell} \sum_{B\in\mathscr P_\ell} \frac{(\mu*\Nc_{\varrho}-\mu_0*\Nc_{\varrho})((2^nB)\cap B_n)}{\sqrt{\big|(\mu*\Nc_{\varrho}-\mu_0*\Nc_{\varrho})((2^nB)\cap B_n)\big|^2 + \delta_{n,\ell}^2}} \partial_x\phi_n^B(x),
\end{align}
where $\partial_x\phi_n^B$ denotes the gradient of $\phi_n^B$. In order to prove \eqref{partial_mu_rho_2}, we denote, for every $n,\ell\geq0$, $\mu_0\in\Pc_2(\R^d)$, $B\in\mathscr P_\ell$,
\[
u_{n,\ell}^{B,\mu_0}(\mu) \ = \ \sqrt{\bigg|\int_{\R^d} \phi_n^B(y)\,\mu(dy) - \int_{\R^d} \phi_n^B(y)\,\mu_0(dy)\bigg|^2 + \delta_{n,\ell}^2} - \delta_{n,\ell}, \qquad \forall\,\mu\in\Pc_2(\R^d).
\]
Let us determine $\partial_\mu u_{n,\ell}^{B,\mu_0}$. To this end, let us consider the lifting $U_{n,\ell}^{B,\mu_0}\colon L^2(\Omega;\R^d)\rightarrow\R$ of $u_{n,\ell}^{B,\mu_0}$, given by $U_{n,\ell}^{B,\mu_0}(\xi)=u_{n,\ell}^{B,\mu_0}(\mu)$, for every $\xi\in L^2(\Omega;\R^d)$ having distribution $\mu$. Recall from the definition of $\partial_\mu u_{n,\ell}^{B,\mu_0}$ that, for every $\{\eta_k\}_k\subset L^2(\Omega;\R^d)$ such that $|\eta_k|_{L^2(\Omega;\R^d)}\rightarrow0$, it holds that
\begin{equation}\label{Limit_partial_mu}
\lim_{k\rightarrow\infty} \frac{\big|U_{n,\ell}^{B,\mu_0}(\xi+\eta_k) - U_{n,\ell}^{B,\mu_0}(\xi) - \E\big[\big\langle\partial_\mu u_{n,\ell}^{B,\mu_0}(\mu)(\xi),\eta_k\big\rangle\big]\big|}{|\eta_k|_{L^2(\Omega;\R^d)}} \ = \ 0,
\end{equation}
where $\xi\in L^2(\Omega;\R^d)$ has distribution $\mu$. Then, we have
\begin{equation}\label{partial_mu_u_n}
\partial_\mu u_{n,\ell}^{B,\mu_0}(\mu)(x) \ = \ \frac{\int_{\R^d} \phi_n^B(y)\,(\mu-\mu_0)(dy)}{\sqrt{\big|\int_{\R^d} \phi_n^B(y)\,(\mu-\mu_0)(dy)\big|^2 + \delta_{n,\ell}^2}} \partial_x\phi_n^B(x),
\end{equation}
for every $(\mu,x)\in\Pc_2(\R^d)\times\R^d$. Now, by \eqref{Limit_partial_mu} we see that \eqref{partial_mu_rho_2} follows if we prove that the series
\begin{equation}\label{series_partial_mu}
c_d \sum_{n\geq0} 2^{2n} \sum_{\ell\geq0} 2^{-2\ell} \sum_{B\in\mathscr P_\ell} \frac{\big|U_{n,\ell}^{B,\mu_0}(\xi+\eta_k) - U_{n,\ell}^{B,\mu_0}(\xi) - \E\big[\big\langle\partial_\mu u_{n,\ell}^{B,\mu_0}(\mu)(\xi),\eta_k\big\rangle\big]\big|}{|\eta_k|_{L^2(\Omega;\R^d)}}
\end{equation}
converges uniformly with respect to $k$. To this end, denote
\[
h(\lambda) \ := \ U_{n,\ell}^{B,\mu_0}(\xi+\lambda\eta_k), \qquad 0\leq\lambda\leq1.
\]
Since $h(1)=h(0)+\int_0^1 h'(\lambda)d\lambda$, we get
\[
U_{n,\ell}^{B,\mu_0}(\xi+\eta) \ = \ U_{n,\ell}^{B,\mu_0}(\xi) + \int_0^1 \E\big[\big\langle\partial_\mu u_{n,\ell}^{B,\mu_0}(\mu_{k,\lambda})(\xi+\lambda\eta_k),\eta_k\big\rangle\big]d\lambda,
\]
where $\mu_{k,\lambda}$ is the distribution of $\xi+\lambda\eta_k$. Then, \eqref{series_partial_mu} is bounded from above by
\begin{align}\label{series_partial_mu_2}
c_d \sum_{n\geq0} 2^{2n} \sum_{\ell\geq0} 2^{-2\ell} \sum_{B\in\mathscr P_\ell} \int_0^1 \big|\E\big[\big\langle\partial_\mu u_{n,\ell}^{B,\mu_0}(\mu_{k,\lambda})(\xi+\lambda\eta_k),\eta_k/|\eta_k|_{L^2(\Omega;\R^d)}\big\rangle\big]\big|d\lambda \\
+ \, c_d \sum_{n\geq0} 2^{2n} \sum_{\ell\geq0} 2^{-2\ell} \sum_{B\in\mathscr P_\ell} \big|\E\big[\big\langle\partial_\mu u_{n,\ell}^{B,\mu_0}(\mu)(\xi),\eta_k/|\eta_k|_{L^2(\Omega;\R^d)}\big\rangle\big]\big|. \notag
\end{align}
Notice that $\{\eta_k\}_{k\in\N}$ has uniformly integrable second moments (see for instance \cite[Theorem 4.12]{Kallenberg}), so that $\{\xi+\lambda\eta_k\}_{k\in\N,\lambda\in[0,1]}$ also has uniformly integrable second moments. Therefore, the two series in \eqref{series_partial_mu_2} converge uniformly if we prove that ($\nu$ denotes the distribution of $\eta$)
\[
c_d \sum_{n\geq0} 2^{2n} \sum_{\ell\geq0} 2^{-2\ell} \sum_{B\in\mathscr P_\ell} \E\big[\big|\partial_\mu u_{n,\ell}^{B,\mu_0}(\nu)(\eta)\big|\big]
\]
converges uniformly with respect to $\nu$, whenever $\nu$ belongs to a subset $\Mc$ of $\Pc_2(\R^d)$ with uniformly integrable second moments, namely
\begin{equation}\label{eta_UnifIntegr}
\lim_{M\rightarrow\infty}\sup_{\nu\in\Mc} \int_{|x|\geq M}|x|^2\,\nu(dx) \ = \ 0.
\end{equation}
Then, the claim follows if we prove that for every $\eps>0$, there exists $N=N(\eps,\Mc)\in\N$ such that, for every $\nu\in\Mc$, it holds that
\begin{equation}\label{series_partial_mu_claim}
c_d \sum_{n\geq N} 2^{2n} \sum_{\ell\geq0} 2^{-2\ell} \sum_{B\in\mathscr P_\ell} \E\big[\big|\partial_\mu u_{n,\ell}^{B,\mu_0}(\nu)(\eta)\big|\big] \ \leq \ \eps,
\end{equation}
with $\eta\in L^2(\Omega;\R^d)$ having distribution $\nu$. Firstly, from \eqref{partial_mu_u_n} notice that $|\partial_\mu u_{n,\ell}^{B,\mu_0}(\mu)(x)|\leq|\partial_x\phi_n^B(x)|$, $\forall\,x\in\R^d$. Moreover $\partial_x\phi_n^B(x)=\frac{1}{\varrho^2}\int_{(2^n B)\cap B_n}(z-x)\,\zeta_{\varrho}(z-x)dz$. Therefore, the series \eqref{series_partial_mu_claim} is bounded from above by
\begin{align*}
&\frac{c_d}{\varrho^2} \sum_{n\geq N} 2^{2n} \sum_{\ell\geq0} 2^{-2\ell} \sum_{B\in\mathscr P_\ell} \E\bigg[\int_{(2^n B)\cap B_n}|z-\eta|\,\zeta_{\varrho}(z-\eta)\,dz\bigg] \\
&= \ \frac{c_d}{\varrho^2} \sum_{n\geq N} 2^{2n} \sum_{\ell\geq0} 2^{-2\ell}\,\E\bigg[\int_{B_n}|z-\eta|\,\zeta_{\varrho}(z-\eta)\,dz\bigg] \ = \ \frac{4}{3}\,\frac{c_d}{\varrho^2} \sum_{n\geq N} 2^{2n}\,\E\bigg[\int_{B_n}|z-\eta|\,\zeta_{\varrho}(z-\eta)\,dz\bigg].
\end{align*}
Recalling that $B_n=(-2^n,2^n]^d\backslash(-2^{n-1},2^{n-1}]^d$, we obtain $2^{2n}\leq4|z|^2/d$, $\forall\,z\in B_n$, so that
\begin{align*}
&\frac{4}{3}\,\frac{c_d}{\varrho^2} \sum_{n\geq N} 2^{2n}\,\E\bigg[\int_{B_n}|z-\eta|\,\zeta_{\varrho}(z-\eta)\,dz\bigg] \leq \frac{16}{3d}\,\frac{c_d}{\varrho^2}\,\E\bigg[\int_{\R^d\backslash(-2^{N-1},2^{N-1}]^d}|z|^2\,|z-\eta|\,\zeta_{\varrho}(z-\eta)\,dz\bigg] \\
&\leq \frac{16}{3d}\,\frac{c_d}{\varrho^2}\,\E\bigg[\int_{|z|\geq2^{N-1}}|z|^2\,|z-\eta|\,\zeta_{\varrho}(z-\eta)\,dz\bigg] \\
&= \ \frac{16}{3d}\,\frac{c_d}{\varrho^2}\,\int_{\R^d}\bigg(\int_{\R^d}1_{\{|z|\geq2^{N-1}\}}|z|^2\,|z-x|\,\zeta_{\varrho}(z-x)\,dz\bigg)\nu(dx) \\
&= \ \frac{16}{3d}\,\frac{c_d}{\varrho^2}\,\int\!\!\!\int_{\R^d\times\R^d}1_{\{|y+x|\geq2^{N-1}\}}|y+x|^2\,|y|\,\zeta_{\varrho}(y)\,dy\,\nu(dx).
\end{align*}
Applying the elementary inequality \eqref{Ineq_Unif_Int} (with $x,x+z,z,\sqrt{h}$ replaced respectively by $y+x,y,x,2^{N-1}$), we obtain
\begin{align*}
&\frac{16}{3d}\,\frac{c_d}{\varrho^2}\,\int\!\!\!\int_{\R^d\times\R^d}1_{\{|y+x|\geq2^{N-1}\}}|y+x|^2\,|y|\,\zeta_{\varrho}(y)\,dy\,\nu(dx) \\
&\leq \ \frac{64}{3d}\,\frac{c_d}{\varrho^2}\int\!\!\!\int_{\R^d\times\R^d}\Big(1_{\{|y|\geq2^{N-3/2}\}}|y|^2 + 1_{\{|x|\geq2^{N-3/2}\}}|x|^2\Big)|y|\,\zeta_{\varrho}(y)\,dy\,\nu(dx) \\
&= \ \frac{64}{3d}\,\frac{c_d}{\varrho^2}\,\int_{|y|\geq2^{N-3/2}}|y|^3\,\zeta_{\varrho}(y)\,dy + \frac{64}{3d}\,\frac{c_d}{\varrho^2}\bigg(\int_{\R^d}|y|\,\zeta_{\varrho}(y)\,dy\bigg)\bigg(\int_{|x|\geq2^{N-3/2}}|x|^2\,\nu(dx)\bigg).
\end{align*}
Then, \eqref{series_partial_mu_claim} follows from \eqref{eta_UnifIntegr}.

\vspace{1mm}

\noindent\emph{Second-order derivatives.} We claim that $\partial_x\partial_\mu \rho_{2,\varrho}((t,\mu),(t_0,\mu_0))(x)$ is equal to
\begin{align}\label{partial_xmu_rho_2}
&c_d \sum_{n\geq0} 2^{2n} \sum_{\ell\geq0} 2^{-2\ell} \sum_{B\in\mathscr P_\ell} \frac{\int_{\R^d} \phi_n^B(y)\,(\mu-\mu_0)(dy)}{\sqrt{\big|\int_{\R^d} \phi_n^B(y)\,(\mu-\mu_0)(dy)\big|^2 + \delta_{n,\ell}^2}} \partial_{xx}^2\phi_n^B(x) \notag \\
&= \ c_d \sum_{n\geq0} 2^{2n} \sum_{\ell\geq0} 2^{-2\ell} \sum_{B\in\mathscr P_\ell} \frac{(\mu*\Nc_{\varrho}-\mu_0*\Nc_{\varrho})((2^nB)\cap B_n)}{\sqrt{\big|(\mu*\Nc_{\varrho}-\mu_0*\Nc_{\varrho})((2^nB)\cap B_n)\big|^2 + \delta_{n,\ell}^2}} \partial_{xx}^2\phi_n^B(x),
\end{align}
where $\partial_{xx}^2\phi_n^B$ denotes the Hessian matrix of $\phi_n^B$. Proceeding as in the previous substep, we see that this follows if we prove that the series ($|\partial_{xx}^2\phi_n^B(x)|$ stands for the Frobenius norm of the $d\times d$ matrix $\partial_{xx}^2\phi_n^B(x)$)
\[
c_d \sum_{n\geq0} 2^{2n} \sum_{\ell\geq0} 2^{-2\ell} \sum_{B\in\mathscr P_\ell} \frac{\big|\int_{\R^d} \phi_n^B(y)\,(\mu-\mu_0)(dy)\big|}{\sqrt{\big|\int_{\R^d} \phi_n^B(y)\,(\mu-\mu_0)(dy)\big|^2 + \delta_{n,\ell}^2}}|\partial_{xx}^2\phi_n^B(x)|
\]
converges uniformly with respect to $x$, whenever $x$ belongs to a bounded subset of $\R^d$. More precisely, we prove that for all $\eps>0$ and $M\in\N$, there exists $N=N(\eps,M)\in\N$ such that, for every $x\in\R^d$, with $|x|\leq M$, it holds that
\begin{equation}\label{series_partial_xmu_claim}
c_d \sum_{n\geq N} 2^{2n} \sum_{\ell\geq0} 2^{-2\ell} \sum_{B\in\mathscr P_\ell} \frac{\big|\int_{\R^d} \phi_n^B(y)\,(\mu-\mu_0)(dy)\big|}{\sqrt{\big|\int_{\R^d} \phi_n^B(y)\,(\mu-\mu_0)(dy)\big|^2 + \delta_{n,\ell}^2}}|\partial_{xx}^2\phi_n^B(x)| \ \leq \ \eps.
\end{equation}
We begin noting that the latter series is bounded from above by
\begin{equation}\label{series_partial_xmu}
c_d \sum_{n\geq N} 2^{2n} \sum_{\ell\geq0} 2^{-2\ell} \sum_{B\in\mathscr P_\ell} |\partial_{xx}^2\phi_n^B(x)|.
\end{equation}
We also observe that
\begin{equation}\label{2nd_deriv_phi}
\partial_{xx}^2\phi_n^B(x) \ = \  \frac{1}{\varrho^2}I_d\int_{(2^nB)\cap B_n}\zeta_{\varrho}(z-x)\,dz -  \frac{1}{\varrho^4}\int_{(2^nB)\cap B_n}(z-x)\otimes(z-x)\,\zeta_{\varrho}(z-x)\,dz,
\end{equation}
where $I_d$ denotes the identity matrix of order $d$, while $(z-x)\otimes(z-x)$ is the $d\times d$ matrix with $(i,j)$-component equal to $(z_i-x_i)(z_j-x_j)$. Then, \eqref{series_partial_xmu} is bounded from above by (notice that the Frobenius norms $|I_d|$ and $|(z-x)\otimes(z-x)|$ are given respectively by $\sqrt{d}$ and $|z-x|^2$, where $|z-x|$ denotes the Euclidean norm of $z-x$)
\begin{align*}
&c_d\sum_{n\geq N} 2^{2n} \sum_{\ell\geq0} 2^{-2\ell} \sum_{B\in\mathscr P_\ell} \int_{(2^nB)\cap B_n}\big(\sqrt{d}\,\varrho^{-2} + |z-x|^2 \varrho^{-4}\big)\,\zeta(z-x)\,dz \\
&= \ c_d \sum_{n\geq N} 2^{2n} \sum_{\ell\geq0} 2^{-2\ell} \int_{B_n}\big(\sqrt{d}\,\varrho^{-2} + |z-x|^2 \varrho^{-4}\big)\,\zeta(z-x)\,dz \\
&= \ \frac{4}{3} c_d \sum_{n\geq N} 2^{2n} \int_{B_n}\big(\sqrt{d}\, \varrho^{-2} + |z-x|^2 \varrho^{-4}\big)\,\zeta(z-x)\,dz.
\end{align*}
Recalling that $2^{2n}\leq4|z|^2/d$, $\forall\,z\in B_n$, we find
\begin{align}\label{series_partial_xmu_2}
&\frac{4}{3} c_d \sum_{n\geq N} 2^{2n} \int_{B_n}\big(\sqrt{d}\, \varrho^{-2} + |z-x|^2 \varrho^{-4}\big)\,\zeta(z-x)\,dz \notag \\
&\leq \ \frac{16}{3d} c_d \int_{\R^d\backslash(-2^{N-1},2^{N-1}]^d}|z|^2\,\big(\sqrt{d}\, \varrho^{-2} + |z-x|^2  \varrho^{-4}\big)\,\zeta(z-x)\,dz.
\end{align}
Since $|x|\leq M$, from the right-hand side of \eqref{series_partial_xmu_2} we see that \eqref{series_partial_xmu_claim} follows.

\vspace{1mm}

\noindent\emph{Step IV. Bounds.} The bound \eqref{bound_time_derivative} for the time derivative follows directly from the definition of $\rho_{2,\varrho}$. Let us now investigate the derivatives with respect to the measure. Recalling that $\partial_\mu\rho_{2,\varrho}((t,\mu),(t_0,\mu_0))(x)$ is given by \eqref{partial_mu_rho_2}, we obtain (notice that $\partial_x\phi_n^B(x)= \frac{1}{\varrho^2}\int_{(2^nB)\cap B_n}(z-x)\zeta_{\varrho}(z-x)dz$)
\begin{align*}
|\partial_\mu\rho_{2, \varrho}((t,\mu),(t_0,\mu_0))(x)| \ &\leq \ \frac{c_d}{ \varrho^2} \sum_{n\geq0} 2^{2n} \sum_{\ell\geq0} 2^{-2\ell} \sum_{B\in\mathscr P_\ell} \int_{(2^n B)\cap B_n}|z-x|\,\zeta_{\varrho}(z-x)\,dz \\
&= \frac{4}{3}\frac{c_d}{ \varrho^2} \sum_{n\geq0} 2^{2n} \int_{B_n}|z-x|\,\zeta_{ \varrho}(z-x)\,dz.
\end{align*}
Since $2^{2n}\leq4|z|^2/d$, $\forall\,z\in B_n$, we get
\begin{align*}
|\partial_\mu\rho_{2, \varrho}((t,\mu),(t_0,\mu_0))(x)| \ &\leq \ \frac{16}{3d}\frac{c_d}{ \varrho^2} \int_{\R^d} |z|^2\,|z-x|\,\zeta_{ \varrho}(z-x)\,dz \ = \ \frac{16}{3d}\frac{c_d}{ \varrho^2} \int_{\R^d} |y+x|^2\,|y|\,\zeta_{ \varrho}(y)\,dy \\
&\leq \ \frac{32}{3d}\frac{c_d}{ \varrho^2} \int_{\R^d} |y|^3\,\zeta_{ \varrho}(y)\,dy + \frac{32}{3d}\frac{c_d}{ \varrho^2}\,|x|^2\int_{\R^d} |y|\,\zeta_{ \varrho}(y)\,dy,
\end{align*}
 which gives \eqref{bounds_derivatives_1}.\\
From similar calculations, by \eqref{partial_xmu_rho_2} and \eqref{2nd_deriv_phi}, we deduce that $\partial_x\partial_\mu\rho_{2,\varrho}((t,\mu),(t_0,\mu_0))(x)$ is bounded by
\begin{align*}
&|\partial_x\partial_\mu\rho_{2,\varrho}((t,\mu),(t_0,\mu_0))(x)| \ \leq \ \frac{16}{3d}c_d \int_{\R^d}|z|^2\,\big(\sqrt{d}\,\textcolor{black}{\varrho^{-2}} + |z-x|^2\textcolor{black}{\varrho^{-4}}\big)\,\zeta_{\varrho}(z-x)\,dz \\
&= \ \frac{16}{3d}c_d \int_{\R^d}|y+x|^2\,\big(\sqrt{d}\,\textcolor{black}{\varrho^{-2}} + |y|^2\textcolor{black}{\varrho^{-4}}\big)\,\zeta_{\varrho}(y)\,dy \\
&\leq \ \frac{32}{3d}c_d \int_{\R^d}|y|^2\,\big(\sqrt{d}\,\textcolor{black}{\varrho^{-2}} + |y|^2\textcolor{black}{\varrho^{-4}}\big)\,\zeta_{\varrho}(y)\,dy + \frac{32}{3d}c_d\,|x|^2 \int_{\R^d}\big(\sqrt{d}\,\textcolor{black}{\varrho^{-2}} + |y|^2\textcolor{black}{\varrho^{-4}}\big)\,\zeta_{\varrho}(y)\,dy.
\end{align*}
We conclude that \eqref{bounds_derivatives_2} holds.
\end{proof}

\noindent We are in a position to state the smooth variational principle on $[0,T]\times\Pc_2(\R^d)$.

\begin{Theorem}\label{T:SmoothVarPrinc}
Fix \textcolor{black}{$\delta>0$} and let $G\colon[0,T]\times\Pc_2(\R^d)\rightarrow\R$ be upper semicontinuous and bounded from above. Given $\lambda>0$, let $(t_0,\mu_0)\in[0,T]\times\Pc_2(\R^d)$ be such that
\[
\sup G - \lambda \ \leq \ G(t_0,\mu_0).
\]
Then, there exist $(\tilde t,\tilde\mu)\in[0,T]\times\Pc_2(\R^d)$ and a sequence $\{(t_k,\mu_k)\}_{k\geq1}\subset[0,T]\times\Pc_2(\R^d)$ such that:
\begin{enumerate}[\upshape(i)]
\item $\rho_{2,\textcolor{black}{1/\delta}}((\tilde t,\tilde\mu),(t_k,\mu_k))\leq\frac{\lambda}{2^k\textcolor{black}{\delta^2}}$, for every $k\geq0$;
\item $G(t_0,\mu_0)\leq G(\tilde t,\tilde\mu)-\textcolor{black}{\delta^2}\varphi_{\textcolor{black}{\delta}}(\tilde t,\tilde\mu)$, with $\varphi_{\textcolor{black}{\delta}}\colon[0,T]\times\Pc_2(\R^d)\rightarrow[0,+\infty)$ given by
\[
\varphi_{\textcolor{black}{\delta}}(t,\mu) \ = \ \sum_{k=0}^{+\infty} \frac{1}{2^k}\,\rho_{2,\textcolor{black}{1/\delta}}\big((t,\mu),(t_k,\mu_k)\big), \qquad \forall\,(t,\mu)\in[0,T]\times\Pc_2(\R^d);
\]
\item $G(t,\mu) - \textcolor{black}{\delta^2}\,\varphi_{\textcolor{black}{\delta}}(t,\mu)<G(\tilde t,\tilde\mu) - \textcolor{black}{\delta^2}\,\varphi_{\textcolor{black}{\delta}}(\tilde t,\tilde\mu)$, for every $(t,\mu)\in([0,T]\times\Pc_2(\R^d))\backslash\{(\tilde t,\tilde\mu)\}$.
\end{enumerate}
Furthermore, the function $\varphi_{\textcolor{black}{\delta}}$ satisfies the following properties.
\begin{enumerate}[\upshape 1)]
\item $\varphi_{\textcolor{black}{\delta}}\in C^{1,2}([0,T]\times\Pc_2(\R^d))$;
\item its time derivative is bounded by $4T$;
\item its measure derivative \textcolor{black}{is bounded by
\begin{equation}\label{MeasDeriv_phi_Bound}
\big|\partial_\mu\varphi_\delta(t,\mu)(x)\big| \ = \ 2C_d\delta^2\bigg(\int_{\R^d} |y|^3\,\zeta_{1/\delta}(y)\,dy + |x|^2\int_{\R^d} |y|\,\zeta_{1/\delta}(y)\,dy\bigg),
\end{equation}
with the same constant $C_d$ as in \eqref{bounds_derivatives_1} and $\zeta_{1/\delta}$ given by \eqref{zeta_varrho} with $\varrho=1/\delta$;}
\item its second-order measure derivative \textcolor{black}{is bounded by
\begin{align}\label{MeasDeriv2nd_phi_Bound}
\big|\partial_x\partial_\mu\varphi_\delta(t,\mu)(x)\big| \ &= \ 2C_d\delta^2\bigg(\int_{\R^d}|y|^2\,\big(\sqrt{d} + |y|^2\delta^2\big)\,\zeta_{1/\delta}(y)\,dy \\
&\quad \ + |x|^2 \int_{\R^d}\big(\sqrt{d} + |y|^2\delta^2\big)\,\zeta_{1/\delta}(y)\,dy\bigg), \notag
\end{align}
with the same constant $C_d$ as in \eqref{bounds_derivatives_2} and $\zeta_{1/\delta}$ given by \eqref{zeta_varrho} with $\varrho=1/\delta$.}
\end{enumerate}
\end{Theorem}
\begin{proof}[\textbf{Proof.}]
Items (i)-(ii)-(iii) follow directly from the Borwein-Preiss variational principle \cite[Theorem 2.5.2]{BZ05} applied on $[0,T]\times\Pc_2(\R^d)$ with gauge-type function $\rho_{2,\textcolor{black}{1/\delta}}$ (we only remark that, concerning the sequence $\{\delta_i\}_{i\geq0}$ appearing in the statement of \cite[Theorem 2.5.2]{BZ05}, here we take $\delta_i=\delta^2/2^i$, $i\geq0$). Finally, items 2)-3)-4) follow respectively from \eqref{bound_time_derivative}-\eqref{bounds_derivatives_1}-\eqref{bounds_derivatives_2}.
\end{proof}

\section{Comparison theorem and uniqueness}
\label{S:Comparison}

\begin{Theorem}[Comparison]\label{T:Comparison}
Let Assumptions \textup{\ref{AssA}}, \ref{AssB}, \ref{AssC}, \ref{AssD} hold. Consider bounded and continuous functions $u_1\colon[0,T]\times\Pc_2(\R^d)\rightarrow\R$ and $u_2\colon[0,T]\times\Pc_2(\R^d)\rightarrow\R$, with $u_1$ $($resp. $u_2$$)$ being a viscosity subsolution $($resp. supersolution$)$ to equation \eqref{HJB}. Then, it holds that $u_1\leq u_2$ on $[0,T]\times\Pc_2(\R^d)$.
\end{Theorem}
\begin{proof}[\textbf{Proof.}]
Let $v_0$ be the map defined by \eqref{v_eps} with $\eps=0$ (see also Remark \ref{R:v_0}). Our aim is to prove that $u_1\leq v_0$ and $v_0\leq u_2$ on $[0,T]\times\Pc_2(\R^d)$, from which the claim follows.

\vspace{1mm}

\noindent\textsc{Step I.} \emph{Proof of $u_1\leq v_0$.} By contradiction, we suppose that there exists $(t_0,\tilde\mu_0)\in[0,T]\times\Pc_2(\R^d)$ such that
\[
(u_1 - v_0)(t_0,\tilde\mu_0) \ > \ 0.
\]
Since both $u_1$ and $v_0$ are continuous, we can find $q>2$ and $\mu_0\in\Pc_q(\R^d)$ such that
\begin{equation}\label{Contradiction}
(u_1 - v_0)(t_0,\mu_0) \ > \ 0.
\end{equation}
As a matter of fact, let $\xi\in L^2(\Omega,\Fc,\P;\R^d)$ be such that $\P_\xi=\tilde\mu_0$. For every $k\in\N$, let $\mu_0^k\in\Pc_2(\R^d)$ be the distribution of $\xi_k:=\xi\,1_{\{|\xi|\leq k\}}$. We see that $\mu_0^k\in\Pc_q(\R^d)$, for any $q\geq1$. Moreover, it holds that
\[
\Wc_2(\mu_0^k,\tilde\mu_0)^2 \ \leq \ \E\big[|\xi_k - \xi|^2\big] \ = \ \int_{|x|>k} |x|^2\,\tilde\mu_0(dx) \ \overset{k\rightarrow+\infty}{\longrightarrow} \ 0,
\]
from which we deduce that \eqref{Contradiction} holds with $\mu_0:=\mu_0^k$ for some $k$ large enough.\\
We split the rest of the proof of \textsc{Step I} into four substeps.

\vspace{1mm}

\noindent\textsc{Substep I-a}. For every $\eps>0$ and $n,m\in\N$, let $v_{\eps,n,m}$ be the map given by \eqref{v_eps,n}. Now, we define $\check u_1(t,\mu):=\text{e}^{t-t_0} u_1(t,\mu)$, for every $(t,\mu)\in[0,T]\times\Pc_2(\R^d)$, and similarly $\check v_{\eps,n,m}$, $\check f_{n,m}^i$, $\check f$
from $v_{\eps,n,m}$, $f_{n,m}^i$, $f$, respectively. We also define $\check g(x,\mu):=\text{e}^{T-t_0} g(x,\mu)$ and $\check g_{n,m}^i(x,\mu):=\text{e}^{T-t_0} g_{n,m}^i(x,\mu)$, for every $(x,\mu)\in\R^d\times\Pc_2(\R^d)$. We observe that $\check u_1$ is a viscosity subsolution of the following equation:
\begin{equation}\label{HJB_exp}
\begin{cases}
\vspace{2mm}
\displaystyle\partial_t \check u_1(t,\mu) + \int_{\R^d}\sup_{a\in A}\bigg\{\check f(t,x,\mu,a) + \dfrac{1}{2}\text{tr}\big[(\sigma\sigma\trans)(t,x,a)\partial_x\partial_\mu \check u_1(t,\mu)(x)\big] \\
\displaystyle\vspace{2mm}+\,\big\langle b(t,x,\mu,a),\partial_\mu \check u_1(t,\mu)(x)\rangle\bigg\}\mu(dx) = \check u_1(t,\mu), &\hspace{-3cm}(t,\mu)\in[0,T)\times \Pc_2(\R^d), \\
\displaystyle \check u_1(T,\mu) = \int_{\R^d} \check g(x,\mu)\mu(dx), &\hspace{-3cm}\,\mu\in\Pc_2(\R^d).
\end{cases}
\end{equation}
Moreover, by Theorem \ref{T:SmoothApprox} we deduce that $\check v_{\eps,n,m}$ solves the following equation:
\begin{equation}\label{HJB_eps,n_exp}
\hspace{-6mm}\begin{cases}
\vspace{2mm}
\displaystyle\partial_t \check v_{\eps,n,m}(t,\mu) - \check v_{\eps,n,m}(t,\mu) + \int_{\R^{dn}}\sum_{i=1}^n\sup_{a_i\in A}\bigg\{\langle b_{n,m}^i(t,x_1,\ldots,x_n,a_i),\partial_{x_i}\check{\bar v}_{\eps,n,m}(t,\bar x)\rangle \\
\displaystyle\vspace{2mm}+\,\frac{1}{2}\textup{tr}\Big[\big((\sigma\sigma\trans)(t,x_i,a_i) + \eps^2\big)\partial_{x_ix_i}^2\check{\bar v}_{\eps,n,m}(t,\bar x)\Big] \\
\displaystyle\vspace{2mm}+\,\frac{1}{n}\check f_{n,m}^i(t,x_1,\ldots,x_n,a_i)\bigg\}\mu(dx_1)\otimes\cdots\otimes\mu(dx_n) = 0, &\hspace{-4.5cm}(t,\mu)\in[0,T)\times \Pc_2(\R^d), \\
\displaystyle \check v_{\eps,n,m}(T,\mu) = \frac{1}{n}\sum_{i=1}^n \int_{\R^{dn}} \check g_{n,m}^i(\bar x)\,\mu(dx_1)\otimes\cdots\otimes\mu(dx_n), &\hspace{-2.4cm}\,\mu\in\Pc_2(\R^d),
\end{cases}
\end{equation}
where $\check{\bar v}_{\eps,n,m}(t,\bar x):=\textup{e}^{t-t_0}\bar v_{\eps,n,m}(t,\bar x)$, for every $(t,\bar x)\in[0,T]\times\R^{dn}$, $\bar x=(x_1,\ldots,x_n)$, with $\bar v_{\eps,n,m}$ being the same function appearing in Theorem \ref{T:SmoothApprox}.\\
Finally, notice that, by Assumption \ref{AssA}-(iii), $v_{\eps,n,m}$ is bounded by a constant independent of $\eps,n,m$. Since also $u_1$ is bounded, there exists $\lambda\geq0$, independent of $\eps,n,m$, satisfying
\begin{equation}\label{lambda}
\sup (\check u_1 - \check v_{\eps,n,m}) \ \leq \ (\check u_1 - \check v_{\eps,n,m})(t_0,\mu_0) + \lambda.
\end{equation}

\vspace{1mm}

\noindent\textsc{Substep I-b.} Since $\check u_1 - \check v_{\eps,n,m}$ is bounded and continuous, by \eqref{lambda} and Theorem \ref{T:SmoothVarPrinc} with $G=\check u_1-\check v_{\eps,n,m}$, we obtain that for every $\delta>0$ there exist $\{(t_k,\mu_k)\}_{k\geq1}\subset[0,T]\times\Pc_2(\R^d)$, converging to some $(\tilde t,\tilde\mu)\in[0,T]\times\Pc_2(\R^d)$, and $\varphi_\delta$ such that items (i)-(ii)-(iii) and 1)-2)-3) of Theorem \ref{T:SmoothVarPrinc} hold.\\
Now, recall from the proof of Lemma \ref{L:SmoothGauge}, and in particular from \eqref{item_c}-\eqref{eta_eps}, that for all $(t,\mu),(s,\nu)\in[0,T]\times\Pc_2(\R^d)$ satisfying $\rho_{2,1/\delta}((t,\mu),(s,\nu))\leq\eta_\epsilon$, with $\eta_\epsilon$ as in \eqref{eta_eps}, namely
\[
\eta_\epsilon \ := \ \Big(\sqrt{8c_d+\epsilon^2/2} - \sqrt{8c_d}\Big)^2,
\]
it holds that $\Wc_2^{(1/\delta)}(\mu,\nu)\leq\epsilon$. Since by item (i) of Theorem \ref{T:SmoothVarPrinc} we have $\rho_{2,1/\delta}((\tilde t,\tilde\mu),(t_0,\mu_0))\leq\lambda/\delta^2$, we get
\[
\Wc_2^{(1/\delta)}(\tilde\mu,\mu_0) \ \leq \ \frac{1}{\delta}\sqrt{2\lambda+8\delta\sqrt{2c_d\lambda}}.
\]
Finally, by \cite[Lemma 1]{NGK21} we obtain
\begin{equation}\label{delta_x}
\Wc_2(\tilde\mu,\mu_0) \ \leq \ \Wc_2^{(1/\delta)}(\tilde\mu,\mu_0) + \frac{2}{\delta}\sqrt{d+2} \ \leq \ \frac{1}{\delta}\bigg(\sqrt{2\lambda+8\delta\sqrt{2c_d\lambda}} + 2\sqrt{d+2}\bigg).
\end{equation}

\vspace{1mm}

\noindent\textsc{Substep I-c}. Let us prove that $\tilde t<T$. If $\tilde t=T$, from item (ii) of Theorem \ref{T:SmoothVarPrinc} we have
\[
(u_1 - v_{\eps,n,m})(t_0,\mu_0) \ = \ (\check u_1 - \check v_{\eps,n,m})(t_0,\mu_0) \ \leq \ \big(\check u_1 - \check v_{\eps,n,m}-\delta^2\varphi_\delta\big)(T,\tilde\mu) \ \leq \ (\check u_1 - \check v_{\eps,n,m})(T,\tilde\mu),
\]
where the last inequality follows from $\varphi_\delta\geq0$. Hence
\begin{align*}
&(u_1 - v_{\eps,n,m})(t_0,\mu_0) \\
&\leq \ \textup{e}^{T-t_0}\int_{\R^d}g(x,\tilde\mu)\,\tilde\mu(dx) - \frac{\textup{e}^{T-t_0}}{n}\sum_{i=1}^n \int_{\R^{dn}} g_{n,m}^i(x_1,\ldots,x_n)\,\tilde\mu(dx_1)\otimes\cdots\otimes\tilde\mu(dx_n) \\
&= \ \frac{\textup{e}^{T-t_0}}{n}\sum_{i=1}^n\int_{\R^d}\big(g(x_i,\tilde\mu) - g_{n,m}^i(x_1,\ldots,x_n)\big)\tilde\mu(dx_1)\otimes\cdots\otimes\tilde\mu(dx_n) \\
&= \ \frac{\textup{e}^{T-t_0}}{n}\sum_{i=1}^n\int_{\R^d}\big(g(x_i,\tilde\mu) - g(x_i,\widehat\mu^{n,\bar x})\big)\,\tilde\mu(dx_1)\otimes\cdots\otimes\tilde\mu(dx_n) \\
&\quad \ + \frac{\textup{e}^{T-t_0}}{n}\sum_{i=1}^n\int_{\R^d}\big(g(x_i,\widehat\mu^{n,\bar x}) - g_{n,m}^i(x_1,\ldots,x_n)\big)\,\tilde\mu(dx_1)\otimes\cdots\otimes\tilde\mu(dx_n),
\end{align*}
where $\widehat\mu^{n,\bar x}$ is given by
\[
\widehat\mu^{n,\bar x} \ := \ \frac{1}{n} \sum_{j=1}^n \delta_{x_j}
\]
for every $n\in\N$, $\bar x=(x_1,\ldots,x_n)\in\R^{dn}$ and $x_1,\ldots,x_n\in\R^d$. Then, from the Lipschitz property of $g$ we obtain
\begin{align}
(u_1 - v_{\eps,n,m})(t_0,\mu_0) \ &\leq \ \textup{e}^{T-t_0}\int_{\R^d}K\Wc_2\big(\tilde\mu,\widehat\mu^{n,\bar x}\big)\,\tilde\mu(dx_1)\otimes\cdots\otimes\tilde\mu(dx_n) \label{ProofTermCondition} \\
&\quad \ +
\frac{\textup{e}^{T-t_0}}{n}
\sum_{i=1}^n\int_{\R^d}\big(g(x_i,\widehat\mu^{n,\bar x}) - g_{n,m}^i(x_1,\ldots,x_n)\big)\,\tilde\mu(dx_1)
\otimes\cdots\otimes\tilde\mu(dx_n),
\notag
\end{align}
From \cite[Theorem 1]{FG15} we have that
\begin{align*}
&\int_{\R^{dn}}\Wc_2(\tilde\mu,\widehat\mu^{n,\bar x})\,\tilde\mu(dx_1)\otimes\cdots\otimes\tilde\mu(dx_n) \\
&\leq \ c_d \bigg(\int_{\R^d}|x|^q\,\tilde\mu(dx)\bigg)^{1/q}
\begin{cases}
\frac{1}{\sqrt{n}} + \frac{1}{n^{(q-1)/q}}, \qquad &\text{if }d=1\text{ and }q\neq2, \\
\frac{1}{\sqrt{n}}\log(1+n)	+ \frac{1}{n^{(q-1)/q}}, \qquad &\text{if }d=2\text{ and }q\neq2, \\
\frac{1}{n^{1/d}}	+ \frac{1}{n^{(q-1)/q}}, \qquad &\text{if }d>2\text{ and }q\neq\frac{d}{d-1},
\end{cases}
\end{align*}
with $q\in(1,2]$ and for some constant $c_d\geq0$, depending only on $d$. So, in particular, there exists some $q\in(1,2)$ such that
\begin{equation}\label{EmpirDistrLimit}
\int_{\R^{dn}}\Wc_2(\tilde\mu,\widehat\mu^{n,\bar x})\,\tilde\mu(dx_1)\otimes\cdots\otimes\tilde\mu(dx_n) \ \leq \ c_d \bigg(\int_{\R^d}|x|^q\,\tilde\mu(dx)\bigg)^{1/q}\,h_n
\end{equation}
for some sequence $\{h_n\}_n$ satisfying $\lim_{n\rightarrow+\infty}h_n=0$.
\\
Hence, plugging \eqref{EmpirDistrLimit} and \eqref{eq:limnewg} into \eqref{ProofTermCondition}, we get
\begin{align*}
(u_1 - v_{\eps,n,m})(t_0,\mu_0) \ &\leq \ c_d K\textup{e}^{T-t_0} \bigg(\int_{\R^d}|x|^q\,\tilde\mu(dx)\bigg)^{1/q}\,h_n \\
&\quad \ + K\textup{e}^{T-t_0}m^{nd}\int_{\R^{dn}}
\bigg(\frac{2}{n}\sum_{i=1}^n|y_i| 
\bigg)\prod_{j=1}^{n}\Phi(my_j)dy_j.
\end{align*}
Now, recalling that $q\in(1,2)$, we get (denoting by $\delta_0$ the Dirac measure centered at zero)
\begin{align}\label{delta_x_2}
&\bigg(\int_{\R^d}|x|^q\,\tilde\mu(dx)\bigg)^{1/q} \ \leq \ \bigg(\int_{\R^d}|x|^2\,\tilde\mu(dx)\bigg)^{1/2} \ = \ \Wc_2(\tilde\mu,\delta_0) \notag \\
&\leq \ \Wc_2(\tilde\mu,\mu_0) + \Wc_2(\mu_0,\delta_0) \ \leq \ \frac{1}{\delta}\bigg(\sqrt{2\lambda+8\delta\sqrt{2c_d\lambda}} + 2\sqrt{d+2}\bigg) + \Wc_2(\mu_0,\delta_0),
\end{align}
where the last inequality follows from \eqref{delta_x}. Hence
\begin{align*}
&(u_1 - v_{\eps,n,m})(t_0,\mu_0) \\
&\leq \ c_d K\textup{e}^{T-t_0} \bigg(\frac{1}{\delta}\bigg(\sqrt{2\lambda+8\delta\sqrt{2c_d\lambda}} + 2\sqrt{d+2}\bigg) + \Wc_2(\mu_0,\delta_0)\bigg)h_n \\
&\quad \ + K\textup{e}^{T-t_0} m^{nd}\int_{\R^{dn}}\bigg(\frac{2}{n}\sum_{i=1}^n|y_i|
\bigg)\prod_{j=1}^{n}\Phi(my_j)dy_j.
\end{align*}
Sending $m\rightarrow+\infty$, then $n\rightarrow+\infty$, and finally $\eps\rightarrow0^+$, we end up, using Theorem \ref{T:PropagChaos} and Lemma \ref{L:eps}, with
\[
(u_1 - v_0)(t_0,\mu_0) \ \leq \ 0,
\]
which gives a contradiction to \eqref{Contradiction}.

\vspace{1mm}

\noindent\textsc{Substep I-d}. From item (iii) of Theorem \ref{T:SmoothVarPrinc} and the fact that $\check u_1$ is a viscosity subsolution of \eqref{HJB_exp}, we find
\begin{align*}
&- \partial_t (\check v_{\eps,n,m}+\delta^2\varphi_\delta)(\tilde t,\tilde\mu) - \int_{\R^d}\sup_{a\in A}\bigg\{\big\langle b(\tilde t,x,\tilde\mu,a),\partial_\mu(\check v_{\eps,n,m}+\delta^2\varphi_\delta)(\tilde t,\tilde\mu)(x)\big\rangle \\
&+ \frac{1}{2}\text{tr}\big[(\sigma\sigma\trans)(\tilde t,x,a)\partial_x\partial_\mu(\check v_{\eps,n,m}+\delta^2\varphi_\delta)(\tilde t,\tilde\mu)(x)\big] + \check f(\tilde t,x,\tilde\mu,a)\bigg\}\tilde\mu(dx) + \check u_1(\tilde t,\tilde\mu) \ \leq \ 0.
\end{align*}
Then
\begin{align*}
&\check u_1(\tilde t,\tilde\mu) \ \leq \ \delta^2\,\partial_t \varphi_\delta(\tilde t,\tilde\mu) + \delta^2\int_{\R^d}\sup_{a\in A}\bigg\{\big\langle b(\tilde t,x,\tilde\mu,a),\partial_\mu \varphi_\delta(\tilde t,\tilde\mu)(x)\big\rangle \\
&+ \frac{1}{2}\text{tr}\big[(\sigma\sigma\trans)(\tilde t,x,a)\partial_x\partial_\mu\varphi_\delta(\tilde t,\tilde\mu)(x)\big]\bigg\}\tilde\mu(dx) + \int_{\R^d}\sup_{a\in A}\bigg\{\big\langle b(\tilde t,x,\tilde\mu,a),\partial_\mu \check v_{\eps,n,m}(\tilde t,\tilde\mu)(x)\big\rangle \\
&+ \frac{1}{2}\text{tr}\big[(\sigma\sigma\trans)(\tilde t,x,a)\partial_x\partial_\mu\check v_{\eps,n,m}(\tilde t,\tilde\mu)(x)\big] + \check f(\tilde t,x,\tilde\mu,a)\bigg\}\tilde\mu(dx) + \partial_t \check v_{\eps,n,m}(\tilde t,\tilde\mu).
\end{align*}
Using that the $\check v_{\eps,n,m}$ satisfies equation \eqref{HJB_eps,n_exp}, the above implies
\begin{align}\label{ComparisonIneqStepI}
&(\check u_1 - \check v_{\eps,n,m})(\tilde t,\tilde\mu) \ \leq \ \delta^2\,\partial_t \varphi_\delta(\tilde t,\tilde\mu) \\
&+ \delta^2\int_{\R^d}\sup_{a\in A}\bigg\{\big\langle b(\tilde t,x,\tilde\mu,a),\partial_\mu \varphi_\delta(\tilde t,\tilde\mu)(x)\big\rangle + \frac{1}{2}\text{tr}\big[(\sigma\sigma\trans)(\tilde t,x,a)\partial_x\partial_\mu\varphi_\delta(\tilde t,\tilde\mu)(x)\big]\bigg\}\tilde\mu(dx) \notag \\
&+ \int_{\R^d}\sup_{a\in A}\bigg\{\big\langle b(\tilde t,x,\tilde\mu,a),\partial_\mu \check v_{\eps,n,m}(\tilde t,\tilde\mu)(x)\big\rangle + \frac{1}{2}\text{tr}\big[(\sigma\sigma\trans)(\tilde t,x,a)\partial_x\partial_\mu\check v_{\eps,n,m}(\tilde t,\tilde\mu)(x)\big] \notag \\
&+ \check f(\tilde t,x,\tilde\mu,a)\bigg\}\tilde\mu(dx) - \int_{\R^{dn}}\sum_{i=1}^n\sup_{a_i\in A}\bigg\{\frac{1}{n}\check f_{n,m}^i(\tilde t,\bar x,a_i) + \langle b_{n,m}^i(\tilde t,\bar x,a_i),\partial_{x_i}\check{\bar v}_{\eps,n,m}(\tilde t,\bar x)\rangle \notag \\
&+ \frac{1}{2}\textup{tr}\Big[\big(\sigma\sigma\trans)(\tilde t,x_i,a_i) + \eps^2\big)\partial_{x_ix_i}^2\check{\bar v}_{\eps,n,m}(\tilde t,\bar x)\Big]\bigg\}\tilde\mu(dx_1)\otimes\cdots\otimes\tilde\mu(dx_n), \notag
\end{align}
with $\bar x=(x_1,\ldots,x_n)\in\R^{dn}$.
Now, recalling item (ii) of Theorem \ref{T:SmoothVarPrinc} and that $\varphi_\delta\geq0$, we find, using \eqref{ComparisonIneqStepI}, 
\begin{align*}
&(u_1 - v_{\eps,n,m})(t_0,\mu_0) \ =\
(\check u_1 - \check v_{\eps,n,m})(t_0,\mu_0) \ \leq \
(\check u_1 - \check v_{\eps,n,m})(\tilde t,\tilde\mu)
-\delta^2\varphi_\delta(\tilde t,\tilde\mu)
\\[2mm]
&\le
(\check u_1 - \check v_{\eps,n,m})(\tilde t,\tilde\mu)
\le
\delta^2\,\partial_t \varphi_\delta(\tilde t,\tilde\mu) \\
&+ \delta^2\int_{\R^d}\sup_{a\in A}\bigg\{\big\langle b(\tilde t,x,\tilde\mu,a),\partial_\mu \varphi_\delta(\tilde t,\tilde\mu)(x)\big\rangle + \frac{1}{2}\text{tr}\big[(\sigma\sigma\trans)(\tilde t,x,a)\partial_x\partial_\mu\varphi_\delta(\tilde t,\tilde\mu)(x)\big]\bigg\}\tilde\mu(dx) \\
&+ \int_{\R^d}\sup_{a\in A}\bigg\{\big\langle b(\tilde t,x,\tilde\mu,a),\partial_\mu \check v_{\eps,n,m}(\tilde t,\tilde\mu)(x)\big\rangle + \frac{1}{2}\text{tr}\big[(\sigma\sigma\trans)(\tilde t,x,a)\partial_x\partial_\mu\check v_{\eps,n,m}(\tilde t,\tilde\mu)(x)\big] \\
&+ \check f(\tilde t,x,\tilde\mu,a)\bigg\}\tilde\mu(dx) - \int_{\R^{dn}}\sum_{i=1}^n\sup_{a_i\in A}\bigg\{\frac{1}{n}\check f_{n,m}^i(\tilde t,\bar x,a_i) + \langle b_{n,m}^i(\tilde t,\bar x,a_i),\partial_{x_i}\check{\bar v}_{\eps,n,m}(\tilde t,\bar x)\rangle \\
&+ \frac{1}{2}\textup{tr}\Big[\big((\sigma\sigma\trans)(\tilde t,x_i,a_i) + \eps^2\big)\partial_{x_ix_i}^2\check{\bar v}_{\eps,n,m}(\tilde t,\bar x)\Big]\bigg\}\tilde\mu(dx_1)\otimes\cdots\otimes\tilde\mu(dx_n).
\end{align*}
Now, recalling that $b$ and $\sigma$ are bounded, by item 2) and estimates \eqref{MeasDeriv_phi_Bound}-\eqref{MeasDeriv2nd_phi_Bound} of Theorem \ref{T:SmoothVarPrinc}, we deduce that
\begin{align*}
&\partial_t \varphi_\delta(\tilde t,\tilde\mu) + \int_{\R^d}\sup_{a\in A}\bigg\{\big\langle b(\tilde t,x,\tilde\mu,a),\partial_\mu \varphi_\delta(\tilde t,\tilde\mu)(x)\big\rangle + \frac{1}{2}\text{tr}\big[(\sigma\sigma\trans)(\tilde t,x,a)\partial_x\partial_\mu\varphi_\delta(\tilde t,\tilde\mu)(x)\big]\bigg\}\tilde\mu(dx) \\
&\leq 4T + \Lambda\delta^2\bigg(\int_{\R^d} \big(|y|^2+|y|^3+|y|^4\delta^2\big)\zeta_{1/\delta}(y)dy + \Wc_2(\tilde\mu,\delta_0)^2\!\!\int_{\R^d} \big(1+|y|+|y|^2\delta^2\big)\zeta_{1/\delta}(y)dy\bigg),
\end{align*}
for some constant $\Lambda\geq0$, independent of $\eps,n,m,\delta$, where $\delta_0$ is the Dirac measure centered at zero, so that $\Wc_2(\tilde\mu,\delta_0)^2=\int_{\R^d}|x|^2\tilde\mu(dx)$. Hence
\begin{align}\label{ComparisonIneqStepII}
&(u_1 - v_{\eps,n,m})(t_0,\mu_0) \ \leq \ 4\delta^2T \\
&+ \Lambda\delta^4\bigg(\int_{\R^d} \big(|y|^2+|y|^3+|y|^4\delta^2\big)\zeta_{1/\delta}(y)dy + \Wc_2(\tilde\mu,\delta_0)^2\!\!\int_{\R^d} \big(1+|y|+|y|^2\delta^2\big)\zeta_{1/\delta}(y)dy\bigg) \notag \\
&+ \int_{\R^d}\sup_{a\in A}\bigg\{\big\langle b(\tilde t,x,\tilde\mu,a),\partial_\mu \check v_{\eps,n,m}(\tilde t,\tilde\mu)(x)\big\rangle + \frac{1}{2}\text{tr}\big[(\sigma\sigma\trans)(\tilde t,x,a)\partial_x\partial_\mu\check v_{\eps,n,m}(\tilde t,\tilde\mu)(x)\big] \notag \\
&+ \check f(\tilde t,x,\tilde\mu,a)\bigg\}\tilde\mu(dx) - \int_{\R^{dn}}\sum_{i=1}^n\sup_{a_i\in A}\bigg\{\frac{1}{n}\check f_{n,m}^i(\tilde t,\bar x,a_i) + \langle b_{n,m}^i(\tilde t,\bar x,a_i),\partial_{x_i}\check{\bar v}_{\eps,n,m}(\tilde t,\bar x)\rangle \notag \\
&+ \frac{1}{2}\textup{tr}\Big[\big((\sigma\sigma\trans)(\tilde t,x_i,a_i) + \eps^2\big)\partial_{x_ix_i}^2\check{\bar v}_{\eps,n,m}(\tilde t,\bar x)\Big]\bigg\}\tilde\mu(dx_1)\otimes\cdots\otimes\tilde\mu(dx_n). \notag
\end{align}
By formulae \eqref{partial_mu_identity} and \eqref{partial_x_partial_mu_identity}, we have
\begin{align}\label{ComparisonIneqStepIII}
&\int_{\R^d}\sup_{a\in A}\bigg\{\big\langle b(\tilde t,x,\tilde\mu,a),\partial_\mu \check v_{\eps,n,m}(\tilde t,\tilde\mu)(x)\big\rangle + \frac{1}{2}\text{tr}\big[(\sigma\sigma\trans)(\tilde t,x,a)\partial_x\partial_\mu\check v_{\eps,n,m}(\tilde t,\tilde\mu)(x)\big] \notag \\
&\quad \ + \check f(\tilde t,x,\tilde\mu,a)\bigg\}\tilde\mu(dx) \notag \\
&= \ \int_{\R^d}\sup_{a\in A}\bigg\{\int_{\R^{d(n-1)}}\sum_{i=1}^n\bigg\{\big\langle b(\tilde t,x,\tilde\mu,a),\partial_{x_i} \check{\bar v}_{\eps,n,m}(t,x_1,\ldots,x_{i-1},x,x_{i+1},\ldots,x_n)\big\rangle \notag \\
&\quad \ + \frac{1}{2}\text{tr}\big[(\sigma\sigma\trans)(\tilde t,x,a)\partial_{x_ix_i}^2 \check{\bar v}_{\eps,n,m}(t,x_1,\ldots,x_{i-1},x,x_{i+1},\ldots,x_n)\big] \notag \\
&\quad \ + \frac{1}{n} \check f(\tilde t,x,\tilde\mu,a)\bigg\}\tilde\mu(dx_1)\cdots\tilde\mu(dx_{i-1})\,\tilde\mu(dx_{i+1})\cdots\tilde\mu(dx_n)\bigg\}\tilde\mu(dx) \notag \\
&\leq \ \int_{\R^d}\int_{\R^{d(n-1)}}\sum_{i=1}^n\sup_{a\in A}\bigg\{\big\langle b(\tilde t,x,\tilde\mu,a),\partial_{x_i} \check{\bar v}_{\eps,n,m}(t,x_1,\ldots,x_{i-1},x,x_{i+1},\ldots,x_n)\big\rangle \notag \\
&\quad \ + \frac{1}{2}\text{tr}\big[(\sigma\sigma\trans)(\tilde t,x,a)\partial_{x_ix_i}^2 \check{\bar v}_{\eps,n,m}(t,x_1,\ldots,x_{i-1},x,x_{i+1},\ldots,x_n)\big] \notag \\
&\quad \ + \frac{1}{n} \check f(\tilde t,x,\tilde\mu,a)\bigg\}\tilde\mu(dx_1)\cdots\tilde\mu(dx_{i-1})\,\tilde\mu(dx_{i+1})\cdots\tilde\mu(dx_n)\tilde\mu(dx) \notag \\
&= \ \int_{\R^{dn}}\sum_{i=1}^n\sup_{a_i\in A}\bigg\{\big\langle b(\tilde t,x_i,\tilde\mu,a_i),\partial_{x_i} \check{\bar v}_{\eps,n,m}(t,\bar x)\big\rangle + \frac{1}{2}\text{tr}\big[(\sigma\sigma\trans)(\tilde t,x_i,a_i)\partial_{x_ix_i}^2 \check{\bar v}_{\eps,n,m}(t,\bar x)\big] \notag \\
&\quad \ + \frac{1}{n} \check f(\tilde t,x_i,\tilde\mu,a_i)\bigg\}\tilde\mu(dx_1)\cdots\tilde\mu(dx_n).
\end{align}
Plugging \eqref{ComparisonIneqStepIII} into \eqref{ComparisonIneqStepII}, we obtain
\begin{align*}
&(u_1 - v_{\eps,n,m})(t_0,\mu_0) \ \leq \ 4\delta^2T \\
&+ \Lambda\delta^4\bigg(\int_{\R^d} \big(|y|^2+|y|^3+|y|^4\delta^2\big)\zeta_{1/\delta}(y)dy + \Wc_2(\tilde\mu,\delta_0)^2\!\!\int_{\R^d} \big(1+|y|+|y|^2\delta^2\big)\zeta_{1/\delta}(y)dy\bigg) \\
&+ \int_{\R^{dn}}\sum_{i=1}^n\sup_{a_i\in A}\bigg\{\big\langle b(\tilde t,x_i,\tilde\mu,a_i) - b_{n,m}^i(\tilde t,\bar x,a_i),\partial_{x_i} \check{\bar v}_{\eps,n,m}(t,\bar x)\big\rangle \notag \\
&+ \frac{1}{n} \check f(\tilde t,x_i,\tilde\mu,a_i) - \frac{1}{n}\check f_{n,m}^i(\tilde t,\bar x,a_i)\bigg\}\tilde\mu(dx_1)\otimes\cdots\otimes\tilde\mu(dx_n) \notag \\
&- \frac{1}{2}\eps^2\int_{\R^{dn}}\sum_{i=1}^n\textup{tr}\big[\partial_{x_ix_i}^2\check{\bar v}_{\eps,n,m}(\tilde t,\bar x)\big]\tilde\mu(dx_1)\otimes\cdots\otimes\tilde\mu(dx_n) \notag \\
&\leq \ 4\delta^2T \\
&+ \Lambda\delta^4\bigg(\int_{\R^d} \big(|y|^2+|y|^3+|y|^4\delta^2\big)\zeta_{1/\delta}(y)dy + \Wc_2(\tilde\mu,\delta_0)^2\!\!\int_{\R^d} \big(1+|y|+|y|^2\delta^2\big)\zeta_{1/\delta}(y)dy\bigg) \notag \\
&+ \int_{\R^{dn}}\sum_{i=1}^n\sup_{a_i\in A}\Big\{\big| b(\tilde t,x_i,\tilde\mu,a_i) - b_{n,m}^i(\tilde t,\bar x,a_i)\big|\big|\partial_{x_i} \check{\bar v}_{\eps,n,m}(t,\bar x)\big|\Big\}\tilde\mu(dx_1)\otimes\cdots\otimes\tilde\mu(dx_n) \notag \\
&+ \frac{1}{n} \int_{\R^{dn}}\sum_{i=1}^n\sup_{a_i\in A}\Big\{\big|\check f(\tilde t,x_i,\tilde\mu,a_i) - \check f_{n,m}^i(\tilde t,\bar x,a_i)\big|\Big\}\tilde\mu(dx_1)\otimes\cdots\otimes\tilde\mu(dx_n) \notag \\
&- \frac{1}{2}\eps^2\int_{\R^{dn}}\sum_{i=1}^n\textup{tr}\big[\partial_{x_ix_i}^2\check{\bar v}_{\eps,n,m}(\tilde t,\bar x)\big]\tilde\mu(dx_1)\otimes\cdots\otimes\tilde\mu(dx_n) \notag \\
&\leq \ 4\delta^2T \\
&+ \Lambda\delta^4\bigg(\int_{\R^d} \big(|y|^2+|y|^3+|y|^4\delta^2\big)\zeta_{1/\delta}(y)dy + \Wc_2(\tilde\mu,\delta_0)^2\!\!\int_{\R^d} \big(1+|y|+|y|^2\delta^2\big)\zeta_{1/\delta}(y)dy\bigg) \notag \\
&+ \frac{C_K}{n} \textup{e}^{\tilde t-t_0}\int_{\R^{dn}}\sum_{i=1}^n\sup_{a\in A}\Big\{\big| b(\tilde t,x_i,\tilde\mu,a) - b_{n,m}^i(\tilde t,\bar x,a)\big|\Big\}\tilde\mu(dx_1)\otimes\cdots\otimes\tilde\mu(dx_n) \notag \\
&+ \frac{1}{n} \textup{e}^{\tilde t-t_0} \int_{\R^{dn}}\sum_{i=1}^n\sup_{a\in A}\Big\{\big|f(\tilde t,x_i,\tilde\mu,a) - f_{n,m}^i(\tilde t,\bar x,a)\big|\Big\}\tilde\mu(dx_1)\otimes\cdots\otimes\tilde\mu(dx_n) \notag \\
&- \frac{1}{2}\eps^2\int_{\R^{dn}}\sum_{i=1}^n\textup{tr}\big[\partial_{x_ix_i}^2\check{\bar v}_{\eps,n,m}(\tilde t,\bar x)\big]\tilde\mu(dx_1)\otimes\cdots\otimes\tilde\mu(dx_n), \notag
\end{align*}
where the last inequality follows from estimate \eqref{Estimate_1stDeriv_n}. Recalling the left estimate in \eqref{Estimate_2ndDeriv}, we obtain
\begin{align*}
&(u_1 - v_{\eps,n,m})(t_0,\mu_0) \ \leq \ \frac{1}{2}\eps^2\,n\,d\,C_{n,m}\,\textup{e}^{\tilde t-t_0} + 4\delta^2T \\
&+ \Lambda\delta^4\bigg(\int_{\R^d} \big(|y|^2+|y|^3+|y|^4\delta^2\big)\zeta_{1/\delta}(y)dy + \Wc_2(\tilde\mu,\delta_0)^2\!\!\int_{\R^d} \big(1+|y|+|y|^2\delta^2\big)\zeta_{1/\delta}(y)dy\bigg) \notag \\
&+ \frac{C_K}{n} \textup{e}^{\tilde t-t_0}\int_{\R^{dn}}\sum_{i=1}^n\sup_{a\in A}\Big\{\big| b(\tilde t,x_i,\tilde\mu,a) - b_{n,m}^i(\tilde t,\bar x,a)\big|\Big\}\tilde\mu(dx_1)\otimes\cdots\otimes\tilde\mu(dx_n) \notag \\
&+ \frac{1}{n} \textup{e}^{\tilde t-t_0} \int_{\R^{dn}}\sum_{i=1}^n\sup_{a\in A}\Big\{\big|f(\tilde t,x_i,\tilde\mu,a) - f_{n,m}^i(\tilde t,\bar x,a)\big|\Big\}\tilde\mu(dx_1)\otimes\cdots\otimes\tilde\mu(dx_n) \notag \\
&\leq \ \frac{1}{2}\eps^2\,n\,d\,C_{n,m}\,\textup{e}^{\tilde t-t_0} + 4\delta^2T \\
&+ \Lambda\delta^4\bigg(\int_{\R^d} \big(|y|^2+|y|^3+|y|^4\delta^2\big)\zeta_{1/\delta}(y)dy + \Wc_2(\tilde\mu,\delta_0)^2\!\!\int_{\R^d} \big(1+|y|+|y|^2\delta^2\big)\zeta_{1/\delta}(y)dy\bigg) \notag \\
&+ \frac{C_K}{n} \textup{e}^{\tilde t-t_0}\int_{\R^{dn}}\sum_{i=1}^n\sup_{a\in A}\Big\{\big| b(\tilde t,x_i,\tilde\mu,a) - b(\tilde t,x_i,\widehat\mu^{n,\bar x},a)\big|\Big\}\tilde\mu(dx_1)\otimes\cdots\otimes\tilde\mu(dx_n) \notag \\
&+ \frac{C_K}{n} \textup{e}^{\tilde t-t_0}\int_{\R^{dn}}\sum_{i=1}^n\sup_{a\in A}\Big\{\big|b(\tilde t,x_i,\widehat\mu^{n,\bar x},a) - b_{n,m}^i(\tilde t,\bar x,a)\big|\Big\}\tilde\mu(dx_1)\otimes\cdots\otimes\tilde\mu(dx_n) \notag \\
&+ \frac{1}{n}  \textup{e}^{\tilde t-t_0} \int_{\R^{dn}}\sum_{i=1}^n\sup_{a\in A}\Big\{\big|f(\tilde t,x_i,\tilde\mu,a) - f(\tilde t,x_i,\widehat\mu^{n,\bar x},a)\big|\Big\}\tilde\mu(dx_1)\otimes\cdots\otimes\tilde\mu(dx_n) \notag \\
&+ \frac{1}{n}  \textup{e}^{\tilde t-t_0} \int_{\R^{dn}}\sum_{i=1}^n\sup_{a\in A}\Big\{\big|f(\tilde t,x_i,\widehat\mu^{n,\bar x},a) - f_{n,m}^i(\tilde t,\bar x,a)\big|\Big\}\tilde\mu(dx_1)\otimes\cdots\otimes\tilde\mu(dx_n) \notag \\
&\leq \ \frac{1}{2}\eps^2\,n\,d\,C_{n,m}\,\textup{e}^{\tilde t-t_0} + 4\delta^2T \\
&+ \Lambda\delta^4\bigg(\int_{\R^d} \big(|y|^2+|y|^3+|y|^4\delta^2\big)\zeta_{1/\delta}(y)dy + \Wc_2(\tilde\mu,\delta_0)^2\!\!\int_{\R^d} \big(1+|y|+|y|^2\delta^2\big)\zeta_{1/\delta}(y)dy\bigg) \notag \\
&+ \frac{C_K}{n} \textup{e}^{\tilde t-t_0}\int_{\R^{dn}}\sum_{i=1}^n K\,\Wc_2(\tilde\mu,\widehat\mu^{n,\bar x})\,\tilde\mu(dx_1)\otimes\cdots\otimes\tilde\mu(dx_n) \notag \\
&+ \frac{C_K}{n} \textup{e}^{\tilde t-t_0}\int_{\R^{dn}}\sum_{i=1}^n\sup_{a\in A}\Big\{\big|b(\tilde t,x_i,\widehat\mu^{n,\bar x},a) - b_{n,m}(\tilde t,\bar x,a)\big|\Big\}\tilde\mu(dx_1)\otimes\cdots\otimes\tilde\mu(dx_n) \notag \\
&+ \frac{1}{n}  \textup{e}^{\tilde t-t_0} \int_{\R^{dn}}\sum_{i=1}^n K\,\Wc_2(\tilde\mu,\widehat\mu^{n,\bar x})\,\tilde\mu(dx_1)\otimes\cdots\otimes\tilde\mu(dx_n) \notag \\
&+ \frac{1}{n}  \textup{e}^{\tilde t-t_0} \int_{\R^{dn}}\sum_{i=1}^n\sup_{a\in A}\Big\{\big|f(\tilde t,x_i,\widehat\mu^{n,\bar x},a) - f_{n,m}(\tilde t,\bar x,a)\big|\Big\}\tilde\mu(dx_1)\otimes\cdots\otimes\tilde\mu(dx_n), \notag
\end{align*}
where the last inequality follows from the Lipschitz property of $b$ and $f$. Recalling \eqref{EmpirDistrLimit} and \eqref{delta_x_2} we find
\begin{align}\label{ComparisonIneqStepV}
&(u_1 - v_{\eps,n,m})(t_0,\mu_0) \ \leq \ \frac{1}{2}\eps^2\,n\,d\,C_{n,m}\,\textup{e}^{\tilde t-t_0} + 4\delta^2T \\
&+ \Lambda\delta^4\bigg(\int_{\R^d} \big(|y|^2+|y|^3+|y|^4\delta^2\big)\zeta_{1/\delta}(y)dy + \Wc_2(\tilde\mu,\delta_0)^2\!\!\int_{\R^d} \big(1+|y|+|y|^2\delta^2\big)\zeta_{1/\delta}(y)dy\bigg) \notag \\
&+ (C_K+1)\,K\, \textup{e}^{\tilde t-t_0}c_d \bigg(\frac{1}{\delta}\bigg(\sqrt{2\lambda+8\delta\sqrt{2c_d\lambda}} + 2\sqrt{d+2}\bigg) + \Wc_2(\mu_0,\delta_0)\bigg)\,h_n \notag \\
&+ \frac{C_K}{n} \textup{e}^{\tilde t-t_0}\int_{\R^{dn}}\sum_{i=1}^n\sup_{a\in A}\Big\{\big|b(\tilde t,x_i,\widehat\mu^{n,\bar x},a) - b_{n,m}^i(\tilde t,\bar x,a)\big|\Big\}\tilde\mu(dx_1)\otimes\cdots\otimes\tilde\mu(dx_n) \notag \\
&+ \frac{1}{n}  \textup{e}^{\tilde t-t_0} \int_{\R^{dn}}\sum_{i=1}^n\sup_{a\in A}\Big\{\big|f(\tilde t,x_i,\widehat\mu^{n,\bar x},a) - f_{n,m}^i(\tilde t,\bar x,a)\big|\Big\}\tilde\mu(dx_1)\otimes\cdots\otimes\tilde\mu(dx_n). \notag
\end{align}
Now, from the Lemma \ref{lm:conv} we get
\begin{align*}
&\big|b(\tilde t,x_i,\widehat\mu^{n,\bar x},a) - b_{n,m}^i(\tilde t,\bar x,a)\big| \\
&\leq \ K m^{nd+1}\int_{\R^{dn+1}}\bigg(\big|\tilde t - T\wedge(\tilde t-s)^+\big|^\beta + |y_i| + \frac{1}{n} \sum_{j=1}^n |y_j|\bigg)\zeta(ms)\prod_{j=1}^{n}\Phi(my_j)dy_jds.
\end{align*}
An analogous estimate holds for $|f(\tilde t,x_i,\widehat\mu^{n,\bar x},a)-f_{n,m}^i(\tilde t,\bar x,a)|$. Then, plugging these estimates into \eqref{ComparisonIneqStepV} we obtain
\begin{align*}
&(u_1 - v_{\eps,n,m})(t_0,\mu_0) \ \leq \ \frac{1}{2}\eps^2\,n\,d\,C_{n,m}\,\textup{e}^{\tilde t-t_0} + 4\delta^2T \\
&+ \Lambda\delta^4\bigg(\int_{\R^d} \big(|y|^2+|y|^3+|y|^4\delta^2\big)\zeta_{1/\delta}(y)dy + \Wc_2(\tilde\mu,\delta_0)^2\!\!\int_{\R^d} \big(1+|y|+|y|^2\delta^2\big)\zeta_{1/\delta}(y)dy\bigg) \notag \\
&+ (C_K+1)\,K\, \textup{e}^{\tilde t-t_0}c_d \bigg(\frac{1}{\delta}\bigg(\sqrt{2\lambda+8\delta\sqrt{2c_d\lambda}} + 2\sqrt{d+2}\bigg) + \Wc_2(\mu_0,\delta_0)\bigg)\,h_n \notag \\
&+ (C_K+1)\,K\, \textup{e}^{\tilde t-t_0}m^{nd+1}\int_{\R^{dn+1}}\bigg(\big|\tilde t - T\wedge(\tilde t-s)^+\big|^\beta + \frac{1}{n}\sum_{i=1}^n|y_i| \notag \\
&+ \frac{1}{n} \sum_{j=1}^n |y_j|\bigg)\zeta(ms)\prod_{j=1}^{n}\Phi(my_j)dy_jds \notag \\
&\leq \ \frac{1}{2}\eps^2\,n\,d\,C_{n,m}\,\textup{e}^{\tilde t-t_0} + 4\delta^2T + \Lambda\delta^4\bigg(\int_{\R^d} \big(|y|^2+|y|^3+|y|^4\delta^2\big)\zeta_{1/\delta}(y)dy \notag \\
&+ \frac{1}{\delta^2}\bigg(\sqrt{2\lambda+8\delta\sqrt{2c_d\lambda}} + 2\sqrt{d+2}\bigg)^2\int_{\R^d} \big(1+|y|+|y|^2\delta^2\big)\zeta_{1/\delta}(y)dy\bigg) \notag \\
&+ (C_K+1)\,K\, \textup{e}^{\tilde t-t_0}c_d \bigg(\frac{1}{\delta}\bigg(\sqrt{2\lambda+8\delta\sqrt{2c_d\lambda}} + 2\sqrt{d+2}\bigg) + \Wc_2(\mu_0,\delta_0)\bigg)\,h_n \notag \\
&+ (C_K+1)\,K\, \textup{e}^{\tilde t-t_0}m^{nd+1}\int_{\R^{dn+1}}\bigg(\big|\tilde t - T\wedge(\tilde t-s)^+\big|^\beta + \frac{1}{n}\sum_{i=1}^n|y_i| \notag \\
&+ \frac{1}{n} \sum_{j=1}^n |y_j|\bigg)\zeta(ms)\prod_{j=1}^{n}\Phi(my_j)dy_jds, \notag
\end{align*}
where in the last inequality we have used again \eqref{delta_x_2}.

Now, we send $\eps\rightarrow0^+$ so the left hand side goes to $u_1-v_{0,n,m}$ and the first term after the last inequality above goes to zero. Then we send $m\rightarrow+\infty$ (so the last term above goes to zero), and afterwards $n\rightarrow+\infty$ (the second to last term above goes to zero) and use Theorem A.6. Finally, we send $\delta\rightarrow0^+$, from which we obtain (notice that $\int_{\R^d}|y|^\ell\zeta_{1/\delta}(y)dy
=\int_{\R^d}|z|^\ell\delta^{-\ell}\zeta_1(z)dz$, for every $\ell\in\N$, with $\zeta_1$ given by \eqref{zeta_varrho} with $\varrho=1$)
\[
(u_1 - v_0)(t_0,\mu_0) \ \leq \ 0.
\]
This gives a contradiction to \eqref{Contradiction}.

\vspace{2mm}

\noindent\textsc{Step II.} \emph{Proof of $v_0\leq u_2$.} 
Our aim is to prove that
\begin{equation}\label{ClaimStepII}
u_2(t,\mu) \ \geq \ \E\bigg[\int_t^s f\big(r,X_r^{t,\xi,\mathfrak a},\P_{X_r^{t,\xi,\mathfrak a}},\mathfrak a\big)\,dr\bigg] + u_2\big(s,\P_{X_s^{t,\xi,\mathfrak a}}\big),
\end{equation}
for every $(t,\mu)\in[0,T]\times\Pc_2(\R^d)$, $s\in[t,T]$, $\xi\in L^2(\Omega,\Fc_t,\P;\R^d)$, with $\P_\xi=\mu$, and $\mathfrak a\in\Mc_t$, where $\Mc_t$ denotes the set of $\Fc_t$-measurable random variables $\mathfrak a\colon\Omega\rightarrow A$. $(X^{t,\xi,\mathfrak a})_{r\in[t,T]}$ is equal to the process $(X_r^{t,\xi,\alpha})_{r\in[t,T]}$ with $\alpha_r= \mathfrak a$ for $r\in[t,T]$.

To see why $u_2\geq v_0$ follows from \eqref{ClaimStepII}, we need to introduce some notation. First of all, following \cite[Definition 3.2.3]{Krylov80}, we define on $\Ac$ the metric $\rho_{\textup{Kr}}$ given by
\[
\rho_{\textup{Kr}}(\alpha,\beta) \ := \ \E\bigg[\int_0^T |\alpha_t-\beta_t|\,dt\bigg], \qquad \forall\,\alpha,\beta\in\Ac.
\]
Recall that, by Assumption \ref{AssC} $A$ is a compact subset of some Euclidean space, so that $|\alpha_t-\beta_t|$ denotes the Euclidean distance between $\alpha_t$ and $\beta_t$, moreover $|\alpha_t-\beta_t|$ is bounded by some constant which depends only on $A$. Following \cite{Krylov80}, we also define the class of step control processes for the problem starting at $t\in[0,T]$:
\begin{align*}
\Ac_{\textup{\tiny{step}}}^t \ := \ \Big\{\alpha\in\Ac\colon \text{there exist $n\in\N$ and $t=t_0<t_1<\cdots<t_{n-1}<t_n=T$}& \\
\text{such that }\alpha_s = \alpha_{t_i},\;\forall\,s\in[t_i,t_{i+1}),\;i=0,\ldots,n-1&\Big\}.
\end{align*}
By \cite[Lemma 3.2.6]{Krylov80}, we know that $\Ac_{\textup{\tiny{step}}}^t$ is dense in $\Ac$ with respect to the metric $\rho_{\textup{Kr}}$. Moreover, by similar arguments as in \cite[Lemma 3.2.7]{Krylov80}, it is easy to prove that the reward functional $J=J(s,\xi,\alpha)$ in \eqref{RewardFunct} is continuous in $\alpha$ with respect to the metric $\rho_{\textup{Kr}}$. Then, using the density of $\Ac_{\textup{\tiny{step}}}^t$ in $\Ac$ and the continuity of the reward functional with respect to $\rho_{\textup{Kr}}$, we deduce that $v_0(t,\mu)$ can be equivalently defined as the supremum of the reward functional over $\Ac_{\textup{\tiny{step}}}^t$ (rather than $\Ac$). Now, let $t\in[0,T]$ and $\alpha\in\Ac_{\textup{\tiny{step}}}^t$, so that there exist $n\in\N$, $t=t_0<t_1<\cdots<t_{n-1}<t_n=T$, and $\mathfrak a_0,\ldots,\mathfrak a_{n-1}\colon\Omega\rightarrow A$, with $\mathfrak a_i\in\Fc_{t_i}$, such that
\[
\alpha_s \ = \ \mathfrak a_i, \qquad \forall\,s\in[t_i,t_{i+1}),\;i=0,\ldots,n-1.
\]
If \eqref{ClaimStepII} holds true, applying it recursively on the intervals $[t_i,t_{i+1})$, $i=0,\dots,n-1$ with $\mathfrak a=\mathfrak a_0,\ldots,\mathfrak a_{n-1}$, we get
\[
u_2(t,\mu) \ \geq \ \E\bigg[\int_t^T f\big(r,X_r^{t,\xi,\alpha},\P_{X_r^{t,\xi,\alpha}},\alpha_r\big)\,dr + g\big(X_T^{t,\xi,\alpha},\P_{X_T^{t,\xi,\alpha}}\big)\bigg].
\]
Since $\alpha$ was arbitrary, the above inequality holds for every $\alpha\in\Ac_{\textup{\tiny{step}}}^t$, proving that $u_2\geq v_0$. It remains to prove \eqref{ClaimStepII}. To this end, for every $t\in[0,T]$, $\xi\in L^2(\Omega,\Fc,\P;\R^d)$, $\mathfrak a\in\Mc_t$, we consider the system of uncontrolled stochastic differential equations:
\[
\begin{cases}
X_s \ = \ \xi + \int_t^s b\big(r,X_r,\P_{X_r},Y_r\big)\,dr
+ \int_t^s \sigma(r,X_r,Y_r)\,dB_r, \qquad s\in[t,T], \\
Y_s = \mathfrak a, \qquad s\in[t,T].
\end{cases}
\]
We denote by $(X^{t,\xi,\mathfrak a},Y^{t,\mathfrak a})$ the unique solution to the above system of equations. Then, fixed $\underline t\in[0,T)$, $s\in(\underline t,T]$, we set
\[
v^s(t,\nu) \ := \ \E\bigg[\int_t^s f\big(r,X_r^{t,\xi,\mathfrak a},\P_{X_r^{t,\xi,\mathfrak a}},Y_r^{t,\mathfrak a}\big)\,dr\bigg] + u_2\big(s,\P_{X_s^{t,\xi,\mathfrak a}}\big),
\]
for all $(t,\nu)\in[\underline t,s]\times\Pc_2(\R^d\times A)$, $\xi\in L^2(\Omega,\Fc_t,\P;\R^d)$ and $\mathfrak a\in\Mc_{\underline t}$ such that $\P_{(\xi,\mathfrak a)}=\nu$. Then, our aim is to prove that $u_2(t,\mu)\geq v^s(t,\nu)$, for every $(t,\nu)\in[\underline t,s]\times\Pc_2(\R^d\times A)$, with $\mu$ being the first marginal of $\nu$, from which we get \eqref{ClaimStepII} for $t=\underline t$.\\
\textcolor{black}{Suppose for a moment that $u_2(s,\cdot)$ is Lipschitz continuous. Then, reasoning as in the proof of Proposition \ref{P:ValueFunction}, we obtain that $v^{s}$ is bounded and Lipschitz continuous. If $u_2(s,\cdot)$ is not Lipschitz continuous, following \cite[formula (5.1.4)]{AGS08} we can pointwise approximate $u_2(s,\cdot)$ from below with an increasing sequence of bounded Lipschitz functions $u_k$
\[
u_{2,k}(\mu) := \inf_{\nu\in\Pc_2(\R^d)}\big\{u_2(s,\nu)+k\Wc_2(\mu,\nu)\big\}, \, \text{with} \,
\begin{cases}
\inf u_2(s,\cdot) \leq u_{2,k}(\mu) \leq u_2(s,\mu) \leq \sup u_2(s,\cdot), \\
u_2(s,\mu)=\lim_{k\rightarrow\infty} u_{2,k}(\mu) = \sup_{k\in\N} u_{2,k}(\mu).
\end{cases}
\]
Then, we define
\[
v_k^s(t,\nu) \ := \ \E\bigg[\int_t^s f\big(r,X_r^{t,\xi,\mathfrak a},\P_{X_r^{t,\xi,\mathfrak a}},Y_r^{t,\mathfrak a}\big)\,dr\bigg] + u_{2,k}\big(\P_{X_s^{t,\xi,\mathfrak a}}\big).
\]
If we prove that $u_2(t,\mu)\geq v_k^s(t,\nu)$, for every $k\in\N$, then sending $k\rightarrow\infty$ we conclude that $u_2(t,\mu)\geq v^s(t,\nu)$. In what follow we suppose that $u_2(s,\cdot)$ is Lipschitz continuous and therefore we consider the function $v^s$ and we prove that $u_2(t,\mu)\geq v^s(t,\nu)$. This is not a loss of generality. As a matter of fact, if $u_2(s,\cdot)$ is not Lipschitz continuous we repeat the same arguments reported below to $v_k^s$ instead of $v^s$, therefore proving that $u_2(t,\mu)\geq v_k^s(t,\nu)$, for every $k\in\N$. As already noticed, from the arbitrariness of $k$, we conclude that $u_2(t,\mu)\geq v^s(t,\nu)$.}\\

\noindent Now, let us prove $u_2(t,\mu)\geq v^s(t,\nu)$, for every $\underline t\in[0,T)$, $s\in(\underline t,T]$, $(t,\nu)\in[\underline t,s]\times\Pc_2(\R^d\times A)$, with $\mu$ being the first marginal of $\nu$. We proceed by contradiction and suppose that there exist $\underline t_0\in[0,T)$, $s_0\in(\underline t_0,T]$, $(t_0,\mu_0,\nu_0)\in[\underline t_0,s_0)\times\Pc_2(\R^d)\times\Pc_2(\R^d\times A)$, with $\mu_0$ being the first marginal of $\nu_0$, such that
\[
v^{s_0}(t_0,\nu_0) \ > \ u_2(t_0,\mu_0).
\]
As in \textsc{Step I}, we can suppose that there exists some $q>2$ such that $\nu_0\in\Pc_q(\R^d)$.

For every $n,m\in\N$, let $v_{n,m}^{s_0}$ be the map given by \eqref{v_n,m^s,a}. Now, we define $\check u_2(t,\mu):=\text{e}^{t-t_0} u_2(t,\mu)$, for every $(t,\mu)\in[\underline t_0,s_0]\times\Pc_2(\R^d)$, and similarly $\check v_{n,m}^{s_0}$, $\check f_{n,m}^i$, $\check f$
from $v_{n,m}^{s_0}$, $\tilde f_{n,m}^i$, $f$, respectively. We also define $\check g(x,\mu):=\text{e}^{T-t_0} g(x,\mu)$ and $\check u_{n,m}(s_0,\mu):=\text{e}^{s_0-t_0} u_{n,m}(s_0,\mu)$, for every $(x,\mu)\in\R^d\times\Pc_2(\R^d)$. We observe that, given $\mathfrak a_0\in\Mc_t$ with distribution being equal to the marginal of $\nu_0$ on $A$, $\check u_2$ is a viscosity supersolution of the following equation (see e.g. \cite[Section 7]{CKGPR20}):
\[
\begin{cases}
\vspace{2mm}
\displaystyle\partial_t \check u_2(t,\mu) + \E\bigg\{\check f(t,\xi,\mu,\mathfrak a_0) + \dfrac{1}{2}\text{tr}\big[(\sigma\sigma\trans)(t,\xi,\mathfrak a_0)\partial_x\partial_\mu \check u_2(t,\mu)(\xi)\big] \\
\displaystyle\vspace{2mm}+\,\big\langle b(t,\xi,\mu,\mathfrak a_0),\partial_\mu \check u_2(t,\mu)(\xi)\rangle\bigg\} = \check u_2(t,\mu), &\hspace{-3cm}(t,\mu)\in[\underline t_0,s_0)\times \Pc_2(\R^d), \\
\displaystyle \check u_2(s_0,\mu) = \check u_2(s_0,\mu), &\hspace{-3cm}\,\mu\in\Pc_2(\R^d),
\end{cases}
\]
for any $\xi\in L^2(\Omega,\Fc_t,\P;\R^d)$, with $\P_\xi=\mu$. Moreover, by Theorem \ref{T:SmoothApprox2} we deduce that $\check v_{n,m}^{s_0}$ solves the following equation:
\[
\hspace{-5mm}\begin{cases}
\vspace{2mm}
\displaystyle\partial_t \check v_{n,m}^{s_0}(t,\nu) + \bar\E\sum_{i=1}^n\bigg\{\frac{1}{n} \check f_{n,m}^i(t,\xi_1,\ldots,\xi_n,\mathfrak a_0^i) + \langle \tilde b_{n,m}^i(t,\xi_1,\ldots,\xi_n,\mathfrak a_0^i),\partial_{x_i}\check{\bar v}_{n,m}^{s_0}(t,\bar\xi,\bar{\mathfrak a}_0)\rangle \\
\displaystyle\vspace{2mm}+\,\frac{1}{2}\textup{tr}\Big[(\sigma\sigma\trans)(t,\xi_i,\mathfrak a_0^i)\partial_{x_ix_i}^2\check{\bar v}_{n,m}^{s_0}(t,\bar\xi,\bar{\mathfrak a}_0)\Big]\bigg\} = \check v_{n,m}^{s_0}(t,\nu), &\hspace{-5.7cm}(t,\nu)\in[\underline t_0,s_0)\times \Pc_2(\R^d\times A), \\
\displaystyle \check v_{n,m}^{s_0}(s_0,\nu) = \bar\E\big[\check u_{n,m}(s_0,\bar\xi)\big], &\hspace{-5.05cm}\nu\in\Pc_2(\R^d\times A),
\end{cases}
\]
for any $\bar\xi=(\xi_1,\ldots,\xi_n)\in L^2(\Omega,\Fc_t,\P;\R^{dn})$ and $\bar{\mathfrak a}_0=(\mathfrak a_0^1,\ldots,\mathfrak a_0^n)$, with $\mathfrak a_0^i\in\Mc_t$, such that $\bar\P_{(\bar\xi,\bar{\mathfrak a}_0)}=\nu\otimes\cdots\otimes\nu$, where $\check{\bar v}_{n,m}^{s_0}(t,\bar x,\bar a):=\textup{e}^{t-t_0}\bar v_{n,m}^{s_0}(t,\bar x,\bar a)$, for every $(t,\bar x,\bar a)\in[\underline t_0,s_0]\times(\R^d\times A)^n$, with $\bar v_{n,m}^{s_0}$ being the same function appearing in Theorem \ref{T:SmoothApprox2}.

\textcolor{black}{In the sequel it is useful to see at $u_2$ as a function on $[0,T]\times\Pc_2(\R^d\times A)$ rather than $[0,T]\times\Pc_2(\R^d)$. In other words, it is useful to consider the function
\[
\tilde u_2(t,\nu) \ := \ u_2(t,\mu), \qquad \forall\,(t,\nu)\in[0,T]\times\Pc_2(\R^d),
\]
with $\mu$ being the first marginal of $\nu$. To avoid introducing additional notations, we denote $\tilde u_2$ still by $u_2$.}

Now, notice that $v_{n,m}^{s_0}$ is bounded by a constant independent of $n,m$. As a consequence, there exists $\lambda\geq0$, independent of $n,m$, satisfying
\begin{equation}\label{lambda_II}
\sup_{[0,T]\times\Pc_2(\R^d\times A)} (\check v_{n,m}^{s_0}-\check u_2) \ \leq \ (\check v_{n,m}^{s_0}-\check u_2)(t_0,\nu_0) + \lambda.
\end{equation}
Since $\check v_{n,m}^{s_0}-\check u_2$ is bounded and continuous, by \eqref{lambda_II} and Theorem \ref{T:SmoothVarPrinc} applied on $[\underline t_0,s_0]\times\Pc_2(\R^d\times A)$ with $G=\check v_{n,m}^{s_0}-\check u_2$, we obtain that for every $\delta>0$ there exist $\{(t_k,\nu_k)\}_{k\geq1}\subset[\underline t_0,s_0]\times\Pc_2(\R^d\times A)$ converging to some $(\tilde t,\tilde\nu)\in[\underline t_0,s_0]\times\Pc_2(\R^d\times A)$ and $\varphi_\delta$ such that items (i)-(ii)-(iii) and 1)-2)-3) of Theorem \ref{T:SmoothVarPrinc} hold.

Now, as in the proof of \textsc{Step I} we distinguish two cases. If $\tilde t=s_0$ and, in addition, $s_0=T$, then we proceed as in \textsc{Substep I-c} to get a contradiction. On the other hand, if $s_0<T$, then we proceed as in \textsc{Substep I-d} in order to find a contradiction and conclude the proof.\end{proof}

\begin{Corollary}[Uniqueness]\label{C:Uniqueness}
Let Assumptions \ref{AssA}, \ref{AssB}, \ref{AssC}, \ref{AssD} hold. Then, the value function $v$, given by \eqref{Value}, is the unique bounded and continuous viscosity solution of equation \eqref{HJB}.
\end{Corollary}
\begin{proof}[\textbf{Proof.}]
From Proposition \ref{P:ValueFunction} and Theorem \ref{T:Exist} we know that $v$ is bounded, continuous, and it is a viscosity solution of equation \eqref{HJB}. Now, let $u$ be another bounded and continuous viscosity solution of equation \eqref{HJB}. Then, by Theorem \ref{T:Comparison} we deduce that $u\leq v$ and $v\leq u$ (in fact, both $v$ and $u$ are viscosity sub/supersolution of equation \eqref{HJB}), from which we conclude that $v\equiv u$.
\end{proof}

\appendix

\section{Smooth finite-dimensional approximations of the value function}
\label{S:AppApprox}

\subsection{Mean field control problem on a different probabilistic setting and approximation by non-degenerate control problems}
\label{SubS:A1}

In the present appendix we formulate the mean field control problem on a different probabilistic setting, supporting an independent $d$-dimensional Brownian motion $\hat W$.\\
Let $(\hat\Omega,\hat\Fc,\hat\P)$ be a complete probability space on which a $m$-dimensional Brownian motion $\hat B=(\hat B_t)_{t\geq0}$ and a $d$-dimensional Brownian motion $\hat W=(\hat W_t)_{t\geq0}$ are defined, with $\hat B$ and $\hat W$ being independent. We denote by $\hat\F^{B,W}=(\hat\Fc_t^{B,W})_{t\geq0}$ the $\hat\P$-completion of the filtration generated by $\hat B$ and $\hat W$. We also assume that there exists a sub-$\sigma$-algebra $\hat\Gc$ of $\hat\Fc$ satisfying the following properties.
\begin{enumerate}[i)]
\item $\hat\Gc$ and $\hat\Fc_\infty^{B,W}$ are independent.
\item $\Pc_2(\R^d)=\{\P_{\hat\xi}$ such that $\hat\xi\colon\hat\Omega\rightarrow\R^d,$ with $\hat\xi$ being $\hat\Gc$-measurable and $\hat\E|\hat\xi|^2<\infty\}$.
\end{enumerate}
We denote by $\hat\F=(\hat\Fc_t)_{t\geq0}$ the $\hat \P$-completed filtration of $(\hat\Gc\vee\hat\Fc_t^{B,W})_{t\geq0}$, for all $t\geq0$. Finally, we denote by $\hat\Ac$ the set of control processes, namely the family of all $\hat\F$-progressively measurable processes $\hat\alpha\colon[0,T]\times\hat\Omega\rightarrow A$.

\noindent Now, for every $\eps\geq0$, $t\in[0,T]$, $\hat\xi\in L^2(\hat\Omega,\hat\Fc_t,\hat\P;\R^d)$, $\hat\alpha\in\hat\Ac$, let $\hat X^{\eps,t,\hat\xi,\hat\alpha}=(\hat X_s^{\eps,t,\hat\xi,\hat\alpha})_{s\in[t,T]}$ be the unique solution to the following controlled McKean-Vlasov stochastic differential equation:
\[
\hat X_s \ = \ \hat \xi + \int_t^s b\big(r,\hat X_r,\P_{\hat X_r},\hat \alpha_r\big)\,dr + \int_t^s \textcolor{black}{\sigma(r,\hat X_r,\hat \alpha_r)}\,d\hat B_r + \eps\,(\hat W_s - \hat W_t), \quad \forall\,s\in[t,T].
\]
Moreover, consider the lifted value function
\[
V_\eps(t,\hat\xi) \ = \ \sup_{\hat\alpha\in\hat\Ac} \hat\E\bigg[\int_t^s f\big(r,\hat X_r^{\eps,t,\hat\xi,\hat\alpha},\P_{\hat X_r^{\eps,t,\hat\xi,\hat\alpha}},\hat\alpha_r\big)\,dr + g\big(\hat X_T^{\eps,t,\hat\xi,\hat\alpha},\P_{\hat X_T^{\eps,t,\hat\xi,\hat\alpha}}\big)\bigg],
\]
for every $t\in[0,T]$, $\hat\xi\in L^2(\hat\Omega,\hat\Fc_t,\hat\P;\R^d)$. Under Assumption \ref{AssA}, from Theorem \ref{T:LawInv} applied in the present probabilistic setting, with $\sigma$ and $B$ replaced respectively by $(\sigma,\eps I_d)$ and $(\hat B,\hat W)$, we know that $V_\eps$ satisfies the law invariance property. Therefore we can define the value function $v_\eps\colon[0,T]\times\Pc_2(\R^d)\rightarrow\R$ as follows:
\begin{equation}\label{v_eps}
v_\eps(t,\mu) \ = \ V_\eps(t,\hat\xi),
\end{equation}
for every $(t,\mu)\in[0,T]\times\Pc_2(\R^d)$ and any $\hat\xi\in L^2(\hat\Omega,\hat\Fc_t,\hat\P;\R^d)$ such that $\P_{\hat\xi}=\mu$. Moreover, applying Proposition \ref{P:ValueFunction} in the present probabilistic setting, it follows that $v_\eps$ is bounded, jointly continuous on $[0,T]\times\Pc_2(\R^d)$ and Lipschitz continuous in the measure: there exists  $L>0$ such that
\begin{align*}
 |v_\eps(t,\mu)-v_\eps(t',\mu')|\ \leq \ L \Wc_2(\mu,\mu'),
\end{align*}
for any $t\in[0,T]$ and $\mu,\mu'\in\Pc_2(\R^d)$.

\begin{Remark}\label{R:v_0}
Notice that, under Assumption \ref{AssA}, it is not immediately clear if $v_0\equiv v$. However, under Assumptions \ref{AssA} and \ref{AssB}, applying Theorem \ref{T:Exist} in the present probabilistic setting we deduce that $v_0$ is a viscosity solution of the Master Bellman equation \eqref{HJB}. As a consequence, under Assumptions \ref{AssA}-\ref{AssB}-\ref{AssC}-\ref{AssD}, by Corollary \ref{C:Uniqueness} we conclude that $v_0\equiv v$.
\end{Remark}

\begin{Lemma}\label{L:eps}
Suppose that Assumption \ref{AssA} holds. Then, there exists a constant $C_{K,T}\geq0$, depending only on $K$ and $T$, such that, for every $\eps\geq0$,
\[
|v_\eps(t,\mu) - v_0(t,\mu)| \ \leq \ C_{K,T}\,\eps,
\]
for every $(t,\mu)\in[0,T]\times\Pc_2(\R^d)$.
\end{Lemma}
\begin{proof}[\textbf{Proof.}]
By usual calculations (as in \cite[Theorem 2.5.9]{Krylov80}), we obtain
\begin{equation}\label{Estimate_eps}
\hat\E\bigg[\sup_{t\leq s\leq T}\big|\hat X_s^{\eps,t,\hat\xi,\hat\alpha} - \hat X_s^{0,t,\hat\xi,\hat\alpha}\big|^2\bigg] \ \leq \ C_KT\,\textup{e}^{C_KT}\,\eps^2,
\end{equation}
for every $\eps\geq0$, $t\in[0,T]$, $\hat\xi\in L^2(\hat\Omega,\hat\Fc_t,\hat\P;\R^d)$, $\hat\alpha\in\hat\Ac$, for some constant $C_K\geq0$, depending only on $K$. Then, we have
\begin{align*}
|v_\eps(t,\mu) - v_0(t,\mu)| &\leq \sup_{\hat\alpha\in\hat\Ac}\hat\E\bigg[\int_t^T \big|f\big(s,\hat X_s^{\eps,t,\hat\xi,\hat\alpha},\P_{\hat X_s^{\eps,t,\hat\xi,\hat\alpha}},\hat\alpha_s\big) - f\big(s,\hat X_s^{0,t,\hat\xi,\hat\alpha},\P_{\hat X_s^{0,t,\hat\xi,\hat\alpha}},\hat\alpha_s\big)\big|\,ds \\
&\quad + \big|g\big(\hat X_T^{\eps,t,\hat\xi,\hat\alpha},\P_{\hat X_T^{\eps,t,\hat\xi,\hat\alpha}}\big) - g\big(\hat X_T^{0,t,\hat\xi,\hat\alpha},\P_{\hat X_T^{0,t,\hat\xi,\hat\alpha}}\big)\big|\bigg] \\
&\leq \ K\sup_{\hat\alpha\in\hat\Ac}\bigg\{\int_t^T \!\! \Big\{\hat\E\big[\big|\hat X_s^{\eps,t,\hat\xi,\hat\alpha} - \hat X_s^{0,t,\hat\xi,\hat\alpha}\big|\big] + \Wc_2\big(\P_{\hat X_s^{\eps,t,\hat\xi,\hat\alpha}},\P_{\hat X_s^{0,t,\hat\xi,\hat\alpha}}\big)\Big\}ds \\
&\quad + \hat\E\big[\big|\hat X_T^{\eps,t,\hat\xi,\hat\alpha} - \hat X_T^{0,t,\hat\xi,\hat\alpha}\big|\big] + \Wc_2\big(\P_{\hat X_T^{\eps,t,\hat\xi,\hat\alpha}},\P_{\hat X_T^{0,t,\hat\xi,\hat\alpha}}\big)\bigg\} \\
&\leq K\sup_{\hat\alpha\in\hat\Ac}\bigg\{\int_t^T \!\! \Big\{\hat\E\big[\big|\hat X_s^{\eps,t,\hat\xi,\hat\alpha} - \hat X_s^{0,t,\hat\xi,\hat\alpha}\big|^2\big]^{1/2} + \E\big[\big|\hat X_s^{\eps,t,\hat\xi,\hat\alpha} - \hat X_s^{0,t,\hat\xi,\hat\alpha}\big|^2\big]^{1/2}\Big\}ds \\
&\quad + \hat\E\big[\big|\hat X_T^{\eps,t,\hat\xi,\hat\alpha} - \hat X_T^{0,t,\hat\xi,\hat\alpha}\big|^2\big]^{1/2} + \E\big[\big|\hat X_T^{\eps,t,\hat\xi,\hat\alpha} - \hat X_s^{0,t,\hat\xi,\hat\alpha}\big|^2\big]^{1/2}\bigg\} \\
&\leq \ 2K(T+1)\,\sqrt{C_KT\,\textup{e}^{C_KT}}\,\eps,
\end{align*}
where the last inequality follows from estimate \eqref{Estimate_eps}.
\end{proof}

\subsection{Cooperative $n$-player stochastic differential game and\\ propagation of chaos result}
\label{SubS:CooperativeGame}

Let $n\in\N$ and let $(\bar\Omega,\bar\Fc,\bar\P)$ be a complete probability space, supporting independent Brownian motions $\bar B^1,\ldots,\bar B^n,\bar W^1,\ldots,\bar W^n$, with $\bar B^i$ (resp. $\bar W^i$) being $m$-dimensional (resp. $d$-dimensional). Let also $\bar\F^{B,W}=(\bar\Fc_t^{B,W})_{t\geq0}$ denote the $\bar\P$-completion of the filtration generated by $\bar B$ and $\bar W$, with $\bar B=(\bar B^1,\ldots,\bar B^n)$ and $\bar W=(\bar W^1,\ldots,\bar W^n)$. Moreover, let $\bar\Gc$ be a sub-$\sigma$-algebra of $\bar\Fc$ satisfying the following properties.
\begin{enumerate}[i)]
\item $\bar\Gc$ and $\bar\Fc_\infty^{B,W}$ are independent.
\item $\Pc_2(\R^d)=\{\P_{\bar\xi}$ such that $\bar\xi\colon\bar\Omega\rightarrow\R^d,$ with $\bar\xi$ being $\bar\Gc$-measurable and $\bar\E|\bar\xi|^2<\infty\}$.
\end{enumerate}
Furthermore, let $\bar\F=(\bar\Fc_t)_{t\geq0}$ be given by $\bar\Fc_t:=\bar\Gc\vee\bar\Fc_t^{B,W}$, for every $t\geq0$. Finally, let $\bar\Ac^n$ be the family of all $\bar\F$-progressively measurable processes $\bar\alpha=(\bar\alpha^1,\ldots,\bar\alpha^n)\colon[0,T]\times\bar\Omega\rightarrow A^n$. Now, for every $\eps>0$, $t\in[0,T]$, $\bar\alpha\in\bar\Ac^n$, $\bar\xi^1,\ldots,\bar\xi^n\in L^2(\bar\Omega,\bar\Fc_t,\bar\P;\R^d)$, with $\bar\xi:=(\bar\xi^1,\ldots,\bar\xi^n)$, let $\bar X^{\eps,t,\bar\xi,\bar\alpha}=(\bar X^{1,\eps,t,\bar\xi,\bar\alpha},\ldots,\bar X^{n,\eps,t,\bar\xi,\bar\alpha})$ be the unique solution to the following system of controlled stochastic differential equations:
\begin{equation}\label{SDE_bar_X}
\bar X_s^i \ = \ \bar\xi^i + \int_t^s b(r,\bar X_r^i,\widehat\mu_r^n,\bar\alpha_r^i)\,dr + \int_t^s \textcolor{black}{\sigma(r,\bar X_r^i,\bar\alpha_r^i)}\,dB_r^i + \eps\,(\bar W_s^i - \bar W_t^i), \qquad \forall\,s\in[t,T],
\end{equation}
for $i=1,\ldots,n$, with
\[
\widehat\mu_r^n \ = \ \frac{1}{n} \sum_{j=1}^n \delta_{\bar X_r^j}, \qquad \forall\,r\in[t,T].
\]
We denote $\widehat\mu_r^{n,\eps,t,\bar\xi,\bar\alpha}=\frac{1}{n} \sum_{j=1}^n \delta_{\bar X_r^{j,\eps,t,\bar\xi,\bar\alpha}}$. 
We consider the cooperative $n$-players game where
a planner maximizes, over $\bar\alpha\in\bar\Ac^n$,
the payoff
\begin{equation*}\label{payofftilde_v_eps,n}
\tilde J_{\eps,n}(t,\bar\mu;\bar\alpha):=
\frac{1}{n}\sum_{i=1}^n\bar\E\bigg[\int_t^T f\big(s,\bar X_s^{i,\eps,t,\bar\xi,\bar\alpha},\widehat\mu_s^{n,\eps,t,\bar\xi,\bar\alpha},\bar\alpha_s^i\big)\,ds + g\big(\bar X_T^{i,\eps,t,\bar\xi,\bar\alpha},\widehat\mu_T^{n,\eps,t,\bar\xi,\bar\alpha}\big)\bigg],
\end{equation*}
Then, the value function $\tilde v_{\eps,n}\colon[0,T]\times\Pc_2(\R^{dn})\rightarrow\R$ of such  cooperative $n$-player game is given by
\begin{equation}\label{tilde_v_eps,n}
\tilde v_{\eps,n}(t,\bar\mu) \ = \ \sup_{\bar\alpha\in\bar\Ac^n}\frac{1}{n}\sum_{i=1}^n\bar\E\bigg[\int_t^T f\big(s,\bar X_s^{i,\eps,t,\bar\xi,\bar\alpha},\widehat\mu_s^{n,\eps,t,\bar\xi,\bar\alpha},\bar\alpha_s^i\big)\,ds + g\big(\bar X_T^{i,\eps,t,\bar\xi,\bar\alpha},\widehat\mu_T^{n,\eps,t,\bar\xi,\bar\alpha}\big)\bigg],
\end{equation}
for every $t\in[0,T]$, $\bar\mu\in\Pc_2(\R^{dn})$, with $\bar\xi\in L^2(\bar\Omega,\bar\Fc_t,\bar\P;\R^{dn})$ such that $\P_{\bar\xi}=\bar\mu$.\\
We also introduce the following approximation of the value function: we call
$$\tilde v_{\eps,n,m}\colon[0,T]\times\Pc_2(\R^{dn})\rightarrow\R$$ the map given by 
\begin{align*}
\tilde v_{\eps,n,m}(t,\bar\mu) \ = \ \sup_{\bar\alpha\in\bar\Ac^n}\frac{1}{n}\sum_{i=1}^n\bar\E\bigg[\int_t^T f_{n,m}^i\big(s,\textcolor{black}{\bar X_s^{1,\textcolor{black}{m},\eps,t,\bar\xi,\bar\alpha},\ldots,\bar X_s^{n,\textcolor{black}{m},\eps,t,\bar\xi,\bar\alpha}},\bar\alpha_s^i\big)\,ds \\
+ \; g_{n,m}^i\big(\textcolor{black}{\bar X_T^{1,\textcolor{black}{m},\eps,t,\bar\xi,\bar\alpha},\ldots,\bar X_T^{n,\textcolor{black}{m},\eps,t,\bar\xi,\bar\alpha}}\big)&\bigg], \notag
\end{align*}
\textcolor{black}{where $\bar X^{i,\textcolor{black}{m},\eps,t,\bar\xi,\bar\alpha}$ solves equation \eqref{SDE_bar_X} with $b$ replaced by $b_{n,m}^i$, where
$$
b_{n,m}^i\colon[0,T]\times\R^{dn}\times A\rightarrow\R^d,
\qquad
f_{n,m}^i\colon[0,T]\times\R^{dn}\times A\rightarrow\R,
\qquad
g_{n,m}^i\colon\R^{dn}\rightarrow\R
$$ are smooth approximations of $b$, $f$, $g$ defined as follows:}
\begin{align*}
\textcolor{black}{b_{n,m}^i(t,\bar x,a)} \ &\textcolor{black}{=} \ \textcolor{black}{m^{nd\textcolor{black}{+1}}
\int_{\R^{dn\textcolor{black}{+1}}}b\bigg(\textcolor{black}
{T\wedge(t-s)^+},x_i-y_i,{1\over n}
\sum_{j=1}^n \delta_{x_j-y_j},a\bigg)\textcolor{black}{\zeta(ms)}\prod_{j=1}^{n}\Phi(my_j)dy_j\textcolor{black}{ds},}\\
f_{n,m}^i(t,\bar x,a) \ &= \ m^{nd\textcolor{black}{+1}}\int_{\R^{dn\textcolor{black}{+1}}}f\bigg(\textcolor{black}{T\wedge(t-s)^+},x_i-y_i,{1\over n}\sum_{j=1}^n \delta_{x_j-y_j},a\bigg)\textcolor{black}{\zeta(ms)}\prod_{j=1}^{n}\Phi(my_j)dy_j\textcolor{black}{ds},\\
g_{n,m}^i(\bar x) \ &= \ m^{nd}\int_{\R^{dn}}g\bigg(x_i-y_i,{1\over n}\sum_{j=1}^n \delta_{x_j-y_j}\bigg)\prod_{j=1}^{n}\Phi(my_j)dy_j,
\end{align*}
for all $n,m\in\N$, $i=1,\ldots,n$, $\bar x=(x_1,\ldots,x_n)\in\R^{dn}$, $(t,a)\in[0,T]\times A$, with $\Phi\colon\R^d\rightarrow[0,+\infty)$ \textcolor{black}{and $\zeta\colon\R\rightarrow[0,+\infty)$} being $C^\infty$ functions with compact support satisfying $\int_{\R^d}\Phi(y)\,dy=1$ \textcolor{black}{and $\int_{-\infty}^{+\infty}\zeta(s)\,ds=1$}.
 
\begin{Lemma}
\label{lm:conv}
Suppose that Assumptions \ref{AssA} and \ref{AssB} hold. Let $\widehat\mu^{n,\bar x}$ be given by
\[
\widehat\mu^{n,\bar x} \ := \ \frac{1}{n} \sum_{j=1}^n \delta_{x_j}
\]
Let $i=1,\dots,n$.
We have
$$
\lim_{m\to+\infty}{b_{n,m}^i(t,\bar x,a)}
=b\left(t,x_i,\widehat\mu^{n,\bar x}\right),
\qquad
\lim_{m\to+\infty}{f_{n,m}^i(t,\bar x,a)}
=f\left(t,x_i,\widehat\mu^{n,\bar x}\right),
$$
uniformly for $(t,\bar x,a)$ in $[0,T]\times \R^{dn}\times A$.
Moreover,
$$
\lim_{m\to+\infty}{g_{n,m}^i(\bar x)}
=g\left(x_i,\widehat\mu^{n,\bar x}\right),
$$
uniformly for $\bar x$ in $\R^{dn}$.
\\
Furthermore we have the following estimates
\begin{align*}
\big|b\left(t,x_i,\widehat\mu^{n,\bar x},a\right) - b_{n,m}^i(t,\bar x,a)\big|
&\leq \
Km\int_\R \left|t-\left(T\wedge(t-s)^+\right)\right|^\beta
{\zeta(ms)}ds
\notag
\\
&+
K\,m^{nd}\int_{\R^{dn}}\bigg(|y_i| + \frac{1}{n}\sum_{j=1}^n |y_j|\bigg)\prod_{j=1}^{n}\Phi(my_j)dy_j
\end{align*}
\begin{align*}
\big|f\left(t,x_i,\widehat\mu^{n,\bar x},a\right) - f_{n,m}^i(t,\bar x,a)\big|
&\leq \
Km\int_\R \left|t-\left(T\wedge(t-s)^+\right)\right|^\beta
{\zeta(ms)}ds
\notag
\\
&+
K\,m^{nd}\int_{\R^{dn}}\bigg(|y_i| + \frac{1}{n}\sum_{j=1}^n |y_j|\bigg)\prod_{j=1}^{n}\Phi(my_j)dy_j
\end{align*}
\begin{equation}\label{eq:limnewg}
\big|g(x_i,\widehat\mu^{n,\bar x}) - g_{n,m}^i(x_1,\ldots,x_n)\big|
\leq \ K\,m^{nd}\int_{\R^{dn}}\bigg(|y_i| + \frac{1}{n}\sum_{j=1}^n |y_j|\bigg)\prod_{j=1}^{n}\Phi(my_j)dy_j
\end{equation}
Finally
\begin{align}
&
\big|b_{n,m}^i(t,\bar x,a) - b_{n,m}^i(t,\bar z,a)\big|
\vee
\big|f_{n,m}^i(t,\bar x,a) - f_{n,m}^i(t,\bar z,a)\big|
\vee
\big|g_{n,m}^i(\bar x) - g_{n,m}^i(\bar z)\big|
\notag\\
& \ \leq \ K\left[|x_i-z_i| + \frac{1}{n}\sum_{j=1}^n |x_j-z_j|\right].
\label{eq:Lipbfg}
\end{align}
\end{Lemma}
\begin{proof}[\textbf{Proof.}]
We first prove the claims for $g$.
From the definition of $g_{n,m}^i$ we get
\begin{align*}
&\big|g(x_i,\widehat\mu^{n,\bar x}) - g_{n,m}^i(x_1,\ldots,x_n)\big| \notag \\
&\leq \ m^{nd}\int_{\R^{dn}}\bigg|g(x_i,\widehat\mu^{n,\bar x}) - g\bigg(x_i-y_i,{1\over n}\sum_{j=1}^n \delta_{x_j-y_j}\bigg)\bigg|\prod_{j=1}^{n}\Phi(my_j)dy_j \notag \\
&\leq \ K\,m^{nd}\int_{\R^{dn}}\bigg(|y_i| + \frac{1}{n}\sum_{j=1}^n |y_j|\bigg)\prod_{j=1}^{n}\Phi(my_j)dy_j,
\end{align*}
where the last inequality follows from the Lipschitz continuity of $g$, and also from the following:
\[
\Wc_2\bigg(\widehat\mu^{n,\bar x},{1\over n}\sum_{j=1}^n \delta_{x_j-y_j}\bigg) \ = \ \Wc_2\bigg({1\over n}\sum_{j=1}^n \delta_{x_j},{1\over n}\sum_{j=1}^n \delta_{x_j-y_j}\bigg) \ \leq \ \frac{1}{n} \sum_{j=1}^n |y_j|,
\]
where the last inequality follows from the definition of $\Wc_2$ (see \eqref{Wasserstein}) taking the probability measure $\pi$ on $\R^d\times\R^d$ satisfying $\pi(\{(x_j,x_j-y_j)\})=\frac{1}{n}$, $\forall\,j=1,\ldots,n$.
\\
For the other claim on $g$ we use that
\begin{align*}
&\big|g_{n,m}^i(\bar x) - g_{n,m}^i(\bar z)\big| \ \leq \ m^{nd}\int_{\R^{dn}}\bigg|g\bigg(x_i-y_i,{1\over n}\sum_{j=1}^n \delta_{x_j-y_j}\bigg) \\
&- g\bigg(z_i-y_i,{1\over n}\sum_{j=1}^n \delta_{z_j-y_j}\bigg)\bigg|\prod_{j=1}^{n}\Phi(my_j)dy_j \ \leq \
K\left[|x_i-z_i| + \frac{1}{n}\sum_{j=1}^n |x_j-z_j|\right],
\end{align*}
for every $i=1,\ldots,n$, where the last inequality follows from the Lipschitz continuity of $g$, and also from the following:
\[
\Wc_2\bigg(\frac{1}{n}\sum_{j=1}^n \delta_{x_j-y_j},\frac{1}{n}\sum_{j=1}^n \delta_{z_j-y_j}\bigg) \ \leq \ \frac{1}{n} \sum_{j=1}^n |x_j-z_j|,
\]
which follows from the definition of $\Wc_2$ (see \eqref{Wasserstein}) taking the probability measure $\pi$ on $\R^d\times\R^d$ such that $\pi(\{(x_j-y_j,z_j-y_j)\})=\frac{1}{n}$, $\forall\,j=1,\ldots,n$.
\\
Now we prove the claims for $b$ (the ones for $f$ are proved exactly in the same way)
\begin{align*}
&\big|b(t,x_i,\widehat\mu^{n,\bar x},a)
- b_{n,m}^i(t,\bar x,a)\big| \ \leq \ m^{nd+1}\int_{\R^{dn+1}}
\bigg|b(t,x_i,\widehat\mu^{n,\bar x},a) \\
&- b\bigg(T\wedge(t-s)^+,x_i-y_i,{1\over n}\sum_{j=1}^n \delta_{x_j-y_j},a\bigg)\bigg|\zeta(ms)\prod_{j=1}^{n}
\Phi(my_j)dy_jds \ \leq \ m\int_\R
\bigg|b(t,x_i,\widehat\mu^{n,\bar x},a) \\
&- b\bigg(T\wedge(t-s)^+,x_i,\widehat\mu^{n,\bar x},a\bigg)\bigg|\zeta(ms)ds + m^{nd}\int_{\R^{dn}}
\bigg|b(T\wedge(t-s)^+,x_i,\widehat\mu^{n,\bar x},a) \\
&- b\bigg(T\wedge(t-s)^+,x_i-y_i,{1\over n}\sum_{j=1}^n \delta_{x_j-y_j},a\bigg)\bigg|\zeta(ms)\prod_{j=1}^{n}
\Phi(my_j)dy_j ds.
\end{align*}
Thanks to the Lipschitz property of $b$ (see Assumption \ref{AssA}-(ii)) the second integral is estimated as in the case of $g$ considered above.
Moreover, the first integral, thanks to the Assumption \ref{AssB} is estimated by $K|t-(T\wedge(t-s)^+)|^\beta$.
Finally, the Lipschitz estimates for $b$ and $f$ are proved exactly in the same way as for $g$ using Assumption \ref{AssA} (ii)-(iii).
 \end{proof}

\begin{Remark}
Notice that it is not a priori clear the fact that the right-hand side of \eqref{tilde_v_eps,n} depends on $\bar\xi$ only through its law $\bar\mu$. However, as the cooperative $n$-player game is an example of mean field control problem (indeed, it is a standard stochastic optimal control problem) we can apply the results of Section \ref{S:MKV_Control_Pb} to it. In particular, from Theorem \ref{T:LawInv} we deduce the law invariance property, which explains why we can consider the value function $\tilde v_{\eps,n}$ (and, similarly, $\tilde v_{\eps,n,m}$), which depends only on $\bar\mu$ rather than on $\bar\xi$.
\end{Remark}

\noindent We also consider the functions $v_{\eps,n,m},\ v_{\eps,n}\colon[0,T]\times\Pc_2(\R^d)\rightarrow\R$ defined as
\begin{equation}\label{v_eps,n}
v_{\eps,n,m}(t,\mu) \ = \ \tilde v_{\eps,n,m}(t,\mu\otimes\cdots\otimes\mu) \qquad \text{ and } \qquad
v_{\eps,n}(t,\mu) \ = \ \tilde v_{\eps,n}(t,\mu\otimes\cdots\otimes\mu),
\end{equation}
for every $(t,\mu)\in[0,T]\times\Pc_2(\R^d)$.

We first show that the analogous of Lemma \ref{L:eps} holds for $v_{\eps,n,m}$.
\\
\begin{Lemma}\label{L:epsnm}
Suppose that Assumption \ref{AssA} holds. Then, there exists a constant $C_{K,T}\geq0$, depending only on $K$ and $T$, such that, for every $\eps\geq0$,
\[
|v_{\eps,n,m}(t,\mu) - v_{0,n,m}(t,\mu)| \ \leq \ C_{K,T}\,\eps,
\]
for every $(t,\mu)\in[0,T]\times\Pc_2(\R^d)$.
\end{Lemma}
\begin{proof}[\textbf{Proof.}]
The proof is similar to the one of Lemma \ref{L:eps}.
We provide a sketch for the reader convenience.
By usual calculations (as in \cite[Theorem 2.5.9]{Krylov80}), we obtain
\begin{equation}\label{Estimate_eps}
\bar\E\bigg[\sup_{t\leq s\leq T}\big|
\bar X_s^{i,m,\eps,t,\bar\xi,\bar\alpha}-
\bar X_s^{i,m,0,t,\bar\xi,\bar\alpha}\big|^2\bigg] \ \leq \ C_KT\,\textup{e}^{C_KT}\,\eps^2,
\end{equation}
for every $i=1,\dots,n$, $m\in \N$, $\eps\geq0$, $t\in[0,T]$,
$\bar\xi\in L^2(\bar\Omega,\bar\Fc_t,\bar\P;\R^{dn})$,
$\bar\alpha\in\bar\Ac^n$, for some constant $C_K\geq0$, depending only on $K$.
Then, we have, writing
$\bar X_s^{m,\eps,t,\bar\xi,\bar\alpha}$
for
$\big(\bar X_s^{1,m,\eps,t,\bar\xi,\bar\alpha},
\dots,\bar X_s^{n,m,\eps,t,\bar\xi,\bar\alpha}\big)$,
\begin{align*}
&|v_{\eps,n,m}(t,\mu) - v_{0,n,m}(t,\mu)| \ \leq \ \frac1n\sum_{i=1}^{n}\sup_{\bar\alpha\in\bar\Ac^n}
\bar\E\bigg[\int_t^T \big|
f_{n,m}^i\big(s,\bar X_s^{m,\eps,t,\bar\xi,\bar\alpha},
\bar\alpha^i_s\big) - f_{n,m}^i\big(s,\bar X_s^{m,0,t,\bar\xi,\bar\alpha},
,\bar\alpha^i_s\big)\big|\,ds \\
&
+ \big|
g_{n,m}^i\big(\bar X_T^{m,\eps,t,\bar\xi,\bar\alpha}
\big)
- g_{n,m}^i\big(\bar X_T^{m,0,t,\bar\xi,\bar\alpha}
\big)\big|\bigg] \ \leq \ \frac Kn\sum_{i=1}^{n}
\sup_{\bar\alpha\in\bar\Ac^n}\bigg\{\int_t^T \!\! \Big\{\bar\E\big[\big|\bar X_s^{m,\eps,t,\bar\xi,\bar\alpha} - \bar X_s^{m,0,t,\bar\xi,\bar\alpha}\big|\big]
\Big\}ds \\
&+ \bar\E\big[\big|\bar X_T^{m,\eps,t,\bar\xi,\bar\alpha} - \bar X_T^{0,t,\bar\xi,\bar\alpha}\big|\big]
\big)\bigg\} \ \leq \ \frac Kn\sum_{i=1}^{n}
\sup_{\bar\alpha\in\bar\Ac^n}\bigg\{\int_t^T \!\! \Big\{\bar\E\big[\big|\bar X_s^{m,\eps,t,\bar\xi,\bar\alpha} - \bar X_s^{m,0,t,\bar\xi,\bar\alpha}\big|^2\big]^{1/2} \\
&+ \E\big[\big|\bar X_s^{\eps,t,\bar\xi,\bar\alpha} - \bar X_s^{0,t,\bar\xi,\bar\alpha}\big|^2\big]^{1/2}\Big\}ds + \bar\E\big[\big|\bar X_T^{\eps,t,\bar\xi,\bar\alpha} - \bar X_T^{0,t,\bar\xi,\bar\alpha}\big|^2\big]^{1/2} + \E\big[\big|\bar X_T^{\eps,t,\bar\xi,\bar\alpha} - \bar X_s^{0,t,\bar\xi,\bar\alpha}\big|^2\big]^{1/2}\bigg\} \\
&\leq \ 2K(T+1)\,\sqrt{C_KT\,\textup{e}^{C_KT}}\,\eps,
\end{align*}
where the last inequality follows from estimate \eqref{Estimate_eps}.
\end{proof}
Now, we can state the following propagation of chaos result for $v_{\eps,n,m}$ (a more general propagation of chaos result holds for $\tilde v_{\eps,n,m}$, see \cite[Theorem 2.12]{Lacker}).

\begin{Theorem}\label{T:PropagChaos}
Suppose that  Assumptions \ref{AssA}, \ref{AssB}, \ref{AssC} hold. Let $\eps\geq0$ and $(t,\mu)\in\Pc_2(\R^d)$. If there exists $q>2$ such that $\mu\in\Pc_q(\R^d)$, then
\[
\lim_{n\rightarrow+\infty}\lim_{m\rightarrow+\infty} v_{\eps,n,m}(t,\mu) \ = \ v_\eps(t,\mu).
\]
\end{Theorem}
\begin{proof}[\textbf{Proof.}]
From the definitions of $b_{n,m}^i$, $f_{n,m}^i$, $g_{n,m}^i$ we get, through straightforward arguments, the convergence, a.s., when $m\to +\infty$, of
$\bar X_s^{i,m,\varepsilon,t,\bar \xi,\bar \alpha}$ to
$\bar X_s^{i,\varepsilon,t,\bar \xi,\bar \alpha}$.
This implies, using the definitions of $v_{\eps,n,m}$ and $v_{\eps,n}$,
\[
\lim_{m\rightarrow+\infty} v_{\eps,n,m}(t,\mu) \ = \ v_{\eps,n}(t,\mu).
\]
Then, the convergence
\[
\lim_{n\rightarrow+\infty} v_{\eps,n}(t,\mu) \ = \ v_{\eps}(t,\mu)
\]
is a consequence of \cite[Theorem 2.12]{Lacker}. \textcolor{black}{More precisely, for every $n\in\N$, $\eps>0$, denote
\[
J_{\eps,n}(t,\mu,\bar\alpha) \ = \ \frac{1}{n}\sum_{i=1}^n\bar\E\bigg[\int_t^T f\big(s,\bar X_s^{i,\eps,t,\bar\xi,\bar\alpha},\widehat\mu_s^{n,\eps,t,\bar\xi,\bar\alpha},\bar\alpha_s^i\big)\,ds + g\big(\bar X_T^{i,\eps,t,\bar\xi,\bar\alpha},\widehat\mu_T^{n,\eps,t,\bar\xi,\bar\alpha}\big)\bigg],
\]
for every $t\in[0,T]$, $\bar\alpha\in\bar\Ac^n$, $\mu\in\Pc_2(\R^d)$, with $\bar\xi\in L^2(\bar\Omega,\bar\Fc_t,\bar\P;\R^{dn})$ such that $\P_{\bar\xi}=\mu\otimes\cdots\otimes\mu$. Notice that
\[
v_{\eps,n}(t,\mu) \ = \ \sup_{\bar\alpha\in\bar\Ac^n}J_{\eps,n}(t,\mu,\bar\alpha), \qquad (t,\mu)\in[0,T]\times\Pc_2(\R^d).
\]
Now, by \cite[Theorem 2.12]{Lacker}, for every $n\in\N$ there exist $\epsilon_n\geq0$ and $\bar\alpha_n\in\bar\Ac^n$ such that $\lim_{n\rightarrow+\infty}\epsilon_n=0$ and $\bar\alpha_n$ is an $\epsilon_n$-optimal control for \eqref{tilde_v_eps,n}, namely it holds that
\begin{equation}\label{Lacker_Ineq}
J_{\eps,n}(t,\mu,\bar\alpha_n) \ \leq \ v_{\eps,n}(t,\mu) \ \leq \ J_{\eps,n}(t,\mu,\bar\alpha_n) + \epsilon_n, \qquad \forall\,n\in\N.
\end{equation}
In addition, by the beginning of Step 3 of \cite[Theorem 2.12]{Lacker} we have that $\{\bar\alpha_n\}_n$ is converging in a suitable way to some optimal relaxed control $m^*$, and also we have the convergence of the reward functionals: $J_{\eps,n}(t,\mu,\bar\alpha_n)\rightarrow v_\eps(t,\mu)$, where here we used that $v_\eps(t,\mu)$ coincides with the reward functional evaluated at the optimal relaxed control $m^*$, that is $v_\eps(t,\mu)$ is equal to the value obtained optimizing over relaxed controls, see \cite[Theorem 2.4]{Lacker}. Then, using the convergence $J_{\eps,n}(t,\mu,\bar\alpha_n)\rightarrow v_\eps(t,\mu)$, we see that the claim follows letting $n\rightarrow\infty$ in \eqref{Lacker_Ineq}.}
\end{proof}

\subsection{Smooth finite-dimensional approximations}
\label{SubS:A3}

We consider the same probabilistic setting as in Section \ref{SubS:CooperativeGame}.

\begin{Theorem}\label{T:SmoothApprox}
Suppose that Assumptions \ref{AssA}, \ref{AssB}, \ref{AssD} hold. Then, for every $\eps>0$, $n,m\in\N$, there exists $\bar v_{\eps,n,m}\colon[0,T]\times\R^{dn}\rightarrow\R$ such that
\begin{equation}\label{v_eps,n=bar_v_eps,n}
v_{\eps,n,m}(t,\mu) \ = \ \int_{\R^{dn}}\bar v_{\eps,n,m}(t,x_1,\ldots,x_n)\,\mu(dx_1)\cdots\mu(dx_n),
\end{equation}
for every $(t,\mu)\in[0,T]\times\Pc_2(\R^d)$, with $v_{\eps,n,m}$ given by \eqref{v_eps,n}, and the following holds.
\begin{enumerate}[\upshape 1)]
\item $\bar v_{\eps,n,m}\in C^{1,2}([0,T]\times\R^{dn})$ and $v_{\eps,n,m}\in C^{1,2}([0,T]\times\Pc_2(\R^d))$.
\item \textcolor{black}{For all $(t,\bar x)\in[0,T]\times\R^{dn}$, with $\bar x=(\bar x_1,\ldots,\bar x_\ell,\ldots,\bar x_{dn})=(x_1,\ldots,x_n)$ and $\bar x_\ell\in\R$, $x_i\in\R^d$, it holds that
\begin{align}
\big|\partial_{x_i}\bar v_{\eps,n,m}(t,\bar x)\big| \ &\leq \ \frac{C_K}{n},\label{Estimate_1stDeriv_n} \\
- C_{n,m} \ \leq \ \partial_{\bar x_\ell\bar x_h}\bar v_{\eps,n,m}(t,\bar x) \ &\leq \ \frac{1}{\eps^2}C_{n,m}, \label{Estimate_2ndDeriv}
\end{align}
for every $i=1,\ldots,n$, $\ell,h=1,\ldots,dn$, for some constants $C_K\geq0$ and $C_{n,m}\geq0$, with $C_K$ $($resp. $C_{n,m}$$)$ possibly depending on $K$ $($resp. $K,n,m$$)$, but independent of $\eps,n,m$ $($resp. $\eps$$)$, where $K$ is as in Assumption \ref{AssA}.}
\item $v_{\eps,n,m}$ solves the following equation:
\begin{equation}\label{HJB_eps,n}
\hspace{-5mm}\begin{cases}
\vspace{2mm}
\displaystyle\partial_t v_{\eps,n,m}(t,\mu) + \int_{\R^{dn}}\sum_{i=1}^n\sup_{a_i\in A}\bigg\{\langle\textcolor{black}{b_{n,m}^i(t,x_1,\ldots,x_n,a_i)},\partial_{x_i}\bar v_{\eps,n,m}(t,\bar x)\rangle \\
\displaystyle\vspace{2mm}+\,\frac{1}{2}\textup{tr}\Big[\big(\textcolor{black}{(\sigma\sigma\trans)(t,x_i,a_i)} + \eps^2\big)\partial_{x_ix_i}^2\bar v_{\eps,n,m}(t,\bar x)\Big] &\hspace{-2.5cm}(t,\mu)\in[0,T)\times \Pc_2(\R^d), \\
\displaystyle\vspace{2mm}+\,\frac{1}{n} f_{n,m}^i(t,x_1,\ldots,x_n,a_i)\bigg\}\mu(dx_1)\otimes\cdots\otimes\mu(dx_n) = 0, \\
\displaystyle v_{\eps,n,m}(T,\mu) = \frac{1}{n}\sum_{i=1}^n \int_{\R^{dn}} g_{n,m}^i(x_1,\ldots,x_n)\,\mu(dx_1)\otimes\cdots\otimes\mu(dx_n), &\hspace{-5mm}\,\mu\in\Pc_2(\R^d),
\end{cases}
\end{equation}
for every $n,m\in\N$, $\bar x=(x_1,\ldots,x_n)\in\R^{dn}$ and $x_1,\ldots,x_n\in\R^d$.
\end{enumerate}
\end{Theorem}
\begin{proof}[\textbf{Proof.}]
We split the proof into four steps.

\vspace{1mm}

\noindent\emph{Step I. Definition of $\bar v_{\eps,n,m}$ and its properties.} Fix $\eps>0$ and $n,m\in\N$. For every $t\in[0,T]$, $\bar x=(x_1,\ldots,x_n)\in\R^{dn}$, let $\bar v_{\eps,n,m}\colon[0,T]\times\R^{dn}\rightarrow\R$ be given by
\begin{equation}\label{bar_v_eps,n}
\bar v_{\eps,n,m}(t,x_1,\ldots,x_n) \ = \ \tilde v_{\eps,n,m}(t,\delta_{x_1}\otimes\cdots\otimes\delta_{x_n}),
\end{equation}
with $\tilde v_{\eps,n,m}$ defined by \eqref{tilde_v_eps,n}. In other words, $\bar v_{\eps,n,m}$ corresponds to the value function of the cooperative $n$-player game (see Section \ref{SubS:CooperativeGame}) with deterministic initial state $\bar x$ in place of the random vector $\bar\xi$. Hence
\begin{align*}
\bar v_{\eps,n,m}(t,x_1,\ldots,x_n) \ &= \ \sup_{\bar\alpha\in\bar\Ac^n}\frac{1}{n}\sum_{i=1}^n\bar\E\bigg[\int_t^T f_{n,m}^i\big(s,\textcolor{black}{\bar X_s^{1,\textcolor{black}{m},\eps,t,\bar x,\bar\alpha},\ldots,\bar X_s^{n,\textcolor{black}{m},\eps,t,\bar x,\bar\alpha}},\bar\alpha_s^i\big)\,ds \\
&\quad \ + g_{n,m}^i\big(\textcolor{black}{\bar X_T^{1,\textcolor{black}{m},\eps,t,\bar x,\bar\alpha},\ldots,\bar X_T^{n,\textcolor{black}{m},\eps,t,\bar x,\bar\alpha}}\big)\bigg].
\end{align*}
This optimal control problem involve coefficients satisfying Assumption \ref{AssA}. Therefore $\bar v_{\eps,n,m}$ is bounded, jointly continuous, and Lipschitz with respect to $\bar x$. Moreover, $\bar v_{\eps,n,m}$ is a viscosity solution of the following Bellman equation:
\begin{equation}\label{HJB_n,eps}
\begin{cases}
\displaystyle\partial_t \bar v_{\eps,n,m}(t,\bar x) + \sup_{(a_1,\ldots,a_n)\in A^n}\bigg\{\frac{1}{n}\sum_{i=1}^n f_{n,m}^i(t,\bar x,a_i) \\
\displaystyle+ \sum_{i=1}^n\frac{1}{2}\textup{tr}\Big[\big(\textcolor{black}{(\sigma\sigma\trans)(t,x_i,a_i)} + \eps^2\big)\partial_{x_ix_i}^2\bar v_{\eps,n,m}(t,\bar x)\Big] \\
\displaystyle+ \sum_{i=1}^n\langle \textcolor{black}{b_{n,m}^i(t,\bar x,a_i)},\partial_{x_i}\bar v_{\eps,n,m}(t,\bar x)\rangle\bigg\} = 0, \hspace{2.25cm} \forall\,(t,\bar x)\in[0,T)\times\R^{dn}, \\
\displaystyle\bar v_{\eps,n}(t,\bar x) = \frac{1}{n}\sum_{i=1}^n g_{n,m}^i(\bar x), \hspace{5.4cm} \forall\,\bar x\in\R^{dn}.
\end{cases}
\end{equation}
We notice that equation \eqref{HJB_n,eps} is uniformly parabolic, with coefficients satisfying Assumptions \ref{AssA} and \ref{AssB}. It follows that $\bar v_{\eps,n,m}\in C^{1,2}([0,T]\times\R^{dn})$ (see \cite[Theorem 14.15]{Lieberman} and the comments just below Theorem 14.15 regarding the case with linear operators ``$L_\nu$''). \textcolor{black}{In addition, from \cite[Theorem 4.7.4]{Krylov80} we deduce estimate \eqref{Estimate_2ndDeriv}. Concerning \eqref{Estimate_1stDeriv_n}, we just notice that it follows if we prove that
\[
\big|\bar v_{\eps,n,m}(t,\bar x) - \bar v_{\eps,n,m}(t,\bar z)\big| \ \leq \ \frac{C_K}{n} |\bar x - \bar z|,
\]
whenever the components of $\bar x=(x_1,\ldots,x_n)$ and $\bar z=(z_1,z_2,\ldots,z_n)$ are equal, apart for one component $x_k\neq z_k$.
Such a Lipschitz continuity} easily follows from the Lipschitz continuity estimates
for $b_{n,m}^i$, $f_{n,m}^i$ and $g_{n,m}^i$ proved in Lemma \ref{lm:conv}, formula \eqref{eq:Lipbfg}. 
\\
Finally, we observe that equation \eqref{HJB_n,eps} can be equivalently written as
\begin{align}\label{HJB_n,eps_2}
\partial_t \bar v_{\eps,n,m}(t,\bar x) + \sum_{i=1}^n\sup_{a_i\in A}\bigg\{\frac{1}{n}f_{n,m}^i(t,\bar x,a_i) + \langle \textcolor{black}{b_{n,m}^i(t,\bar x,a_i)},\partial_{x_i}\bar v_{\eps,n,m}(t,\bar x)\rangle& \\
+ \, \frac{1}{2}\textup{tr}\Big[\big(\textcolor{black}{(\sigma\sigma\trans)(t,x_i,a_i)} + \eps^2\big)\partial_{x_ix_i}^2\bar v_{\eps,n,m}(t,\bar x)\Big]&\bigg\} \ = \ 0, \notag
\end{align}
for all $(t,\bar x)\in[0,T)\times\R^{dn}$.

\vspace{1mm}

\noindent\emph{Step II. Proof of equality \eqref{v_eps,n=bar_v_eps,n}.} We prove the more general equality
\begin{equation}\label{tilde_v_eps,n=bar_v_eps,n}
\tilde v_{\eps,n,m}(t,\bar\mu) \ = \ \int_{\R^{dn}}\bar v_{\eps,n,m}(t,x_1,\ldots,x_n)\,\bar\mu(dx_1,\ldots,dx_n),
\end{equation}
for every $(t,\bar\mu)\in[0,T]\times\Pc_2(\R^{dn})$, from which \eqref{v_eps,n=bar_v_eps,n} follows. Notice that equality \eqref{tilde_v_eps,n=bar_v_eps,n} can be equivalently written as
\[
\tilde v_{\eps,n,m}(t,\bar\mu) \ = \ \E[\bar v_{\eps,n,m}(t,\bar\xi)],
\]
for every $t\in[0,T]$, $\bar\xi\in L^2(\bar\Omega,\bar\Fc_t,\bar\P;\R^{dn})$, with $\P_{\bar\xi}=\bar\mu$. We split the rest of the proof of Step II into two substeps.

\vspace{1mm}

\noindent\emph{Step II-a. General case: $\bar\xi\in L^2(\bar\Omega,\bar\Fc_t,\bar\P;\R^{dn})$.} Observe that we can apply Proposition \ref{P:ValueFunction} to the cooperative $n$-player game, from which we deduce that $\tilde v_{\eps,n,m}$ is bounded, jointly continuous, and Lipschitz with respect to $\bar\mu$. Moreover, recall from Step I above that $\bar v_{\eps,n,m}$ is also bounded, jointly continuous, and Lipschitz with respect to $\bar x$. As a consequence, the general case with $\bar\xi\in L^2(\bar\Omega,\bar\Fc_t,\bar\P;\R^{dn})$ can be deduced, relying on an approximation argument, from the case where $\bar\xi$ takes only a finite number of values, namely from the next Step II-b.

\vspace{1mm}

\noindent\emph{Step II-b. $\bar\xi\in L^2(\bar\Omega,\bar\Fc_t,\bar\P;\R^{dn})$ taking only a finite number of values.} Firstly, we fix some notation. For every $t\in[0,T]$, let $\bar\F^{B,W,t}=(\bar{\Fc}_s^{B,W,t})_{s\geq0}$ be the $\bar\P$-completion of the filtration generated by $(\bar B_{s\vee t}-\bar B_t)_{s\geq0}$ and $(\bar W_{s\vee t}-\bar W_t)_{s\geq0}$, where we recall that $\bar B=(\bar B^1,\ldots,\bar B^n)$ and $\bar W=(\bar W^1,\ldots,\bar W^n)$. Let also $Prog(\bar\F^{B,W,t})$ denote the $\sigma$-algebra of $[0,T]\times\bar\Omega$ of all $\bar\F^{B,W,t}$-progressive sets.

\vspace{1mm}

\noindent\emph{Proof of the inequality $\tilde v_{\eps,n,m}(t,\bar\mu)\leq\E[\bar v_{\eps,n,m}(t,\bar\xi)]$.} Suppose that $\bar\xi\in L^2(\bar\Omega,\bar\Fc_t,\bar\P;\R^{dn})$ takes only a finite number of values. In such a case, by \cite[Lemma B.3]{CKGPR20} there exists a $\bar\Fc_t$-measurable random variable $\bar U\colon\bar\Omega\rightarrow\R$, having uniform distribution on $[0,1]$ and being independent of $\bar\xi$. Then, by \cite[Lemma B.2]{CKGPR20}, for every $\bar\alpha\in\bar\Ac^n$ there exists a measurable function
\[
\mathrm a\colon\big([0,T]\times\bar\Omega\times\R^{dn}\times[0,1],Prog(\bar\F^{B,W,t})\otimes\Bc(\R^{dn})\otimes\mathcal B([0,1])\big) \longrightarrow (A^n,\Bc(A^n))
\]
such that $\bar\beta:=(\mathrm a_s(\bar\xi,\bar U))_{s\in[0,T]}\in\bar\Ac^n$ and
\begin{align*}
&\Big(\bar\xi,(\mathrm a_s(\bar\xi,\bar U))_{s\in[t,T]},(\bar B_s-\bar B_t)_{s\in[t,T]},(\bar W_s-\bar W_t)_{s\in[t,T]}\Big) \\
&\hspace{3cm}\overset{\mathscr L}{=} \ \Big(\bar\xi,(\bar\alpha_s)_{s\in[t,T]},(\bar B_s-\bar B_t)_{s\in[t,T]},(\bar W_s-\bar W_t)_{s\in[t,T]}\Big),
\end{align*}
where $\overset{\mathscr L}{=}$ stands for equality in law. As a consequence, proceeding along the same lines as in \cite[Proposition 1.137]{FabbriGozziSwiech}, we deduce that
\[
\big(\bar X_s^{\textcolor{black}{m},\eps,t,\bar\xi,\bar\alpha},\bar\alpha_s\big)_{s\in[t,T]} \ \overset{\mathscr L}{=} \ \big(\bar X_s^{\textcolor{black}{m},\eps,t,\bar\xi,\bar\beta},\bar\beta_s\big)_{s\in[t,T]}.
\]
Moreover, since $\bar\xi$ takes only a finite number of values, it holds that
\begin{equation}\label{xi_discrete}
\bar\xi \ = \ \sum_{k=1}^K \bar x_k\,1_{\bar E_k},
\end{equation}
for some $K\in\N$, $\bar x_k\in\R^{dn}$, $\bar E_k\in\sigma(\bar\xi)$, with $\{\bar E_k\}_{k=1,\ldots,K}$ being a partition of $\bar\Omega$. Let also
\[
\bar\beta_{k,s} \ := \ \mathrm a_s(\bar x_k,\bar U), \qquad \forall\,s\in[0,T],\,k=1,\ldots,K.
\]
It easy to see that $\bar X^{\textcolor{black}{m},\eps,t,\bar\xi,\bar\beta}$ and $\bar X^{\textcolor{black}{m},\eps,t,\bar x_1,\bar\beta_1}\,1_{\bar E_1}+\cdots+\bar X^{\textcolor{black}{m},\eps,t,\bar x_K,\bar\beta_K}\,1_{\bar E_K}$ satisfy the same system of controlled stochastic differential equations, therefore, by pathwise uniqueness, they are $\bar\P$-indistinguishable. Hence
\begin{align*}
&\frac{1}{n}\sum_{i=1}^n\bar\E\bigg[\int_t^T f_{n,m}^i\big(s,\bar X_s^{1,\textcolor{black}{m},\eps,t,\bar\xi,\bar\alpha},\ldots,\bar X_s^{n,\textcolor{black}{m},\eps,t,\bar\xi,\bar\alpha},\bar\alpha_s^i\big)\,ds + g_{n,m}^i\big(\bar X_T^{1,\textcolor{black}{m},\eps,t,\bar\xi,\bar\alpha},\ldots,\bar X_T^{n,\textcolor{black}{m},\eps,t,\bar\xi,\bar\alpha}\big)\bigg] \\
&= \ \frac{1}{n}\sum_{i=1}^n\bar\E\bigg[\int_t^T f_{n,m}^i\big(s,\bar X_s^{1,\textcolor{black}{m},\textcolor{black}{m},\eps,t,\bar\xi,\bar\beta},\ldots,\bar X_s^{n,\textcolor{black}{m},\eps,t,\bar\xi,\bar\beta},\bar\beta_s^i\big)\,ds + g_{n,m}^i\big(\bar X_T^{1,\textcolor{black}{m},\eps,t,\bar\xi,\bar\beta},\ldots,\bar X_T^{n,\textcolor{black}{m},\eps,t,\bar\xi,\bar\beta}\big)\bigg] \\
&= \ \frac{1}{n}\sum_{i=1}^n\bar\E\bigg[\sum_{k=1}^K\bigg(\int_t^T f_{n,m}^i\big(s,\bar X_s^{1,\textcolor{black}{m},\eps,t,\bar x_k,\bar\beta_k},\ldots,\bar X_s^{n,\textcolor{black}{m},\eps,t,\bar x_k,\bar\beta_k},\bar\beta_{k,s}^i\big)ds \\
&\quad \ + g_{n,m}^i\big(\bar X_T^{1,\textcolor{black}{m},\eps,t,\bar x_k,\bar\beta_k},\ldots,\bar X_T^{n,\textcolor{black}{m},\eps,t,\bar x_k,\bar\beta_k}\big)\bigg)1_{\bar E_k}\bigg].
\end{align*}
Since both $\{\bar X^{\textcolor{black}{m},\eps,\bar x_k,\bar\beta_k}\}_{k}$ and $\{\bar\beta_k\}_k$ are independent of $\{\bar E_k\}_k$, we have
\begin{align*}
&\frac{1}{n}\sum_{i=1}^n\bar\E\bigg[\sum_{k=1}^K\bigg(\int_t^T f_{n,m}^i\big(s,\bar X_s^{1,\textcolor{black}{m},\eps,t,\bar x_k,\bar\beta_k},\ldots,\bar X_s^{n,\textcolor{black}{m},\eps,t,\bar x_k,\bar\beta_k},\bar\beta_{k,s}^i\big)\,ds \\
&+ g_{n,m}^i\big(\bar X_T^{1,\textcolor{black}{m},\eps,t,\bar x_k,\bar\beta_k},\ldots,\bar X_T^{n,\textcolor{black}{m},\eps,t,\bar x_k,\bar\beta_k}\big)1_{\bar E_k}\bigg] \\
&= \frac{1}{n}\sum_{i=1}^n\bar\E\bigg[\sum_{k=1}^K\bar\E\bigg[\int_t^T f_{n,m}^i\big(s,\bar X_s^{1,\textcolor{black}{m},\eps,t,\bar x_k,\bar\beta_k},\ldots,\bar X_s^{n,\textcolor{black}{m},\eps,t,\bar x_k,\bar\beta_k},\bar\beta_{k,s}^i\big)ds \\
&+ g_{n,m}^i\big(\bar X_T^{1,\textcolor{black}{m},\eps,t,\bar x_k,\bar\beta_k},\ldots,\bar X_T^{n,\textcolor{black}{m},\eps,t,\bar x_k,\bar\beta_k}\big)\bigg]1_{\bar E_k}\bigg] \\
&= \sum_{k=1}^K\bar\E\bigg[\frac{1}{n}\sum_{i=1}^n\bar\E\bigg[\int_t^T f_{n,m}^i\big(s,\bar X_s^{1,\textcolor{black}{m},\eps,t,\bar x_k,\bar\beta_k},\ldots,\bar X_s^{n,\textcolor{black}{m},\eps,t,\bar x_k,\bar\beta_k},\bar\beta_{k,s}^i\big)ds \\
&+ g_{n,m}^i\big(\bar X_T^{1,\textcolor{black}{m},\eps,t,\bar x_k,\bar\beta_k},\ldots,\bar X_T^{n,\textcolor{black}{m},\eps,t,\bar x_k,\bar\beta_k}\big)\bigg]1_{\bar E_k}\bigg] \ \leq \ \sum_{k=1}^K\bar\E\Big[\bar v_{\eps,n,m}(t,\bar x_k)\,1_{\bar E_k}\Big] \ = \ \bar\E\big[\bar v_{\eps,n,m}(t,\bar\xi)\big].
\end{align*}
As $\bar\alpha$ was arbitrary, we obtain (denoting by $\bar\mu$ the law of $\bar\xi$)
\begin{align*}
\tilde v_{\eps,n,m}(t,\bar\mu) \ &= \ \sup_{\bar\alpha\in\bar\Ac^n}\frac{1}{n}\sum_{i=1}^n\bar\E\bigg[\int_t^T f_{n,m}^i\big(s,\bar X_s^{1,\textcolor{black}{m},\eps,t,\bar\xi,\bar\alpha},\ldots,\bar X_s^{n,\textcolor{black}{m},\eps,t,\bar\xi,\bar\alpha},\bar\alpha_s^i\big)\,ds \\
&\quad \ + g_{n,m}^i\big(\bar X_T^{1,\textcolor{black}{m},\eps,t,\bar\xi,\bar\alpha},\ldots,\bar X_T^{n,\textcolor{black}{m},\eps,t,\bar\xi,\bar\alpha}\big)\bigg] \ \leq \ \bar\E\big[\bar v_{\eps,n,m}(t,\bar\xi)\big].
\end{align*}

\vspace{1mm}

\noindent\emph{Proof of the inequality $\E[\bar v_{\eps,n,m}(t,\bar\xi)]\leq\tilde v_{\eps,n,m}(t,\bar\mu)$.} Let $\bar\Ac_t^n$ be the subset of $\bar\Ac^n$ of all $\bar\F^{B,W,t}$-progressively measurable processes $\bar\alpha=(\bar\alpha^1,\ldots,\bar\alpha^n)\colon[0,T]\times\bar\Omega\rightarrow A^n$. Then, it is well-known that  the value function $\bar v_{\eps,n,m}$ in \eqref{bar_v_eps,n} is also given by
\begin{align}\label{bar_v_eps,n_2}
\bar v_{\eps,n,m}(t,x_1,\ldots,x_n) \ &= \ \sup_{\bar\alpha\in\bar\Ac_t^n}\frac{1}{n}\sum_{i=1}^n\bar\E\bigg[\int_t^T \! f_{n,m}^i\big(s,\bar X_s^{1,\textcolor{black}{m},\eps,t,\bar x,\bar\alpha},\ldots,\bar X_s^{n,\textcolor{black}{m},\eps,t,\bar x,\bar\alpha},\bar\alpha_s^i\big)ds \\
&\quad \ + g_{n,m}^i\big(\bar X_T^{1,\textcolor{black}{m},\eps,t,\bar x,\bar\alpha},\ldots,\bar X_T^{n,\textcolor{black}{m},\eps,t,\bar x,\bar\alpha}\big)\bigg], \notag
\end{align}
where the supremum is taken on $\bar\Ac_t^n$ rather than on $\bar\Ac^n$. Now, let $\bar\xi\in L^2(\bar\Omega,\bar\Fc_t,\bar\P;\R^{dn})$ be given by \eqref{xi_discrete}. By \eqref{bar_v_eps,n_2}, for every $\delta>0$ and $k=1,\ldots,K$, there exists $\bar\beta_k\in\bar\Ac_t^n$ (possibly depending on $\delta$) such that
\begin{align*}
\bar v_{\eps,n,m}(t,\bar x_k)\ &\leq \ \frac{1}{n}\sum_{i=1}^n\bar\E\bigg[\int_t^T f_{n,m}^i\big(s,\bar X_s^{1,\textcolor{black}{m},\eps,t,\bar x_k,\bar\beta_k},\ldots,\bar X_s^{n,\textcolor{black}{m},\eps,t,\bar x_k,\bar\beta_k},\bar\beta_{k,s}^i\big)\,ds \\
&\quad \ + \; g_{n,m}^i\big(\bar X_T^{1,\textcolor{black}{m},\eps,t,\bar x_k,\bar\beta_k},\ldots,\bar X_T^{n,\textcolor{black}{m},\eps,t,\bar x_k,\bar\beta_k}\big)\bigg] + \delta.
\end{align*}
Then, define
\[
\bar\beta \ := \ \sum_{k=1}^K \bar\beta_k\,1_{\bar E_k}.
\]
Notice that $\bar\beta\in\bar\Ac^n$. Moreover, it is easy to see that $\bar X^{\textcolor{black}{m},\eps,t,\bar\xi,\bar\beta}$ and $\bar X^{\textcolor{black}{m},\eps,t,\bar x_1,\bar\beta_1}\,1_{\bar E_1}+\cdots+\bar X^{\textcolor{black}{m},\eps,t,\bar x_K,\bar\beta_K}\,1_{\bar E_K}$ satisfy the same system of controlled stochastic differential equations, therefore, by pathwise uniqueness, they are $\bar\P$-indistinguishable. Hence (using the independence of both $\{\bar X^{\textcolor{black}{m},\eps,\bar x_k,\bar\beta_k}\}_{k}$ and $\{\bar\beta_k\}_k$ from $\{\bar E_k\}_k$)
\begin{align*}
&\bar\E\big[\bar v_{\eps,n,m}(t,\bar\xi)\big] \ = \ \sum_{k=1}^K \bar\E\Big[\bar v_{\eps,n,m}(t,\bar x_k)\,1_{\bar E_k}\Big] \\
&\leq \ \sum_{k=1}^K\bar\E\bigg[\frac{1}{n}\sum_{i=1}^n\bar\E\bigg[\int_t^T f_{n,m}^i\big(s,\bar X_s^{1,\textcolor{black}{m},\eps,t,\bar x_k,\bar\beta_k},\ldots,\bar X_s^{n,\textcolor{black}{m},\eps,t,\bar x_k,\bar\beta_k},\bar\beta_{k,s}^i\big)\,ds \\
&\quad \ + g_{n,m}^i\big(\bar X_T^{1,\textcolor{black}{m},\eps,t,\bar x_k,\bar\beta_k},\ldots,\bar X_T^{n,\textcolor{black}{m},\eps,t,\bar x_k,\bar\beta_k}\big)\bigg]1_{\bar E_k}\bigg] + \delta \\
&= \ \frac{1}{n}\sum_{i=1}^n\bar\E\bigg[\sum_{k=1}^K\bar\E\bigg[\int_t^T f_{n,m}^i\big(s,\bar X_s^{1,\textcolor{black}{m},\eps,t,\bar x_k,\bar\beta_k},\ldots,\bar X_s^{n,\textcolor{black}{m},\eps,t,\bar x_k,\bar\beta_k},\bar\beta_{k,s}^i\big)\,ds \\
&\quad \ + g_{n,m}^i\big(\bar X_T^{1,\textcolor{black}{m},\eps,t,\bar x_k,\bar\beta_k},\ldots,\bar X_T^{n,\textcolor{black}{m},\eps,t,\bar x_k,\bar\beta_k}\big)\bigg]1_{\bar E_k}\bigg] + \delta \\
&= \ \frac{1}{n}\sum_{i=1}^n\bar\E\bigg[\sum_{k=1}^K\bigg(\int_t^T f_{n,m}^i\big(s,\bar X_s^{1,\textcolor{black}{m},\eps,t,\bar x_k,\bar\beta_k},\ldots,\bar X_s^{n,\textcolor{black}{m},\eps,t,\bar x_k,\bar\beta_k},\bar\beta_{k,s}^i\big)\,ds \\
&\quad \ + g_{n,m}^i\big(\bar X_T^{1,\textcolor{black}{m},\eps,t,\bar x_k,\bar\beta_k},\ldots,\bar X_T^{n,\textcolor{black}{m},\eps,t,\bar x_k,\bar\beta_k}\big)\bigg)1_{\bar E_k}\bigg] + \delta \\
&= \ \frac{1}{n}\sum_{i=1}^n\bar\E\bigg[\int_t^T f_{n,m}^i\big(s,\bar X_s^{1,\textcolor{black}{m},\eps,t,\bar\xi,\bar\beta},\ldots,\bar X_s^{n,\textcolor{black}{m},\eps,t,\bar\xi,\bar\beta},\bar\beta_s^i\big)\,ds \\
&\quad \ + g_{n,m}^i\big(\bar X_T^{1,\textcolor{black}{m},\eps,t,\bar\xi,\bar\beta},\ldots,\bar X_T^{n,\textcolor{black}{m},\eps,t,\bar\xi,\bar\beta}\big)\bigg] + \delta \ \leq \ \tilde v_{\eps,n,m}(t,\bar\mu) + \delta,
\end{align*}
with $\bar\mu$ being the law of $\bar\xi$. From the arbitrariness of $\delta$, we conclude that the inequality $\E[\bar v_{\eps,n,m}(t,\bar\xi)]\leq\tilde v_{\eps,n,m}(t,\bar\mu)$ holds.

\vspace{1mm}

\noindent\emph{Step III. Proof of item 1).} We begin noting that, by equality \eqref{v_eps,n=bar_v_eps,n}, we have
\begin{equation}\label{TimeDeriv}
\partial_t v_{\eps,n,m}(t,\mu) \ = \ \int_{\R^{dn}} \partial_t \bar v_{\eps,n,m}(t,x_1,\ldots,x_n)\,\mu(dx_1)\cdots\mu(dx_n),
\end{equation}
which proves that $\partial_t v_{\eps,n,m}$ exists and is continuous. Now, for every $\eps>0$ and $n\geq2$, let $\hat v_{\eps,n,m}\colon[0,T]\times(\Pc_2(\R^d))^n\rightarrow\R$ be given by
\[
\hat v_{\eps,n,m}(t,\mu_1,\ldots,\mu_n) \ := \ \tilde v_{\eps,n,m}(t,\mu_1\otimes\cdots\otimes\mu_n) \ = \ \int_{\R^{dn}} \bar v_{\eps,n,m}(t,x_1,\ldots,x_n)\,\mu_1(dx_1)\cdots\mu_n(dx_n),
\]
for every $(t,\mu_1,\ldots,\mu_n)\in[0,T]\times(\Pc_2(\R^d))^n$. Then, by direct calculation, we obtain
\begin{align*}
&\partial_{\mu_i} \hat v_{\eps,n,m}(t,\mu_1,\ldots,\mu_n)(x) \\
&= \int_{\R^{d(n-1)}} \!\!\! \partial_{x_i} \bar v_{\eps,n,m}(t,x_1,\ldots,x_{i-1},x,x_{i+1}\ldots,x_n)\mu_1(dx_1)\cdots\mu_{i-1}(dx_{i-1})\mu_{i+1}(dx_{i+1})\cdots\mu_n(dx_n),
\end{align*}
for every $(t,\mu_1,\ldots,\mu_n,x)\in[0,T]\times(\Pc_2(\R^d))^n\times\R^d$, $i=1,\ldots,n$. Since $v_{\eps,n,m}(t,\mu)=\hat v_{\eps,n,m}(t,\mu,\ldots,\mu)$, we obtain
\begin{align}\label{partial_mu_identity}
&\partial_\mu v_{\eps,n,m}(t,\mu)(x) \\
&= \ \sum_{i=1}^n \int_{\R^{d(n-1)}} \partial_{x_i} \bar v_{\eps,n,m}(t,x_1,\ldots,x_{i-1},x,x_{i+1},\ldots,x_n)\,\mu(dx_1)\cdots\mu(dx_{i-1})\,\mu(dx_{i+1})\cdots\mu(dx_n), \notag	
\end{align}
for every $(t,\mu,x)\in[0,T]\times\Pc_2(\R^d)\times\R^d$. Hence
\begin{align}\label{partial_x_partial_mu_identity}
&\partial_x\partial_\mu v_{\eps,n,m}(t,\mu)(x) \\
&= \ \sum_{i=1}^n \int_{\R^{d(n-1)}} \partial_{x_ix_i}^2 \bar v_{\eps,n,m}(t,x_1,\ldots,x_{i-1},x,x_{i+1},\ldots,x_n)\,\mu(dx_1)\cdots\mu(dx_{i-1})\,\mu(dx_{i+1})\cdots\mu(dx_n), \notag
\end{align}
for every $(t,\mu,x)\in[0,T]\times\Pc_2(\R^d)\times\R^d$. In conclusion, we see that $v_{\eps,n,m}\in C^{1,2}([0,T]\times\Pc_2(\R^d))$.

\vspace{1mm}

\noindent\emph{Step IV. Proof of item 3).} Recall that $\bar v_{\eps,n,m}$ solves equation \eqref{HJB_n,eps}. Fix $(t,\mu)\in[0,T]\times\Pc_2(\R^d)$. When $t=T$, integrating the terminal condition of \eqref{HJB_n,eps} with respect to $\mu\otimes\cdots\otimes\mu$ on $\R^{dn}$, we get
\[
v_{\eps,n,m}(T,\mu) \ = \ \frac{1}{n}\sum_{i=1}^n \int_{\R^{dn}} g_{n,m}^i(x_1,\ldots,x_n)\,\mu(dx_1)\otimes\cdots\otimes\mu(dx_n),
\]
which corresponds to the terminal condition of equation \eqref{HJB_eps,n}. On the other hand, when $t<T$, integrating equation \eqref{HJB_n,eps_2} with respect to $\mu\otimes\cdots\otimes\mu$ on $\R^{dn}$, and using \eqref{TimeDeriv}, we find
\begin{align*}
\partial_t v_{\eps,n,m}(t,\mu) + \int_{\R^{dn}}\sum_{i=1}^n\sup_{a_i\in A}\bigg\{\frac{1}{n} f_{n,m}^i(t,\bar x,a_i) + \langle \textcolor{black}{b_{n,m}^i(t,\bar x,a_i)},\partial_{x_i}\bar v_{\eps,n,m}(t,\bar x)\rangle& \\
+ \, \frac{1}{2}\textup{tr}\Big[\big(\textcolor{black}{(\sigma\sigma\trans)(t,x_i,a_i)} + \eps^2\big)\partial_{x_ix_i}^2\bar v_{\eps,n,m}(t,\bar x)\Big]\bigg\}\mu(dx_1)\otimes\cdots\otimes\mu(dx_n)& \ = \ 0,
\end{align*}
which corresponds to equation \eqref{HJB_eps,n}.
\end{proof}

\noindent We end this section with the next result, which is used in the proof of the comparison theorem, in order to prove that $v_0\leq u_2$.
We first need to regularize the coefficients also in the control variable. For that, we fix $p\in\N$ such that $A\subset \R^p$ and a function $\zeta_p\colon\R^p\rightarrow[0,+\infty)$ being of class $C^\infty$ with compact support and satisfying $\int_{\R^p}\zeta_p(a)\,da=1$. Moreover, we extend the continuous and bounded functions $b$ and $f$ defined on $[0,T]\times\R^d\times\Pc_2(\R^d)\times A$ to some continuous and bounded functions, still denoted by $b$ and $f$, defined on $[0,T]\times\R^d\times\Pc_2(\R^d)\times\R^p$. Then, as in Section \ref{SubS:CooperativeGame} we define the coefficients $b_{n,m}^i$ and $f_{n,m}^i$ on the entire space $[0,T]\times\R^{dn}\times\R^p$ (rather than $[0,T]\times\R^{dn}\times A$). Afterwards, we define the coefficients $\tilde b ^i_{n,m}$ and $\tilde f ^i_{n,m}$ by
\begin{align*}
\tilde b_{n,m}^i(t,\bar x,a)  \ &= \ m^{p}
\int_{\R^{p}}b_{n,m}^i(t,\bar x,a-a')\zeta_p(ma')da', \\
\tilde f_{n,m}^i(t,\bar x,a)  \ &= \ m^{p}
\int_{\R^{p}}f_{n,m}^i(t,\bar x,a-a')\zeta_p(ma')da',
\end{align*}
for all $n,m\in\N$, $i=1,\ldots,n$, $\bar x=(x_1,\ldots,x_n)\in\R^{dn}$, $(t,a)\in[0,T]\times\R^p$. We can now state our last result. 

\begin{Theorem}\label{T:SmoothApprox2}
Let Assumptions \ref{AssA}, \ref{AssB}, \ref{AssC}, \ref{AssD} hold. For every $t\in[0,T]$, let $\Mc_t$ denote the set of $\Fc_t$-measurable random variables $\mathfrak a\colon\Omega\rightarrow A$. Let $u_2\colon[0,T]\times\Pc_2(\R^d)\rightarrow\R$ be a continuous and bounded function. For every $\underline t_0\in[0,T)$, $s_0\in(\underline t_0,T]$, let $v^{s_0}\colon[\underline t_0,s_0]\times\Pc_2(\R^d\times A)\rightarrow\R$ be given by
\[
v^{s_0}(t,\nu) \ = \E\bigg[\int_t^{s_0} f\big(r,X_r^{t,\xi,\mathfrak a_0},\P_{X_r^{t,\xi,\mathfrak a_0}},Y_r^{t,\mathfrak a_0}\big)\,dr\bigg] + u_2\big(s_0,\P_{X_{s_0}^{t,\xi,\mathfrak a_0}}\big),
\]
for all $(t,\nu)\in[\underline t_0,s_0]\times\Pc_2(\R^d\times A)$, $\xi\in L^2(\Omega,\Fc_t,\P;\R^d)$ and $\mathfrak a_0\in\Mc_t$ such that $\P_{(\xi,\mathfrak a_0)}=\nu$, \textcolor{black}{where $(X^{t,\xi,\mathfrak a_0},Y^{t,\mathfrak a_0})$ is the unique solution to the following system of McKean-Vlasov stochastic differential equations:
\begin{equation}\label{MKVSDE_X,Y}
\begin{cases}
X_s \ = \ \xi + \int_t^s b\big(r,X_r,\P_{X_r},Y_r\big)\,dr
+ \int_t^s \sigma(r,X_r,Y_r)\,dB_r, \qquad s\in[t,T], \\
Y_s = \mathfrak a_0, \qquad s\in[t,T].
\end{cases}
\end{equation}}
Moreover, for every $n,m\in\N$, let $v_{n,m}^{s_0}\colon[\underline t_0,s_0]\times\Pc_2(\R^d\times A)\rightarrow\R$ be given by
\begin{align}\label{v_n,m^s,a}
v_{n,m}^{s_0}(t,\nu) \ &= \ \frac{1}{n}\sum_{i=1}^n\E\bigg[\int_t^{s_0} \tilde  f_{n,m}^i\big(r,\bar {\tilde  X}_r^{1,m,t,\bar\xi,\bar{\mathfrak a}_0},\ldots,\bar {\tilde  X}_r^{n,m,t,\bar\xi,\bar{\mathfrak a}_0},\tilde Y_r^{i,t,\bar{\mathfrak a}_0}\big)\,dr \\
&\quad \ + u_{n,m}\big(s_0,\bar {\tilde  X}_{s_0}^{1,m,t,\bar\xi,\bar{\mathfrak a}_0},\ldots,\bar {\tilde  X}_{s_0}^{n,m,t,\bar\xi,\bar{\mathfrak a}_0}\big)\bigg], \notag
\end{align}
for every $(t,\nu)\in[\underline t_0,s_0]\times\Pc_2(\R^d\times A)$, $\bar\xi=(\xi_1,\ldots,\xi_n)\in L^2(\Omega,\Fc_t,\P;\R^d)$ and $\bar{\mathfrak a}_0=(\mathfrak a_0^1,\ldots,\mathfrak a_0^n)$, with $\mathfrak a_0^i\in\Mc_t$, such that $\P_{(\bar\xi,\bar{\mathfrak a}_0)}=\nu\otimes\cdots\otimes\nu$. Moreover, $\tilde Y_r^{i,t,\bar{\mathfrak a}_0}=\mathfrak a_0^i$ for $r\in[t,T]$ and $\bar {\tilde  X}^{i,m,t,\bar\xi,\bar{\mathfrak a}_0}$ solves equation \eqref{SDE_bar_X} with $\eps=0$, $\bar\alpha_r^i=\tilde Y_r^{i,t,\bar{\mathfrak a}_0}$ for $r\in[t,T]$, $b$ replaced by $\tilde b_{n,m}^i$. Similarly, $u_{n,m}(s_0,\cdot)\colon\R^{dn}\rightarrow\R$ is given by
\[
u_{n,m}(s_0,\bar x) \ = \ m^{nd}\int_{\R^{dn}}u_2\bigg(s_0,{1\over n}\sum_{j=1}^n \delta_{x_j-y_j}\bigg)\prod_{j=1}^{n}\Phi(my_j)dy_j,
\]
for all $\bar x=(x_1,\ldots,x_n)\in\R^{dn}$, with $\Phi$ as in Section \ref{SubS:CooperativeGame}.\\
Then, for every $n,m\in\N$, there exists $\bar v_{n,m}^{s_0}\colon[\underline t_0,s_0]\times(\R^d\times A)^n\rightarrow\R$ such that
\[
v_{n,m}^{s_0}(t,\nu) \ = \ \int_{\R^{dn}}\bar v_{n,m}^{s_0}(t,x_1,\ldots,x_n,a_1,\ldots,a_n)\,\nu(dx_1,da_1)\cdots\nu(dx_n,da_n),
\]
for every $(t,\nu)\in[\underline t_0,s_0]\times\Pc_2(\R^d\times A)$, and the following holds.
\begin{enumerate}[\upshape 1)]
\item $\bar v_{n,m}^{s_0}\in C^{1,2}([\underline t_0,s_0]\times(\R^d\times A)^n)$ and $v_{n,m}^{s_0}\in C^{1,2}([\underline t_0,s_0]\times\Pc_2(\R^d\times A))$.
\item For all $(t,\bar x,\bar a)\in[\underline t_0,s_0]\times(\R^d\times A)^n$, with $\bar x=(x_1,\ldots,x_n)$, $\bar a=(a_1,\ldots,a_n)$ and $x_1,\ldots,x_n\in\R^d$, $a_1,\ldots,a_n\in A$, it holds that
\[
\big|\partial_{x_i}\bar v_{n,m}^{s_0}(t,\bar x,\bar a)\big| \ \leq \ \frac{C_K}{n},
\]
for every $i=1,\ldots,n$, for some constant $C_K\geq0$, possibly depending on $K$, but independent of $n,m$, where $K$ is as in Assumption \ref{AssA}.
\item $v_{n,m}^{s_0}$ solves the following equation:
\[
\hspace{-5mm}\begin{cases}
\vspace{2mm}
\displaystyle\partial_t v_{n,m}^{s_0}(t,\nu) + \bar\E\sum_{i=1}^n\bigg\{\frac{1}{n} \tilde  f_{n,m}^i(t,\xi_1,\ldots,\xi_n,\mathfrak a_0^i) + \langle \tilde b_{n,m}^i(t,\xi_1,\ldots,\xi_n,\mathfrak a_0^i),\partial_{x_i}\bar v_{n,m}^{s_0}(t,\bar\xi,\bar{\mathfrak a}_0)\rangle \\
\displaystyle\vspace{2mm}+\,\frac{1}{2}\textup{tr}\Big[(\sigma\sigma\trans)(t,\xi_i,\mathfrak a_0^i)\partial_{x_ix_i}^2\bar v_{n,m}^{s_0}(t,\bar\xi,\bar{\mathfrak a}_0)\Big]\bigg\} = 0, &\hspace{-5.7cm}(t,\nu)\in[\underline t_0,s_0)\times \Pc_2(\R^d\times A), \\
\displaystyle v_{n,m}^{s_0}(s_0,\nu) = \bar\E\big[u_{n,m}(s_0,\bar\xi)\big], &\hspace{-5.05cm}\nu\in\Pc_2(\R^d\times A),
\end{cases}
\]
for any $\bar\xi=(\xi_1,\ldots,\xi_n)\in L^2(\Omega,\Fc_t,\P;\R^{dn})$ and $\bar{\mathfrak a}_0=(\mathfrak a_0^1,\ldots,\mathfrak a_0^n)$, with $\mathfrak a_0^i\in\Mc_t$, such that $\bar\P_{(\bar\xi,\bar{\mathfrak a}_0)}=\nu\otimes\cdots\otimes\nu$.
\item If there exists $q>2$ such that $\nu\in\Pc_q(\R^d)$, then
\[
\lim_{n\rightarrow+\infty}\lim_{m\rightarrow+\infty} v_{n,m}^{s_0}(t,\nu) \ = \ v^{s_0}(t,\nu).
\]
\end{enumerate}
\end{Theorem}
\begin{proof}[\textbf{Proof.}]
Items 1)-2)-3) follow from the same arguments as in Theorem \ref{T:SmoothApprox}, taking into account that here we are in a ``linear'' context, while Theorem \ref{T:SmoothApprox} deals with the ``fully non-linear'' case. Since we are in the linear case, the regularity results hold even if $\eps=0$ (that's why here we do not need this extra parameter), as it can be deduced for instance from \cite[Theorem 6.1, Chapter 5]{friedman75vol1}. Finally, item 4) follows from the propagation of chaos result \cite[Theorem 2.12]{Lacker} proceeding as in the proof of Theorem \ref{T:PropagChaos} \textcolor{black}{and noting that, in the present context, Assumption (B) in \cite{Lacker} can be neglected (that is Lipschitz continuity of the coefficients $b$ and $\sigma$ with respect to the extra state variable $a$). As a matter of fact, Assumption (B) in \cite{Lacker} is imposed to have uniqueness of the underlying McKean-Vlasov stochastic differential equations, which in our case correspond to system \eqref{MKVSDE_X,Y} and uniqueness clearly holds under our assumptions, without imposing in addition that $b$ and $\sigma$ are Lipschitz continuous with respect to $a$.}
\end{proof}

\paragraph{Acknowledgments.} The authors are very grateful to Pierre Cardaliaguet, who found a gap in the first version of the paper.

\small

\bibliographystyle{plain}
\bibliography{references}

\end{document}